\documentclass[11pt, a4paper,leqno]{amsart}
\usepackage{amsmath,amsthm,amscd,amssymb,amsfonts,amsbsy,stmaryrd}
\usepackage{mathrsfs}
\usepackage{latexsym}
\usepackage{txfonts}
\usepackage{exscale}
\usepackage{comment}
\usepackage[shortlabels]{enumitem}
\usepackage{bbm}
\usepackage{esint}
\usepackage{graphicx}

\usepackage[colorlinks,citecolor=red,pagebackref,hypertexnames=false]{hyperref}

\usepackage{color}

\calclayout
\allowdisplaybreaks

\theoremstyle{plain}
\newtheorem{theorem}[equation]{Theorem}
\newtheorem{lemma}[equation]{Lemma}
\newtheorem{corollary}[equation]{Corollary}
\newtheorem{proposition}[equation]{Proposition}

\theoremstyle{definition}
\newtheorem{definition}[equation]{Definition}

\theoremstyle{remark}
\newtheorem{remark}[equation]{Remark}

\numberwithin{equation}{section}

\newcommand{\RR}{{\mathbb{R}}}

\newcommand{\eps}{\varepsilon}

\newcommand{\dist}{\operatorname{dist}}

\newcommand{\dv}{\operatorname{div}}

\newcommand{\re}{\mathbb{R}}

\newcommand{\ree}{\mathbb{R}^{n+1}}

\newcommand{\dd}{\mathbb{D}}
\newcommand{\Z}{\mathbb{Z}}

\newcommand{\C}{\mathcal{C}}
\newcommand{\Ac}{\mathcal{A}}

\newcommand{\om}{\Omega}

\newcommand{\F}{\mathcal{F}}

\renewcommand{\H}{\mathcal{H}}

\newcommand{\W}{\mathcal{W}}

\newcommand{\Lc}{\mathcal{L}}

\newcommand{\Sc}{\mathcal{S}}
\newcommand{\Rc}{\mathcal{R}}
\newcommand{\Yc}{\mathcal{Y}}

\newcommand{\sbf}{{\bf S}}
\newcommand{\oo}{\mathcal{O}}

\newcommand{\pom}{\partial\Omega}

\newcommand{\hm}{\omega}

\newcommand{\bb}[1]{\mathbb{#1}}

\usepackage{textcomp}

\renewcommand{\emptyset}{\mbox{\textup{\O}}}

\DeclareMathOperator{\diam}{diam}

\DeclareMathOperator{\interior}{int}

\DeclareMathOperator{\BMO}{BMO}
\DeclareMathOperator{\BV}{BV}

\DeclareMathOperator{\smallc}{c}
\DeclareMathOperator{\bigR}{R}

\DeclareMathOperator{\bigSB}{SB}
\DeclareMathOperator{\bigY}{Y}
\DeclareMathOperator{\bigO}{O}
\DeclareMathOperator{\loc}{loc}
\DeclareMathOperator{\nt}{n.t.}

\def\Xint#1{\mathchoice
   {\XXint\displaystyle\textstyle{#1}}
   {\XXint\textstyle\scriptstyle{#1}}
   {\XXint\scriptstyle\scriptscriptstyle{#1}}
   {\XXint\scriptscriptstyle\scriptscriptstyle{#1}}
   \!\iint}
\def\XXint#1#2#3{{\setbox0=\hbox{$#1{#2#3}{\iint}$}
     \vcenter{\hbox{$#2#3$}}\kern-.5\wd0}}
     
\def\longminus{\raisebox{-2ex}{\rotatebox[origin=c]{20}{${-}\mkern-2mu{-}$}}}
\def\dashiint{\Xint\longminus}

\def\div{\mathop{\operatorname{div}}\nolimits}

\newcommand{\ra}{\rightarrow}
  
\newcommand{\Cfk}{\mathfrak{C}}

\title[The $A_{\infty}$ condition and $\eps$-approximators in uniform domains]{The $A_\infty$ condition, $\eps$-approximators, and Varopoulos extensions in uniform domains} 
\author{S. Bortz, B. Poggi, O. Tapiola, and X. Tolsa}
\address{Simon Bortz, Department of Mathematics
\\
University of Alabama
\\
Tuscaloosa, AL, 35487, USA}
\email{sbortz@ua.edu}
\address{Bruno Poggi, Departament de Matemàtiques
\\
Universitat Autònoma de Barcelona
\\
Bar\-ce\-lo\-na, Catalonia.}
\email{bgpoggi.math@gmail.com}
\address{Olli Tapiola, Departament de Matemàtiques
\\
Universitat Autònoma de Barcelona
\\
Bar\-ce\-lo\-na, Catalonia.}
\email{olli.m.tapiola@gmail.com}
\address{Xavier Tolsa, ICREA
\\
Barcelona
\\
Departament de Matemàtiques
\\
Universitat Autònoma de Barcelona and Centre de Recerca Matemàtica
\\
Barcelona
\\
Catalonia.}
\email{xtolsa@mat.uab.cat}
\date{\today}

\subjclass[2010]{}
\thanks{S.B. was  supported by the Simons foundation grant ``Travel support for Mathematicians'' (grant number 959861). B.P., O.T. and X.T. were supported by the European Research Council (ERC) under the European Union's Horizon 2020 research and innovation programme (grant agreement 101018680). X.T. is also partially supported by MICINN (Spain) under the grant PID2020-114167GB-I00, the María de Maeztu Program for units of excellence (Spain) (CEX2020-001084-M), and 2021-SGR-00071 (Catalonia).}

\begin{document}
\allowdisplaybreaks

\begin{abstract}
Suppose that $\Omega \subset \ree$, $n\geq1$, is a uniform domain with $n$-Ahlfors regular boundary and $L$ is a (not necessarily symmetric) divergence form elliptic, real, bounded operator in $\Omega$. We show that the corresponding elliptic measure $\omega_L$ is quantitatively absolutely continuous with respect to surface measure of $\pom$ in the sense that $\omega_L \in A_\infty(\sigma)$ if and only if any bounded solution $u$ to $Lu = 0$ in $\Omega$ is $\eps$-approximable for any $\eps \in (0,1)$. By $\eps$-approximability of $u$ we mean that there exists a function $\Phi = \Phi^\eps$ such that $\|u-\Phi\|_{L^\infty(\Omega)} \le \eps \|u\|_{L^\infty(\Omega)}$ and the measure $\widetilde{\mu}_\Phi$ with $d\widetilde{\mu} = |\nabla \Phi(Y)| \, dY$ is a Carleson measure with $L^\infty$ control over the Carleson norm. 

As a consequence of this approximability result, we show that boundary $\BMO$ functions with compact support can have Varopoulos-type extensions even in some sets with unrectifiable boundaries, that is, smooth extensions that converge non-tangentially back to the original data and that satisfy $L^1$-type Carleson measure estimates with $\BMO$ control over the Carleson norm. Our result complements the recent work of Hofmann and the third named author who showed the existence of these types of extensions in the presence of a quantitative rectifiability hypothesis.
\end{abstract}
 
\maketitle 
\tableofcontents

\section{Introduction}

Carleson measures (see Definition \ref{CME.def}) and their connections to geometry, harmonic analysis and partial differential equations (PDE) have been studied actively over the previous decades; see \cite{mattila} for a recent survey related to many key developments and other relevant research. Carleson-type  estimates are particularly powerful for measures $\mu_F$ such that $d\mu_F = |\nabla F(Y)|^2 \dist(Y,\pom) \, dY$ and $F$ is a solution to an elliptic PDE in the set $\Omega$: they can be used to, for example, characterize functions of bounded mean oscillation (BMO) \cite{FefSt}, quantitative boundary geometry \cite{HMM,GMT} and absolute continuity properties of harmonic measure \cite{HofPLe}. As it is discussed in \cite[Chapters VI and VIII]{Gar-BAF}, Carleson-type estimates for measures $\widetilde{\mu}_F$ such that $d\widetilde{\mu}_F = |\nabla F(Y)| \, dY$ would be very powerful but they fail even for harmonic functions in the unit disk. To circumvent this problem, Varopoulos \cite{Var2} introduced a way to approximate harmonic functions in $L^\infty$ sense by other functions satisfying these estimates. This \emph{$\eps$-approximability} theory has been studied from many points of view in the past years \cite{Gar-BAF,dahlberg_approximation,KKoPT,HKMP,HMM,hyt-r,GMT,AGMT,HT1,BT,BH_fatou,Gar-eps}.

The initial motivation behind $\eps$-approximability theory was to prove an extension theorem for BMO functions inspired by Carleson's Corona Theorem \cite{carleson}. Varopoulos \cite{Var1,Var2} showed that any compactly supported BMO function $f$ in $\RR^n$ has a smooth extension $V_f$ to the upper half-space such that the extension converges non-tangentially (see Definition \ref{defin:non-tangential}) back to $f$ and $|\nabla V_f(Y)| \, dY$ defines a Carleson measure. Inspired by the power of these estimates, Hofmann and the third named author \cite{HT2} recently showed that \emph{uniform rectifiability} (see Definition \ref{UR.def}) of the boundary is enough to guarantee the existence of these Varopoulos-type extensions. Since Carleson-type estimates for harmonic functions can be used to characterize uniform rectifiability \cite{HMM,GMT} or even stronger geometric properties \cite{AHMMT}, it is natural to ask if the existence of Varopoulos-type extensions (which satisfy better Carleson-type estimates than harmonic functions) characterizes some quantitative geometric properties for the boundary.

In this paper, we show that the existence of Varopoulos-type extensions does not characterize uniform rectifiability but they can exist even in sets with unrectifiable boundaries. Our main result is the following $\eps$-approximability result which can be used to build Varopoulos-type extensions. Throughout, set $\sigma \coloneqq \H^n|_{\pom}$.

\begin{theorem}\label{theorem:approximability}
Let $\Omega \subset \ree$, $n\geq1$, be a uniform domain (see Definition \ref{unifdom.def}) with $n$-Ahlfors regular boundary (see Definition \ref{AR.def}). Let $L=-\dv A\nabla$ be a real, not necessarily symmetric, bounded elliptic operator in $\Omega$ such that the corresponding elliptic measure $\hm_L$ satisfies $\omega_L\in A_\infty(\sigma)$ (see Definition \ref{defin:a_infty}),  and let $\eps \in (0,1)$. Then any solution  $u \in W^{1,2}_{\loc}(\Omega) \cap L^\infty(\Omega)$ to $Lu = 0$ in $\Omega$ is \emph{$\eps$-approximable}: there exists a constant $C_\eps$ and a function $\Phi = \Phi^{\eps} \in C^\infty(\Omega)$ such that 
\begin{enumerate}
\item[i)] $\|u - \Phi\|_{L^\infty(\Omega)} \le \eps \|u\|_{L^\infty(\Omega)}$,
\item[ii)] $\Phi$ satisfies a quantitative $L^1$-type Carleson measure estimate
\begin{align*}
\sup_{x \in \pom, r > 0} \frac{1}{r^n} \iint_{B(x,r) \cap \Omega} |\nabla \Phi(Y)|\, dY \le C_\eps \|u\|_{L^{\infty}(\Omega)},
\end{align*}
\item[iii)] $|\nabla \Phi(X)| \le \tfrac{C_\eps \|u\|_{L^\infty(\Omega)}}{\delta(X)}$ for every $X \in \Omega$,
\item[iv)] if $|X-Y| \ll \dist(X,\pom)$, then $|\Phi(X) - \Phi(Y)| \le \tfrac{C_\eps \|u\|_{L^\infty(\Omega)}}{\delta(X)}|X-Y|$,
\item[v)] there exists a function $\varphi \in L^\infty(\pom)$ such that
\begin{align*}
\lim_{Y \to x, \, \nt} \Phi(Y) = \varphi(x) \text{ for } \sigma\text{-a.e. } x \in \pom.
\end{align*}
\end{enumerate}
The notation $\lim_{Y \to x, \, \nt}$ means non-tangential convergence (see Definition \ref{defin:non-tangential}) and $\delta(\cdot) \coloneqq \dist(\cdot,\pom)$. Here, $C_\eps$ depends on $\eps$, the structural constants related to $\Omega$ and $\pom$, ellipticity and the $\omega_L\in A_\infty(\sigma)$ constants.
\end{theorem}

The proof of Theorem \ref{theorem:approximability} borrows some ideas from the proof of \cite[Theorem 1.3]{HMM} (which is an adaptation of the classical construction in \cite[Chapter VIII, Theorem 6.1]{Gar-BAF}), but very quickly our argument must differ significantly. In \cite{HMM}, the authors constructed the approximators for harmonic functions in the presence of a quantitative rectifiability hypothesis. This allowed them to construct approximating chord-arc domains which, in turn, allowed them to use an ``$N \lesssim S$'' estimate for harmonic functions. This was the most delicate part of their argument and our main challenges are strongly related to overcoming the fact that we cannot use the same tools due to our geometry (our boundary may be purely unrectifiable) and our operator $L$ (no control on its structure or smoothness). 
 
We also obtain the following converse to Theorem \ref{theorem:approximability}.

\begin{theorem}\label{theorem:converse}
  Let $\Omega \subset \ree$, $n\geq1$, be a uniform domain (see Definition \ref{unifdom.def}) with $n$-Ahlfors regular boundary (see Definition \ref{AR.def}), and let $L=-\dv A\nabla$ be a real, not necessarily symmetric, bounded elliptic operator in $\Omega$. Suppose also that every solution   $u \in W^{1,2}_{\loc}(\Omega) \cap L^\infty(\Omega)$ to $Lu = 0$ is $\eps$-approximable for every $\eps \in (0,1)$ in the sense of Theorem \ref{theorem:approximability}. Then $\omega_L \in A_\infty(\sigma)$ (see Definition \ref{defin:a_infty}).
\end{theorem}
In particular, by combining Theorem \ref{theorem:approximability} and Theorem \ref{theorem:converse} with \cite[Theorem 1.1]{CHMT}, we get a new characterization of the $A_\infty$ property of elliptic measure on uniform domains.

\begin{corollary}\label{corollary:characterization}
  Let $\Omega \subset \ree$, $n\geq1$, be a uniform domain (see Definition \ref{unifdom.def}) with $n$-Ahlfors regular boundary (see Definition \ref{AR.def}), and let $L=-\dv A\nabla$ be a real, not necessarily symmetric, bounded elliptic operator in $\Omega$. The following conditions are equivalent:
  \begin{enumerate}
    \item[(a)] $\omega_L \in A_\infty(\sigma)$,
    
    \item[(b)] every solution $u \in W^{1,2}_{\loc}(\Omega) \cap L^\infty(\Omega)$ to $Lu = 0$ in $\Omega$ is $\eps$-approximable for any $\eps \in (0,1)$ in the sense of Theorem \ref{theorem:approximability},
    
    \item[(c)] every solution $u \in W^{1,2}_{\loc}(\Omega) \cap L^\infty(\Omega)$ to $Lu = 0$ in $\Omega$ satisfies an $L^2$-type Carleson measure estimate with $L^\infty$ control over the Carleson norm: there exists $C \ge 1$ such that
    \begin{align*}
      \sup_{\substack{x \in \pom \\ r \in (0,\diam(\pom))}} \frac{1}{r^n} \iint_{B(x,r) \cap \Omega} |\nabla u(X)|^2 \dist(X,\partial\Omega) \, dX \le C\|u\|_{L^\infty{\Omega}}^2.
    \end{align*}
  \end{enumerate}
\end{corollary}
Only the implications ``(a) $\implies$ (b)'' and ``(b) $\implies$ (a)'' in Corollary \ref{corollary:characterization} are new. The equivalence ``(a)$\iff$ (c)'' was already shown in \cite{CHMT} for $n\geq2$, while the case $n=1$ follows as a particular case of a more general result in \cite{fp}, with a  similar method of proof. 

In the setting of a uniform domain with Ahlfors regular boundary, the conditions (a), (b) and (c) in Corollary \ref{corollary:characterization} are known to be equivalent with uniform rectifiability of $\pom$ for the special case $L = -\Delta$ \cite{HMUR1, HMUUR2,  HMM, GMT}, or for $L = -\dv A\nabla$ where $A$ is a locally Lipschitz symmetric matrix such that $|\nabla A|$ satisfies an $L^1$-type Carleson measure condition \cite{CHMT, HMT,AGMT}. Even for the aforementioned operators with nice structure, all the available proofs in the literature of the implications ``(a) $\implies$ (b)'' or ``(c) $\implies$ (b)'' in Corollary \ref{corollary:characterization} rely on uniform rectifiability techniques  in a decisive fashion, and do not extend to domains with rougher boundaries. The main novelty of this manuscript is  that we succeed in proving the implication ``(a) $\implies$ (b)'' without appealing to uniform rectifiability theory. This allows us to establish the equivalence ``(a)$\iff$ (b)'' for arbitrary elliptic operators in settings that are beyond \emph{chord-arc domains} (see Definition \ref{cad.def}). For further characterizations of the  $\omega_L\in A_{\infty}(\sigma)$ property for arbitrary real divergence form elliptic operators, see \cite{cdmt22} and \cite{mpt22}.

Let us see an immediate corollary of our new characterization of the $\omega_L\in A_\infty(\sigma)$ property. We say that $L$ is a \emph{Dahlberg--Kenig--Pipher operator} if $A\in\operatorname{Lip}_{\loc}(\Omega)$ with $|\nabla A|\dist(\cdot,\partial\Omega)\in L^{\infty}(\Omega)$, and the measure $\mu_A$ such that $d\mu_A = |\nabla A(X)|^2\dist(X,\partial\Omega)\,dX$ is a Carleson measure. For these operators, the conditions (a) and (c) in Corollary \ref{corollary:characterization} are equivalent with uniform rectifiability of $\partial\Omega$ when $\Omega$ is a uniform domain with Ahlfors regular boundary \cite{HMMTZ}. Combining the main result of \cite{HMMTZ} with Corollary \ref{corollary:characterization} gives us a new result for Dahlberg--Kenig--Pipher operators:

\begin{corollary} Let $\Omega \subset \ree$, $n\geq2$, be a uniform domain (see Definition \ref{unifdom.def})   with $n$-Ahlfors regular boundary (see Definition \ref{AR.def}), and let $L=-\dv A\nabla$ be a not necessarily symmetric Dahlberg--Kenig--Pipher operator in $\Omega$. The following   are equivalent:
	\begin{enumerate}
		\item[(a)] $\partial\Omega$ is uniformly rectifiable (see Definition \ref{UR.def}).
		\item[(b)] every solution $u \in W^{1,2}_{\loc}(\Omega) \cap L^\infty(\Omega)$ to $Lu = 0$ in $\Omega$ is $\eps$-approximable for any $\eps \in (0,1)$ in the sense of Theorem \ref{theorem:approximability}. 
	\end{enumerate}
\end{corollary}

As a consequence of Theorem \ref{theorem:approximability} and the techniques in \cite{HT2}, we get the following generalization of the Varopoulos extension theorem \cite{Var1,Var2}:
\begin{theorem}\label{main.thrm}
Let $\Omega \subset \ree$, $n\geq1$, be a uniform domain (see Definition \ref{unifdom.def}) with Ahflors regular boundary (see Definition \ref{AR.def}). If there exists a divergence form elliptic operator $L = -\div A\nabla$ such that the corresponding elliptic measure $\hm_L$ satisfies $\omega_L\in A_\infty(\sigma)$ (see Definition \ref{defin:a_infty}), then every $f \in \BMO_{\smallc}(\pom)$ has a Varopoulos extension in $\Omega$. That is, there exists   $F$ with the following properties:
\begin{enumerate}
\item $F \in C^\infty(\Omega)$ and $|\nabla F(X)| \lesssim  \tfrac{\|f\|_{\BMO(\pom)}}{\delta(X)}$, for all $X \in \Omega$,
\item $\lim_{Y \to x, \, \nt} F(Y) = f(x)$ for $\sigma$-a.e. $x \in \pom$, and
\item $|\nabla F(Y)|$ is the density of a Carleson measure in the sense that 
\[\sup_{r > 0, x \in \pom} \frac{1}{r^n} \iint_{B(x,r) \cap \Omega} |\nabla F(Y)| \, dY \le C\|f\|_{\BMO(\pom)}.\]
\end{enumerate}
The notation $\lim_{Y \to x, \, \nt}$ means non-tangential limit and $\delta(\cdot) \coloneqq \dist(\cdot,\pom)$. Here, $C$ depends on the structural constants related to $\Omega$ and $\partial\Omega$, and the $\omega_L\in A_\infty(\sigma)$ constants.
\end{theorem}

Theorem \ref{main.thrm} is not (and is not meant to be) a generalization of the main result in \cite{HT2} where an extension theorem of this type was proven in the presence of a quantitative rectifiability hypothesis for the boundary. The novelty of Theorem \ref{main.thrm} is that its assumptions hold for some sets with very rough boundaries. In particular, recently David and Mayboroda \cite{DM-4cc} showed that the key hypothesis of Theorem \ref{main.thrm} holds for the exterior of the $4$-corner Cantor set:

\begin{theorem}[{\cite[Section 4]{DM-4cc}}]\label{DMthrm.thrm}
  Let $\Omega$ be the complement of the $4$-corner Cantor set in $\re^2$ (see Section \ref{section:example}). There exists a divergence form elliptic operator $L = -\div A\nabla$ in $\Omega$ such that $\omega_L\in A_\infty(\sigma)$  (see Definition \ref{defin:a_infty}). More precisely, one can take the matrix $A$ to be diagonal and equal to the identity outside a ball of radius $1$ concentric with the Cantor set and so that the $L$-elliptic measure  with pole at $\infty$
 equals $\sigma/\sigma(\pom)$. 
\end{theorem}

The remarkable thing about Theorem \ref{DMthrm.thrm} is that since the complement of the $4$-corner Cantor set is an unbounded uniform domain with unrectifiable $1$-Ahlfors regular boundary, we know that harmonic measure for this set cannot satisfy the $A_\infty(\sigma)$ condition (see, for example, \cite{AHMMT}). Thus, constructing operators like this is highly non-trivial. In Section \ref{section:example}, we take the David--Mayboroda example and use it to build an example of a similar operator in $\re^3$ (see Proposition \ref{proposition:3-dimensional_example}).

By combining Theorems \ref{main.thrm} and \ref{DMthrm.thrm} and Proposition \ref{proposition:3-dimensional_example}, we get the following:
\begin{corollary}\label{corollary:main}
 There exist uniform domains $\Omega$ with unrectifiable Ahlfors regular boundaries $\pom$ in $\re^2$ and $\re^3$ such that every function $f \in \BMO(\pom)$ with compact support has a Varopoulos extension in $\Omega$. In particular, the existence of Varopoulos extensions does not imply rectifiability for the boundary.
\end{corollary}

By Corollary \ref{corollary:main}, $L^1$-type Carleson measure estimates are simultaneously too strong and too weak from the point of view of the David--Semmes theory: harmonic functions fail these estimates even in the unit disk but the existence of non-harmonic extensions that satisfy these estimates does not imply even qualitative rectifiability for the boundary.

Finally, we want to mention an upcoming paper that is closely related to our results. As we were finishing this work, we were informed\footnote{Personal communication.} about an upcoming manuscript by M. Mourgoglou and T. Zacharopoulos, where the authors construct Varopoulos extensions in corkscrew domains with Ahlfors regular boundaries satisfying a mild quantitative connectivity hypothesis by different methods.

The paper is organized as follows. In Section \ref{section:preliminaries}, we discuss basic definitions and consider some key tools from dyadic analysis and elliptic PDE theory. In Section \ref{section:set-up_for_approximability}, we prove some important preliminary estimates for Theorem \ref{theorem:approximability}, and in Section \ref{section:proof_of_approximability}, we prove Theorem \ref{theorem:approximability}. In Section \ref{section:proof_of_converse}, we prove Theorem \ref{theorem:converse} (and hence, Corollary \ref{corollary:characterization}). Finally, in Section \ref{section:proof_of_varopoulos} we sketch the proof of Theorem \ref{main.thrm} and in Section \ref{section:example} we construct a David--Mayboroda-type example in $\mathbb{R}^3$ (which completes the proof of Corollary \ref{corollary:main}).

\section{Preliminaries}
\label{section:preliminaries}

Throughout, we let $\Omega \subset \ree$ be an open set with $n \ge 1$. We say that $\Omega$ is a \emph{domain} if it is also connected.

Usually, we use capital letters $X,Y,Z$, and so on to denote points in $\Omega$, and lowercase letters $x,y,z$, and so on to denote points in $\pom$. For $X \in \ree$ and $r > 0$, we let $B(X,r)$ be the Euclidean open ball of radius $r$ centered at $X$. The letters $c$ and $C$ and their obvious variations denote constants that depend only on dimension, $n$-Ahlfors regularity constant (see Definition \ref{AR.def}), corkscrew constant (see Definition \ref{CS.def}),  Harnack chain constants, ellipticity constants (see Section \ref{section:elliptic_pde}), and so on. We call these kinds of constants \emph{structural constants}. We write $a \lesssim b$ if $a \le Cb$ for a structural constant $C$ and $a \approx b$ if $C_1 b \le a \le C_2 b$ for structural constants $C_1$ and $C_2$.

\subsection{Uniform domains, chord-arc domains, Ahlfors regularity and uniform rectifiability}
\label{subsection:uniform_domains}

\begin{definition}[Corkscrew condition]\label{CS.def}
We say that a domain $\Omega \subset \ree$ satisfies the \emph{corkscrew condition} if there exists a constant $\gamma > 0$ such that for every $x \in \partial\Omega$ and $r \in (0,\diam(\Omega))$ there exists $Y_{x,r}$ such that
\[B(Y_{x,r}, \gamma r) \subset B(x,r) \cap \Omega.\]
We call $Y_{x,r}$ a \emph{corkscrew point relative to $x$ at scale $r$}.
\end{definition}

\begin{definition}[Harnack chain Condition]
We say a domain $\Omega \subset \ree$ satisfies the \emph{Harnack chain condition} if there exists a uniform constant $C$ such that for every $\rho > 0$ and $\Lambda \ge 0$ and $X, X' \in \Omega$ with $\dist(X,\pom), \dist(X',\pom) \ge \rho$ and 
$|X - X'| \le \Lambda \rho$ there exists a chain of open balls $B_1,\dots, B_J$ with $J \le N(\Lambda)$ with $X \in B_1$, $X' \in B_J$, $B_{j} \cap B_{j + 1} \neq \emptyset$ and $C^{-1} \diam(B_j) \le \dist(B_j, \pom) \le C\diam(B_j)$.
\end{definition}

\begin{definition}[Uniform domain]\label{unifdom.def}
We say that a domain $\Omega \subset \ree$ is \emph{uniform} if it satisfies the corkscrew and Harnack chain conditions.
\end{definition}

\begin{definition}[Ahlfors regularity]\label{AR.def}
We say $\Sigma \subset \ree$ is \emph{$n$-Ahlfors regular} (or simply Ahlfors regular) if there exists $C$ such that
\[C^{-1} r^n \le \H^n(B(x,r) \cap \Sigma) \le Cr^n, \quad \text{for each } x \in r \in (0, \diam(\Sigma)).\]
Here and below $\H^n$ denotes the $n$-dimensional Hausdorff measure. 
\end{definition}

\begin{definition}[Chord-arc domain]\label{cad.def}
  We say that a domain $\Omega \subset \ree$ is a \emph{chord-arc domain} if $\Omega$ satisfies the Harnack chain condition, both $\Omega$ and $\text{int} \, \Omega^{\smallc}$ satisfy the corkscrew condition, and the boundary $\pom$ is $n$-Ahlfors regular.
\end{definition}

\begin{definition}[Uniform rectifiability]\label{UR.def}
  Following \cite{DS1}, we say that an $n$-Ahlfors regular set $E \subset \ree$ is \emph{uniformly rectifiable} if it contains ``big pieces of Lipschitz images'' of $\re^n$: there exist constants $\theta, M > 0$ such that for every $x \in E$ and $r \in (0,\diam(E))$ there is a Lipschitz mapping $\rho = \rho_{x,r} \colon \re^n \to \ree$, with Lipschitz norm no larger that $M$, such that
  \begin{align*}
    \mathcal{H}^n\big(E \cap B(x,r) \cap \rho(\{y \in \re^n \colon |y| < r\})\big) \ge \theta r^n.
  \end{align*}
\end{definition}

\subsection{Carleson measures, non-tangential convergence, BMO and local BV}

Given a domain $\Omega\subset\ree$, we set $\sigma \coloneqq \H^n|_{\partial\Omega}$ and $\delta(X) \coloneqq \dist(X,\partial\Omega)$ for $X\in\Omega$.

\begin{definition}[Carleson measures]\label{CME.def}
  \label{defin:carleson_measure_constant}
  We say that a Borel measure $\mu$ in $\Omega$ is a \emph{Carleson measure (with respect to $\partial \Omega$)} if we have
  \begin{align*}
    C_\mu \coloneqq \sup_{x \in \pom, r > 0} \frac{\mu(B(x,r) \cap \Omega)}{r^n} < \infty.
  \end{align*}
  We call $C_\mu$ the \emph{Carleson norm of $\mu$}.
\end{definition}

\begin{definition}[Cones and non-tangential convergence]\label{defin:non-tangential}
  Suppose that $m>1$. For every $x \in \partial \Omega$, the \emph{cone of $m$-aperture at $x$} is the set
  \begin{align}
    \label{defin:cones} \widetilde{\Gamma}(x) \coloneqq \widetilde{\Gamma}^m(x) \coloneqq \{Z \in \Omega \colon \dist(Z,x) <m\dist(Z,\pom)\}.
  \end{align}
  Let $G$ be a function defined in $\Omega$ and $g$ be a function defined on $\pom$. We say that $G$ \emph{converges non-tangentially to $g$ at $x \in \pom$} if there exists $m> 1$ such that we have $\lim_{k \to \infty} G(Y_k) = g(x)$ for every sequence $(Y_k)$ in $\widetilde{\Gamma}^m(x)$ such that $\lim_{k \to \infty} Y_k  = x$. We denote this by $\lim_{Y \to x, \, \nt} G(Y) = g(x)$.
\end{definition}

\begin{definition}[Non-tangential maximal operator]
  We denote the \emph{non-tangential maximal operator} by $N_*$, that is, for a function $u \in L^\infty(\Omega)$, the function $N_*u \colon \pom \to \re$ is defined as
  \begin{align*}
    N_* u(x) = \sup_{X \in \widetilde{\Gamma}(x)} |u(X)|.
  \end{align*}
  We call $N_* u$ the \emph{non-tangential maximal function of $u$}.
\end{definition} 

\begin{remark} The aperture constant $m$ in Definition \ref{defin:non-tangential} does not play a big role in this paper and therefore we do not analyze it in detail for our results; we simply have that the results hold for some uniform aperture constant. Naturally, the aperture constant affects the values of the non-tangential maximal function but since the $L^p$ norms of non-tangential maximal functions that are defined with different aperture constants are comparable (with the comparability constant depending on the aperture constants) (see \cite[Lemma 1]{FefSt} and \cite[Lemma 1.10]{HT1}), this is not important for us. In most computations, it is convenient to use \emph{dyadic cones} (see \eqref{defin:dyadic_cone}) instead of cones of the previous type.
\end{remark}

\begin{definition}[BMO]
  The space $\BMO(\pom)$ (\emph{bounded mean oscillation}) consists of  $f\in L^1_{\loc}(\partial\Omega)$ with
  \begin{align*}
    \|f\|_{\BMO} \coloneqq \sup_\Delta \fint_\Delta |f(y) - \langle f \rangle_\Delta| \, d\sigma(y) < \infty,
  \end{align*}
  where the supremum is taken over all surface balls $\Delta = \Delta(x,r) \coloneqq B(x,r) \cap \pom$. We denote $f \in \BMO_{\smallc}(\pom)$ if  $f$ is a $\BMO$ function with compact support.
\end{definition}

\begin{definition}[Local BV]
  \label{defin:local_bv}
  We say that locally integrable function $f$ has \emph{locally bounded variation in $\Omega$} (denote $f \in \BV_{\loc}(\Omega)$) if for any open relatively compact set $\Omega' \subset \Omega$ the \emph{total variation over $\Omega'$} is finite:
  \begin{align*}
    \iint_{\Omega'} |\nabla f(Y)| \, dY 
    \coloneqq \sup_{\substack{\overrightarrow{\Psi} \in C_0^1(\Omega') \\ \|\overrightarrow{\Psi}\|_{L^\infty(\Omega') \le 1}}} \iint_{\Omega'} f(Y) \, \text{div} \overrightarrow{\Psi}(Y) \, dY < \infty,
  \end{align*}
  where $C_0^1(\Omega')$ is the class of compactly supported continuously differentiable vector fields in $\Omega'$.
\end{definition}

\subsection{Dyadic cubes, Whitney regions and approximating domains}
\label{section:dyadic_and_whitney}

An $n$-Ahlfors regular set $E\subset\bb R^{n+1}$ equipped with the Euclidean distance and surface measure can be viewed as a space of homogeneous type of Coifman and Weiss \cite{coifmanweiss}, with ambient dimension $n+1$. All such sets can be decomposed dyadically in the following sense:
 
\begin{lemma}[\cite{Ch,DS1,hyt-k}]\label{cubes} Assume that $E \subset \mathbb R^{n+1}$ is $n$-Ahlfors regular. Then $E$ admits a \emph{dyadic decomposition} in the sense that there exist constants $a_1 \ge a_0 > 0$
such that for each $k \in \mathbb{Z}$ there exists a collection of Borel sets, $\dd_k$,  which we will call \emph{(dyadic) cubes}, such that
$$
\dd_k \coloneqq \{Q_{j}^k\subset E \colon j\in \mathfrak{I}_k\},
$$
where $\mathfrak{I}_k$ denotes a countable index set depending on $k$, satisfying
\begin{enumerate}
  \item[(i)] for each fixed $k \in \mathbb{Z}$, the sets $Q_j^k$ are disjoint and $E=\cup_{j}Q_{j}^k\,\,$,
  
  \item[(ii)] if $m\geq k$ then either $Q_{i}^{m}\subset Q_{j}^{k}$ or $Q_{i}^{m}\cap Q_{j}^{k}=\emptyset$,
  
  \item[(iii)] for each $k \in \Z$, $j \in \mathfrak{I}_k$ and $m<k$, there is a unique $i \in \mathfrak{I}_m$ such that $Q_{j}^k\subset Q_{i}^m$,

  \item[(iv)] $\diam(Q_{j}^k)\leq a_1 2^{-k}$,

  \item[(v)] for each $Q_{j}^k$, there exists a point $z_j^k \in Q_j^k$ such that $E \cap B(z^k_{j}, a_0 2^{-k}) \subset Q_j^k \subset E \cap B(z^k_{j}, a_1 2^{-k})$.
\end{enumerate}
\end{lemma}

We denote by $\dd = \dd(E)$ the collection of all cubes $Q^k_j$, that is,
  $$\dd \coloneqq \cup_{k} \dd_k.$$
If $E$ is bounded, we ignore cubes where $2^{-k} \gtrsim \diam(\pom)$ (in particular, where $a_0 2^{-k} \ge \diam(\pom)$).
Given a cube $Q = Q_j^k \in \mathbb{D}$, we define the \emph{side-length of $Q$} as $\ell(Q) \coloneqq 2^{-k}$. By Ahlfors regularity and property (v) in Lemma \ref{cubes}, we know that $\ell(Q) \approx \diam(Q)$ and $\H^n(Q) \approx \ell(Q)^n$. Given $Q \in \dd$ and $m \in \Z$, we set
\begin{equation*}
  \dd_Q     \coloneqq \left\{Q'\in \dd \colon Q'\subseteq Q\right\},\qquad   \dd_{m,Q} \coloneqq \left\{Q'\in \dd_Q \colon \ell(Q') = 2^{-m} \ell(Q) \right\}.
\end{equation*}
We call the cubes in the collection $\dd_{1,Q}$ the \emph{children} of $Q$. Notice that by Ahlfors regularity and property (v) in Lemma \ref{cubes}, each cube has a uniformly bounded number of children.
 
Given a cube $Q = Q_j^k \in \dd$, we call the point $z_j^k \in Q$ in property (v) in Lemma \ref{cubes} the \emph{center of $Q$}, denote $x_Q \coloneqq z_j^k$, and set
\[
\Delta_Q \coloneqq E \cap B(z_j^k, a_0 \ell(Q)).
\]
We use Lemma \ref{cubes} to decompose $\pom$, so that $\dd(E) = \dd(\pom) \eqqcolon \dd$. For each $Q \in \dd$, we let $X_Q$ be the corkscrew point relative to $x_Q$ at scale $10^{-5}a_0\ell(Q)$. We have $B(X_Q, \gamma 10^{-5} a_0 \ell(Q)) \subset B(x_Q, 10^{-5} a_0 \ell(Q)) \cap \Omega$, where $\gamma$ is the corkscrew constant in Definition \ref{CS.def}.

For many of our techniques, it is important that we show that some collections of dyadic cubes are quantitatively small in the following sense:
\begin{definition}[Carleson packing condition]\label{defin:carleson_packing}
  Let $\dd$ be a dyadic system on $\pom$ and let $\Ac \subset \dd$. We say that $\Ac$ satisfies a \emph{Carleson packing condition} if there exists a constant $C \ge 1$ such that for any $Q_0 \in \dd$ we have
  \begin{align*}
    \sum_{Q \in \Ac, Q \subset Q_0} \sigma(Q) \le C \sigma(Q_0).
  \end{align*}
  We denote the smallest such constant $C$ by $\C_{\Ac}$.
\end{definition}

Next, we   use a standard decomposition of $\Omega$ into Whitney cubes (see e.g. \cite[Chapter VI]{stein-SIOs}), and then associate a collection of such Whitney cubes to each boundary cube to construct suitable Whitney-type regions. These Whitney regions are modeled after regions of the type $Q \times (\ell(Q)/2,\ell(Q))$ (that is, the upper halves of Carleson boxes) in the simpler geometry of the upper half-space. For this, we recall the construction found in \cite{HMUR1} noting that we make some changes to the notation therein (following the notation of more recent papers, e.g. \cite{HMM}). We let $\W = \{I\}_I$ denote a Whitney decomposition of $\Omega$, with the properties that each $I$ is a closed $(n+1)$-dimensional cube satisfying
\begin{align*}
  4\diam(I) \le \dist(4I, \pom) \le \dist(I, \pom) \le 40\diam(I),
\end{align*}
where $4I$ is the standard concentric Euclidean dilate of a cube; the interiors of the cubes $I$ are disjoint, and for all $I_1, I_2 \in \W$ with $I_1 \cap I_2 \neq \emptyset$ we have
\begin{align*}
  \frac{1}{4} \diam(I_1) \le \diam(I_2) \le 4\diam(I_1).
\end{align*}
For $I \in \W$ we let $\ell(I)$ denote the side length of $I$.

For each cube $Q \in \dd$ and constant $K \ge K_0$, with $K_0$ to be described momentarily, we associate an initial collection of Whitney cubes
\[W_Q(K) \coloneqq \{I \in \W \colon K^{-1} \ell(I) \le \ell(Q) \le K \ell(I), \dist(I, Q) \le K \ell(Q)\}.\]
We choose $K_0$ depending on the constants in the corkscrew condition and the Ahlfors regularity condition, insisting on two conditions being met:
\begin{enumerate}
  \item[(1)] If $X \in \Omega$ with $\dist(X, \pom) \le 10^5 \diam(\pom)$ then $X \in I \in W_Q(K)$ for some $Q \in \dd$.
  \item[(2)] For any $Q \in \dd$, we have $B(X_Q, \dist(X_Q,\pom)/2) \subseteq \cup_{I \in W_Q(K)} I$, and if $Q' \in \dd$ is another cube such that $Q' \subset Q$ with $\ell(Q') = \tfrac{1}{2}\ell(Q)$, then we also have $B(X_{Q'}, \dist(X_{Q'},\pom)/2) \subseteq \cup_{I \in W_Q(K)} I$.
\end{enumerate}
Of course, condition (1) above is automatically satisfied if $\diam(\pom) = \infty$.

Following \cite[Section 3]{HMUR1}, we augment the collection $\W_Q(K)$ as follows. For each $I \in W_Q(K)$, we take a Harnack chain $H(I)$ from the center of $I$ to the corkscrew point $X_Q$, and we let $W_{Q,I}(K)$ be be the collection of Whitney cubes in $\W$ that meet at least one ball in the chain $H(I)$. We then set $\W^*_Q(K) \coloneqq \bigcup_{I \in W_Q(K)} W_{Q,I}(K)$. Finally, for a small dimensional parameter $\tau$ (this is the parameter $\lambda$ in \cite[Section 3]{HMUR1}), we define the \emph{Whitney region relative to $Q$} as
\begin{align}\nonumber
  \label{defin:whitney_region} U_Q = U_Q(K_0) \coloneqq \bigcup_{I \in \W^*_Q(K_0)} (1 + \tau) I.
\end{align}
By construction, we know that if $X \in U_Q$, then
\begin{equation}\label{K1def.eq}
  K_1^{-1} \ell(Q) \le \dist(X, \pom) \le \dist(X, Q) \le K_1 \ell(Q),
\end{equation}
where $K_1$ depends on $K_0$, the dimension and the Harnack chain condition. For $\kappa \gg K_0$ to be chosen, we also define the following fattened versions of the Whitney regions:
\[
U_Q^* = U_Q(\kappa) = \bigcup_{I \in \W^*_Q(\kappa)} (1 + \tau)I,\qquad U_Q^{**} = \bigcup_{I \in \W^*_Q(\kappa)} (1 + 2\tau)I,
\]
that is, $U_Q^*$ is constructed the same way as $U_Q$ but we replace the constant $K_0$ by $\kappa$ (similarly for $U_Q^{**}$). We describe the reasoning and choice of $\kappa$ in the next subsection. We note that for $\tau$ small enough the regions $U_Q$, $U^*_Q$ and $U_Q^{**}$ have bounded overlaps, that is, for a collection of dyadic cubes $\mathscr{Q}$ and the $(n+1)$-dimensional Lebesgue measure $|\cdot|$ we have $|\bigcup_{Q \in \mathscr{Q}} U_Q^{**}| \approx \sum_{Q \in \mathscr{Q}} |U_Q^{**}|$.

Using the Whitney regions above, we can now define objects like sawtooth regions, Carleson boxes and dyadic cones. Let $Q_0 \in \dd$ be a fixed cube and $\F \subset \dd_{Q_0}$ a collection of pairwise disjoint cubes. We set
\[\dd_{\mathcal{F}, Q_0} \coloneqq \dd_{Q_0} \setminus \cup_{Q \in \mathcal{F}} \dd_Q.\] 
We then define the \emph{local sawtooth relative to $\F$} (and its fattened version) as
\[\Omega_{\mathcal{F}, Q_0} \coloneqq \interior\Big(\bigcup_{Q \in\dd_{\mathcal{F}, Q_0}} U_Q\Big),\qquad \Omega^*_{\mathcal{F}, Q_0} \coloneqq \interior\Big(\bigcup_{Q \in\dd_{\mathcal{F}, Q_0}} U^*_Q\Big).\]
In the special case where $\F = \emptyset$, we write $T_{Q_0} = \Omega_{\mathcal{F}, Q_0}$ and $T^*_{Q_0} = \Omega^*_{\mathcal{F}, Q}$, that is,
\[T_{Q_0} \coloneqq \interior\Big(\bigcup_{Q \in\dd_{Q_0}} U_Q\Big),\qquad T_{Q_0}^* \coloneqq \interior\Big(\bigcup_{Q \in\dd_{Q_0}} U^*_Q\Big),\qquad T_{Q_0}^{**} \coloneqq \interior\Big(\bigcup_{Q \in\dd_{Q_0}} U^{**}_Q\Big)\] 
We call $T_{Q_0}$ the \emph{Carleson box relative to $Q_0$} and $T_{Q_0}^*$ and $T_{Q_0}^{**}$ its fattened versions. Given a cube $Q_0 \in \dd$ and a point $x \in Q_0$, we also define the \emph{(truncated) dyadic cone at $x \in \pom$}, $\Gamma(x)$, by setting
\begin{align}
  \label{defin:dyadic_cone} \Gamma(x) \coloneqq \Gamma_{Q_0}(x) \coloneqq \interior\Big( \bigcup_{Q \in \dd_{Q_0} \colon x \in Q} U_Q \Big).
\end{align}
Notice that $\Gamma(x) = \Omega_{\F_x, Q_0}$, where $\F_x$ is the collection of maximal\footnote{Since the sizes of the cubes in $\{Q \in \dd_{Q_0} \colon x \notin Q\}$ are bounded from above, we can always choose the maximal cubes. In the case of non-truncated dyadic cones (that is, cones of the form $\interior\left( \bigcup_{Q \in \dd \colon x \in Q} U_Q \right)$), we can choose the maximal cubes if the dyadic system $\dd$ forms a tree-like structure. This structure can be achieved by choosing the center points of the dyadic cubes in a suitable way (see, for example, \cite[Theorem 2.4]{hyt-tap}).} (and hence, disjoint) cubes in the collection $\{Q \in \dd_{Q_0} \colon x \notin Q\}$. It is straightforward to verify that there exists uniform constants $m_1, m_2 > 1$ such that $\Gamma(x)$ contains a truncated version of the cone $\widetilde{\Gamma}^{m_1}(x)$ and it is contained in a truncated version of the cone $\widetilde{\Gamma}^{m_2}(x)$, where $\widetilde{\Gamma}^{m_1}(x)$ and $\widetilde{\Gamma}^{m_2}(x)$ are cones of the type \eqref{defin:cones}. Thus, since the aperture of the cones is not important for us, we mostly use dyadic cones when studying non-tangential convergence.

By the following lemma, the Whitney regions, sawtooth regions, Carleson boxes, and truncated dyadic cones inherit many quantitative geometric properties from $\Omega$:
\begin{lemma}[{\cite[Lemma 3.61]{HMUR1}}]\label{lemma:hm}
  Suppose that $\Omega\subset\bb R^{n+1}$ is a uniform domain with $n$-Ahlfors regular boundary. Let $Q_0 \in \dd(\pom)$ be a cube and $\F \subset \dd_{Q_0}$ be a collection of pairwise disjoint cubes. Then $\Omega_{\mathcal{F}, Q}$ and $\Omega^*_{\mathcal{F}, Q}$ are uniform domains with $n$-Ahlfors regular boundary whose structural constants depend only on the dimension, the structural constants of $\Omega$, and the constant $\kappa$.   In particular, the Whitney regions $U_Q$, $U_Q^*$ and $U_Q^{**}$, the Carleson boxes $T_Q$, $T_Q^*$ and $T_Q^{**}$ and the truncated dyadic cones $\Gamma(x)$ are uniform domains with $n$-Ahlfors regular boundaries, with structural constants depending only on the dimension and the   structural constants of $\Omega$ and the constant $\kappa$.
\end{lemma}

\subsection{The choice of the parameter $\kappa$}\label{kappachoice.subsect}
In contrast to the setting of the upper half space, we do not define the sawtooths by removing Whitney regions. This is due to the overlaps of the regions $U_Q$: we may encounter situations where for $Q_0 \in \dd(\pom)$ and a collection of pairwise disjoint cubes $\F \subset \dd_{Q_0}$ there exists a cube $Q \in \dd_{\F, Q_0}$ such that $\overline{U_Q}$ does not contribute to the boundary of $\Omega_{\F, Q}$. That being said, if $\kappa$ is chosen large enough, then the fattened Whitney region $U_Q^*$ meets the boundary of the unfattened region $\Omega_{\mathcal{F}, Q}$ on a portion roughly the measure of $Q$. We will consider this in Section \ref{section:proof_of_approximability} where it will be convenient for us, but we prove the technical estimates that give this property below.

Let us fix a cube $Q \in \dd$. Recall that $X_Q$ is a corkscrew point relative to $x_Q \in Q$ at scale $r_Q \coloneqq 10^{-5}a_0 \ell(Q)$ and  $\Delta_Q = B(x_Q, 10^5 r_Q) \cap \pom \subset Q$ is the surface ball associated to $Q$. We let $\hat{x}_Q \in \pom$ denote a touching point for $X_Q$, that is, a point such that $|X_Q - \hat{x}_Q| = \delta(X_Q)$. By triangle inequality, $|x_Q - \hat{x}_Q|<2r_Q$, and thus,
\begin{equation}\label{tbalsbal.eq}
  \hat{\Delta}_Q = B(\hat{x}_Q, 10^3 r_Q) \cap \pom \subset \Delta_Q \subseteq Q.
\end{equation}
For every $\theta \in (0,1)$, we let 
\begin{align*}
  P_Q(\theta) \coloneqq \hat{x}_Q + \theta(X_Q - \hat{x}_Q)
\end{align*}
be the ``$\theta$-point'' on the directed line segment from $\hat{x}_Q$ to $X_Q$. Then, by definitions,
\begin{equation}\label{Pqdist.eq}
\gamma\theta r_Q \le |P_Q(\theta) - \hat{x}_Q| = \dist(P_Q(\theta), \pom) \le \theta r_Q.
\end{equation}

\begin{lemma}\label{thetachoice.cl}
There exists $\theta_0 \in (0,1)$ depending on $K_0$ and the structural constants such that if for some $Q'\in\mathbb D$ we have that
\begin{align*}
  B\big(P_Q(\theta_0), \tfrac{\gamma\theta_0}{10} r_Q \big) \cap U_{Q'} \neq \emptyset,
\end{align*}
then $Q' \subseteq Q$ and $\ell(Q') < \ell(Q)$. 
\end{lemma}

\begin{proof}
Suppose that $\theta_0 \le 1 / 4C(K_1)^2$ for a large structural constant $C \ge 1$ and
\begin{align}
  \label{point_x} X \in B\big(P_Q(\theta_0), \tfrac{\gamma\theta_0}{10} r_Q \big) \cap U_{Q'}
\end{align}
for some $Q'$, where $K_1$ is the constant in \eqref{K1def.eq}. By \eqref{K1def.eq}, \eqref{point_x} and \eqref{Pqdist.eq}, it holds that 
\begin{align}
  \nonumber \ell(Q') &\le K_1 \dist(X,\pom) \\
  \label{lQpsmall.eq}         &\le K_1 \bigg(\dist(P_Q(\theta_0),\pom) + \frac{\gamma\theta_0}{10} r_Q \bigg) \le 2K_1 \theta_0 r_Q = 2K_1 \theta_0 10^{-5} a_0 \ell(Q).
\end{align}
In particular, we have $\ell(Q') < \ell(Q)$.

To show that $Q' \subset Q$, we first notice that we have
\begin{align*}
  |X - x_{Q'}| \le C K_1 \ell(Q')
\end{align*}
for a structural constant $C \ge 1$ by \eqref{K1def.eq} and the fact that $\diam(Q') \approx \ell(Q')$. This and \eqref{lQpsmall.eq} then give us
\begin{align*}
  |X - x_{Q'}| \le 2 C (K_1)^2 \theta_0 10^{-5} a_0 \ell(Q).
\end{align*}
Thus, by \eqref{Pqdist.eq} and \eqref{point_x}, it holds that
\begin{align*}
  |X - \hat{x}_Q| \le 2\theta_0 r_Q = 2\theta_0 10^{-5} a_0 \ell(Q).
\end{align*}
Combining the previous two inequalities then gives us
\begin{align}\label{xclosexhat.eq}
  |\hat{x}_Q - x_{Q'}| \le 4 C (K_1)^2 \theta_0 10^{-5} a_0 \ell(Q) = 4C (K_1)^2 \theta_0 r_Q.
\end{align}
In particular,
\begin{align*}
  |x_{Q'} - x_Q| \le |x_{Q'} - \hat{x}_Q| + |\hat{x}_Q - x_Q| \le 4C (K_1)^2 \theta_0 r_Q + 2r_Q < 3r_Q < a_0 \ell(Q),
\end{align*}
by \eqref{xclosexhat.eq}, the fact that $|\hat x_Q-x_Q|<2r_Q$,  and the choice $\theta_0 \le 1 / 4C(K_1)^2$. Thus, $x_{Q'} \in \Delta_Q \subset Q$. Since $Q' \cap Q \neq \emptyset$ and $\ell(Q') < \ell(Q)$, we know that $Q' \subset Q$, which is what we wanted.
\end{proof}

Let us then fix $\theta_0$ so that Lemma \ref{thetachoice.cl} holds. For $Q \in \dd$, we set
\begin{align*}
  \Xi_Q \coloneqq \bigcup_{\theta \in [\theta_0,1]} B\big(P_Q(\theta), \tfrac{\gamma\theta_0}{10} r_Q \big),
\end{align*}
which is a cylinder-like object. We get the following straightforward lemma:

\begin{lemma}\label{lemma:cylinder_inclusion}
  Let $Q \in \dd$ be a fixed cube and let $\mathscr{Q}_Q$ be the collection of cubes that share the same dyadic parent as $Q$, that is,
  \begin{align*}
    \mathscr{Q}_Q \coloneqq \{P \in \dd \colon P,Q \subset Q_0 \text{ for a cube } Q_0 \in \dd \text{ such that } \ell(P) = \ell(Q) = \tfrac{1}{2} \ell(Q_0)\}.
  \end{align*}
  Let $\kappa \gg \max\{K_0,(\theta_0)^{-1}\}$ and $X \in \Xi_P$ for some $P \in \mathscr{Q}_Q$. Then $X \in U_Q^*$.
\end{lemma}

\begin{proof}
  Let $\kappa \gg \max\{K_0,(\theta_0)^{-1}\}$ and $X \in \Xi_P$ for some $P \in \mathscr{Q}_Q$. By the Whitney decomposition, there exists a Whitney cube $I \in \W$ such that $X \in I$. By the definition of $U_Q^*$, it is enough to show that $I \in \W_Q^*(\kappa)$.
  
  By \eqref{Pqdist.eq} and the definitions, we first notice that
  \begin{align*}
    \ell(I) \approx \dist(X,\pom) \approx \gamma \theta_0 r_P \approx \theta_0 \ell(P) = \theta_0 \ell(Q)
  \end{align*}
  with uniformly bounded implicit constants. In particular, since $\theta_0 \gg \tfrac{1}{\kappa}$ and $\kappa \gg 1$, we get
  \begin{align*}
    \frac{1}{\kappa} \ell(I) \le \ell(Q) \le \kappa \ell(I).
  \end{align*}
  On the other hand, since $P \in \mathscr{Q}_Q$, we know that
  \begin{align*}
    \dist(I,Q) \lesssim \dist(I,P) + \ell(Q),
  \end{align*}
  and by \eqref{Pqdist.eq}, \eqref{tbalsbal.eq} and the fact that $X \in \Xi_P$, we know that
  \begin{align*}
    \dist(I,P) \le \dist(X,P) \le |X - \hat{x}_P| \le \frac{\gamma \theta_0}{10}r_P + |X_P - \hat{x}_P| \le 2r_P \le 2\ell(P) = 2\ell(Q).
  \end{align*}
  In particular, since $\kappa \gg 1$, we have
  \begin{align*}
    \dist(I,Q) \le \kappa \ell(Q).
  \end{align*}
  Thus, $I \in \W_Q^*(\kappa)$, which proves the claim.
\end{proof}

Let us also record the following simple lemma for future use:
\begin{lemma}\label{lemma:touching_point_surface_ball}
  Let $Q \in \dd$ and let $X_Q \in \Omega$ be a corkscrew point, $\hat{x}_Q \in \pom$ be a touching point and $r_Q = 10^{-5} a_0 \ell(Q)$ as above. If $Y \in B(P_Q(\theta),r_Q)$ for some $\theta \in [0,1]$ and $\hat{y} \in \pom$ is a point such that $|Y-\hat{y}| = \dist(Y,\pom)$, then $\hat{y} \in \hat{\Delta}_Q \subseteq \Delta_Q \subseteq Q$.
\end{lemma}

\begin{proof}
  By definitions and using \eqref{Pqdist.eq} several times, we get
  \begin{multline*}
    |\hat{y} - \hat{x}_Q|
    \le |\hat{y} - Y| + |Y - P_Q(\theta)| + |P_Q(\theta) - \hat{x}_Q|
    < \dist(Y,\pom) + 2r_Q \\
    \le |Y - P_Q(\theta)| + \dist(P_Q(\theta),\pom) + 2r_Q 
    < r_Q + r_Q + 2r_Q = 4r_Q,
  \end{multline*}
  and thus, $\hat{y} \in \hat{\Delta}_Q \subset \Delta_Q \subset Q$ by \eqref{tbalsbal.eq}.
\end{proof}

\subsection{Elliptic PDE estimates}
\label{section:elliptic_pde}

Here we collect some of the standard estimates for divergence form elliptic operators with real coefficients that will be used throughout the paper. In this section, $\Omega$ always denotes a uniform domain in $\bb R^{n+1}$, $n\geq1$, with $n$-Ahlfors regular boundary.  We recall that a divergence form elliptic operator is of the form 
\begin{align*}
  L (\cdot ) \coloneqq - \div (A \nabla \cdot),
\end{align*}
viewed in the weak sense, where $A$ is a uniformly elliptic matrix, that is, $A = (a_{i,j})_{i,j=1}^{n+1}$ is an $(n+1)\times (n+1)$ matrix-valued function on $\ree$ and there exists a constant $\Lambda$, the ellipticity parameter, such that  
\begin{align}\label{property:ellipticity}\nonumber
  \Lambda^{-1} |\xi|^2 \le A(X)\xi \cdot \xi,\qquad\text{and}\qquad \|a_{i,j}\|_{L^\infty(\ree)} \le \Lambda,
\end{align}
for all $\xi, \zeta \in \ree$ and almost every $X \in \Omega$. We say that a constant depends on ellipticity if it depends on $\Lambda$. Given an open set $\mathcal{O} \subset \ree$ we say a function $u\in W^{1,2}_{\loc}(\mathcal{O})$ is a solution to $Lu = 0$ in $\mathcal{O}$ if 
\begin{align*}
  \iint_{\mathcal{O}}A \nabla u \cdot \nabla \varphi \, dX = 0, \quad \text{ for every } \varphi \in C_{\smallc}^\infty(\mathcal{O}).
\end{align*}

The most fundamental estimate for solutions to divergence form elliptic equations is the following local energy inequality.
\begin{lemma}[Caccioppoli Inequality]\label{lemma:caccioppoli}
  Let $L = -\div A \nabla$ be a divergence form elliptic operator and $u$ a solution to $Lu = 0$ in an open set $\mathcal{O}$. If $a> 0$ and $B$ is a ball such that $(1 + a)B \subset \mathcal{O}$  then
  \begin{align*}
    \iint_B |\nabla u| \, dX \lesssim r^{-2} \iint_{(1 + a)B} u^2 \, dX,
  \end{align*}
  where the implicit constant depends only on $a$, dimension and ellipticity.
\end{lemma}

Solutions to divergence form elliptic equations are locally H\"older continuous.

\begin{lemma}[H\"older continuity of solutions, \cite{DeG,nash}]\label{DGN.lem}
  Let $L = -\div A \nabla$ be a divergence form elliptic operator and $u$ a non-negative solution to $Lu = 0$ in an open set $\mathcal{O}$. Suppose that $B = B(X_0,R)$ is a ball such that  $\lambda B \coloneqq B(X_0,2\lambda R) \subset \mathcal{O}$ for $\lambda > 1$. Then we have
  \begin{align*}
    |u(X) - u(Y)| \le C \left(\frac{|X - Y|}{\lambda R}\right)^{\alpha} \left( \dashiint_{2\lambda B} |u|^2 \, dY \right)^{1/2} \quad \text{ for all } X,Y \in B,
  \end{align*}
  where $\alpha$ and $C$ depend only on dimension and ellipticity.
\end{lemma}

Another celebrated result is Moser's Harnack inequality for non-negative solutions.
  
\begin{lemma}[Harnack inequality \cite{Moser}]\label{lemma:harnack}
  Let $L = -\div A \nabla$ be a divergence form elliptic operator and $u$ a non-negative solution to $Lu = 0$ in an open set $\mathcal{O}$. If $B$ is a ball such that $2B \subset \mathcal{O}$ then $\sup_{B} u \le C \inf_{B} u$, where $C$ depends only on dimension and ellipticity.
\end{lemma}

We now turn our attention to the elliptic measure, for which we borrow the setting of \cite{AGMT}. Consider the compactified space $\overline{\mathbb R}^{n+1}=\mathbb R^{n+1}\cup\{\infty\}$; following \cite[Section 9]{hkm93}, we will understand all topological notions with respect to this space. Hence, for instance, if $\Omega$ is unbounded, then $\infty\in\partial\Omega$, and the functions in the space $C(\partial\Omega)$ are assumed continuous and real-valued, so that all functions in $C(\partial\Omega)$ lie in $L^{\infty}(\partial\Omega)$ even if $\partial\Omega$ is unbounded.

Given a domain $\Omega$ and a divergence form elliptic operator $L$, we let $\hm_{L,\Omega}^X$ denote the elliptic measure with pole at $X \in \Omega$. That is, by the Riesz representation theorem, for every $X \in \Omega$ there exists a probability measure $\omega_{L,\Omega}^X$ such that if $f \in C_{\smallc}(\pom)$, then the solution to $Lu = 0$ in $\Omega$ with $u \in C(\overline{\Omega})$ and $u = f$ on $\pom$, constructed via Perron's method, satisfies
\begin{equation}\label{formula:elliptic_measure}
  u(X) = u_f(X) = \int_{\pom} f(y) \, d\hm^{X}_{L,\Omega}(y).
\end{equation}
When the context is clear, we simply denote $\hm^X \coloneqq \omega^X_{L,\Omega}$ and, with slight abuse of terminology, call the family of elliptic measures $\omega = \omega_L = \omega_{L,\Omega} \coloneqq \{\omega^X\}_X$ just the elliptic measure. 

Our main results consider characterizations and implications given by quantitative absolute continuity of elliptic measure in the sense of Muckenhoupt's $A_\infty$ condition \cite{muckenhoupt,coifman-fefferman}:

\begin{definition}[$A_\infty$ for elliptic measure]\label{defin:a_infty}
Let $L$ be a divergence form elliptic operator in $\Omega$. We say that the   elliptic measure $\omega = \omega_{L,\Omega}$ satisfies the \emph{$A_\infty$ condition with respect to surface measure} (denote $\omega \in A_\infty(\sigma)$) if there exist constants $C \ge 1$ and $s > 0$ such that if $B \coloneqq B(x,r)$ with $x \in \pom$ and $r \in (0,\diam(\pom)/4)$ and $A \subset \Delta \coloneqq B \cap \pom$ is a Borel set, then
\begin{align*}
\omega^Y(A) \le C \left( \frac{\sigma(A)}{\sigma(\Delta)} \right)^s \omega^Y(\Delta),\qquad\text{for every }Y \in \Omega \setminus 4B.
\end{align*} 
We refer to $C$ and $s$ here together as the ``$\omega_L\in A_\infty(\sigma)$ constants''.
\end{definition}

Next we discuss the Green's function and its properties.

\begin{definition}[Green's function]\label{def.green} Let $L=-\dv A\nabla$ be a not necessarily symmetric divergence form elliptic operator with bounded measurable coefficients. There exists a unique non-negative function $G=G_L=G_{L,\Omega}:\Omega\times\Omega\ra\bb R$, called \emph{Green's function} for $L$, satisfying the following properties:
\begin{enumerate}[(i)]
\item For each $X,Y\in\Omega$, 
\begin{equation}\label{fungf1.eq}
0\leq G(X,Y)\lesssim\left\{\begin{matrix}|X-Y|^{1-n},&\qquad n\geq2,&\qquad X\neq Y,\\[1mm] 1,&\qquad n=1,&\qquad|X-Y|\geq\delta(Y)/10,\\[1mm] \log\big(\frac{\delta(Y)}{|X-Y|}\big),&\qquad n=1,&\qquad|X-Y|\leq\delta(Y)/2. \end{matrix}\right.
\end{equation}
\item For every $a \in (0,1)$ there exists $c_a$ such that
\begin{equation}\label{fungf2.eq}
G(X,Y)\geq\left\{\begin{matrix}c_a|X-Y|^{1-n},&\qquad n\geq2, &\qquad|X-Y| \le a\delta(Y),\\[1mm] c_a\log\big(\frac{\delta(Y)}{|X-Y|}\big),&\qquad n=1, &\qquad |X-Y|\leq a\delta(Y). \end{matrix}\right.
\end{equation}
\item For each $Y\in\Omega$, $G(\cdot,Y)\in C(\overline{\Omega}\backslash\{Y\})\cap W^{1,2}_{\loc}(\Omega\backslash\{Y\})$  and  $G(\cdot,Y)|_{\partial\Omega}\equiv0$.
\item For each $X\in\Omega$,   the identity $LG(\cdot,X)=\delta_X$ holds in the distributional sense; that is,
\[
\int_\Omega A(Y)\nabla_YG(Y,X)\nabla\Phi(Y)\,dY=\Phi(X),\qquad\text{for any }\Phi\in C_c^{\infty}(\Omega).
\]
\item For each $X,Y\in\Omega$ with $X\neq Y$, if $L^*=-\dv A^T\nabla$, then
\begin{equation}\label{eq.transpose}\nonumber
G_L(X,Y)=G_{L^*}(Y,X).
\end{equation} 
\end{enumerate}
\end{definition}

If $n\geq2$, then it has been known for a long time that a non-negative Green's function exists for any domain, without any further regularity assumptions on the geometry \cite{GW,HK}. If $n=1$, the situation has been more challenging; for instance, key Sobolev embeddings, available when $n\geq2$, fail when $n=1$, and the fundamental solution changes sign when $n=1$ \cite{kn85}; nevertheless, the paper \cite{dk09} shows the construction of a Green's function for any domain with either finite volume or finite width, and also, for the domain above a Lipschitz graph, improving on the result of \cite{dm95} (but non-negativity is not shown in these works). For our setting of uniform domains $\Omega\subset\bb R^{n+1}$, $n\geq1$, with $n$-Ahlfors regular boundary, the unified (for $n=1$ and $n\geq2$) existence of the non-negative Green's function for arbitrary divergence form elliptic operators $L$ of merely bounded measurable coefficients with the properties stated above follows from the much more general, recent construction in \cite[Theorem 14.60 and Lemma 14.78]{dfm20}. 

The Green's function and the elliptic measure are related through the following Riesz formula: For every $F \in C_{\smallc}^\infty(\ree)$, one has that
\begin{equation}\label{Rieszformfin.eq}
\int_{\pom} F(y) \, d\hm^X(y) = F(X) - \iint_{\Omega} A^T \nabla_Y G_L(X,Y) \cdot \nabla_Y F(Y) \, dY.
\end{equation} 

We need several estimates from the literature for the elliptic measure and Green's function in our proofs and we list these estimates below. Although these have appeared in several works in the literature \cite{CFMS, hkm93, AGMT}, we cite \cite{dfm20} for their unified consideration of the cases $n=1$ and $n\geq2$ and arbitrary elliptic operators on uniform domains with Ahlfors regular boundary. The first estimate is a non-degeneracy estimate for elliptic measure.

\begin{lemma}[Bourgain's estimate, {\cite[Lemma 15.1]{dfm20}}]\label{bour.lem}
  Let $x \in \pom$ and $r \in (0,\diam(\pom)]$ and let $Y_{x,r}$ be a corkscrew point relative to $x$ at scale $r$. We have 
  \begin{align*}
    \hm^{Y_{x,r}}(\Delta(x,r)) \ge c,
  \end{align*}
  where $c$ depends only on dimension, ellipticity, and the Ahlfors regularity constant. Here and below $\Delta(x,r) = B(x,r) \cap \pom$.
\end{lemma}
The elliptic measure is locally doubling in the following sense.
\begin{lemma}[(Local) doubling property, {\cite[Lemma 15.43]{dfm20}}]\label{lm.doubling}
  Let $x \in \pom$ and $r \in (0, \diam(\Omega)] $. If $Y \in \Omega \setminus B(x,4r)$ then $\hm^Y(\Delta(x,2r)) \le C\hm^Y(\Delta(x,r))$,  where $C$ depends on dimension, ellipticity, Harnack chain, corkscrew and Ahlfors regularity constants.
\end{lemma}

The following estimate allows us to connect the Green's function and elliptic measure in a quantitative way:
\begin{lemma}[CFMS estimate, {\cite[Lemma 15.28]{dfm20}}]\label{CFMSest.lem}
If $x \in \pom$ and $r \in (0, \diam(\pom)]$ then 
\begin{align*}
  \frac{G_L(X, Y_{x,r})}{r} \approx \frac{\hm^X_L(\Delta(x,r))}{r^n}
\end{align*}
for every $X \in \Omega \setminus B(x,2r)$, where the implicit constants depend on dimension, ellipticity, Harnack chain, corkscrew and Ahlfors regularity constants, and $Y_{x,r}$ is any corkscrew point relative to $x$ at scale $r$. 
\end{lemma}

Non-negative solutions $u$ to $Lu = 0$ that vanish on an open subset of the boundary of a uniform domain must do so at the same rate: 
\begin{lemma}[Boundary Harnack Principle, {\cite[Theorem 15.64]{dfm20}}]\label{bharn.lem}
Let $x \in \pom$ and $r \in (0, \diam(\pom)]$, and let $u$ and $v$ be positive functions such that $Lu = Lv = 0$ in $\Omega \cap B(x,4r)$ that vanish continuously on $\pom\cap B(x,4r)$. Then
\begin{align*}
\frac{u(X)}{v(X)} \approx \frac{u(Y_{x,r})}{v(Y_{x,r})}, \qquad \text{for all } X \in B(x,r) \cap \Omega,
\end{align*}
where $Y_{x,r}$ is a corkscrew point relative to $x$ at scale $r$. The implicit constants depend on dimension, ellipticity, Harnack chain, corkscrew and Ahlfors regularity constants.
\end{lemma}

We have the following standard consequence of the boundary Harnack Principle.

\begin{lemma}[Change of Poles, {\cite[Lemma 15.61]{dfm20}}]\label{lm.change} Let $x\in\partial\Omega$, $r\in(0,\diam(\partial\Omega))$, $Y_{x,r}$ a corkscrew point relative to $x$ at scale $r$, and $E\subset\Delta(x,r)$ a Borel set. Then
\[
\omega^{Y_{x,r}}(E)\approx\frac{\omega^X(E)}{\omega^X(\Delta(x,r))},\qquad\text{for each }X\in\Omega\backslash B(x,2r),
\]	
where the implicit constants depend on dimension, ellipticity, and structural constants of $\Omega$.
\end{lemma}

Finally, solutions $u$ to $Lu = 0$ that vanish continuously on the boundary do so at a H\"older rate:
\begin{lemma}[{\cite[Lemma 11.32]{dfm20}}]\label{Hvanish.lem}
  Let $x \in \pom$ and $r \in (0, \diam(\pom)]$, and let $u$ be a solution to $Lu= 0$ in $\Omega \cap B(x,4r)$ that vanishes continuously on $\pom\cap B(x,4r)$. We have the bound
  \begin{align*}
    |u(X)| \le C\left(\frac{\dist(X,\pom)}{r}\right)^\alpha \sup_{Y \in B(x,4r)} |u(Y)|
  \end{align*}
  for every $X \in B(x,r) \cap \Omega$, where $\alpha$ and $C$ depend on dimension, the Harnack chain, corkscrew and Ahlfors regularity constants.
\end{lemma}

\section{Set-up for Theorem \ref{theorem:approximability}}
\label{section:set-up_for_approximability}

In this section we provide some of the preliminary estimates and observations required to prove Theorem \ref{theorem:approximability}. We divide this section into two subsections. The first records Carleson measure estimates and non-tangential convergence in our setting. The second subsection contains a few lemmas, which are roughly adapted from ideas in \cite{DJK} and play a crucial role in our analysis. Throughout this section, we suppose that the assumptions of Theorem \ref{theorem:approximability} hold; that is, $\Omega$ is a uniform domain with Ahlfors regular boundary and $L$ is a not necessarily symmetric divergence form elliptic operator  such that $\omega_L\in A_\infty(\sigma)$. For the sequel, recall that $\delta(\cdot) \coloneqq \dist(\cdot,\pom)$.

\subsection{CME and non-tangential convergence}

One of the key tools in most of the constructions of $\eps$-approximators (see, for example, \cite{HMM}) is $L^2$-type Carleson measure estimates (CME), that is, Carleson properties (see Definition \ref{defin:carleson_measure_constant}) of measures $\mu_u$ such that $d\mu_u = |\nabla u|^2 \delta(Y) \, dY$ for a solution $u$ to $Lu = 0$. Under the hypotheses of Theorem \ref{theorem:approximability}, we  have the following 	``classical'' Carleson measure estimate and $L^p$-solvability of the Dirichlet problem for $L$ for some $p$ \cite{AHMT,CHMT,fp}:

\begin{lemma}[{\cite[Corollary 1.32]{fp}}]\label{realCME.lem}
Suppose that $\Omega$ is a uniform domain in $\bb R^{n+1}$, $n\geq1$, with $n$-Ahlfors regular boundary and $L$ is a divergence form elliptic operator such that $\omega_L \in A_\infty(\sigma)$. There exists a constant $C \ge 1$ such that if $u \in W^{1,2}_{\loc}(\Omega) \cap L^\infty(\Omega)$ is a solution to $Lu = 0$ then
  \begin{align*}
    \sup_{r > 0} \sup_{x \in \pom} \frac{1}{r^n} \iint_{B(x,r) \cap \Omega} |\nabla u(Y)|^2 \delta(Y) \, dY \lesssim C \|u\|_{L^\infty(\om)}^2.
  \end{align*}
  The constant $C$ depends on the structural constants and the $\omega_L\in A_\infty(\sigma)$ constants.
\end{lemma}

\begin{lemma}[{\cite[Theorem 1.3]{HofPLe}}]\label{l2solv.lem}
  Suppose that $\Omega \subset \ree$, $n\geq1$, is a uniform domain with Ahlfors regular boundary and $L$ is a divergence form elliptic operator in $\Omega$ such that $\omega_L \in A_\infty(\sigma)$. There exist $1 < p < \infty$ and $C \ge 1$ such that for every $f \in L^p(\pom)$ there exists a solution $u$ to the boundary value problem
  \begin{align*}
    \begin{cases}
      Lu = 0 & \text{ in } \Omega, \\
      \|Nu \|_{L^p(\pom)} \lesssim \|f\|_{L^p(\pom)}, \\
      \lim\limits_{Y \to x, \, \nt} u(Y) = f(x), & \text{ for $\sigma$-a.e.} x \in \pom.
    \end{cases}
  \end{align*}
  Moreover, the solution is of the form
  \begin{align*}
    u(X) = u_f(X) = \int_{\pom} f(y) \, d\hm^X(y).
  \end{align*}
\end{lemma}

Using the $L^p$ result of Lemma \ref{l2solv.lem} gives us the following non-tangential convergence result for $L^\infty$ data:
\begin{lemma}\label{lemma:solution_bounded_data}
  Suppose that $\Omega \subset \ree$, $n\geq1$, is a uniform domain with Ahlfors regular boundary and $L$ is a divergence form elliptic operator in $\Omega$ such that $\omega_L \in A_\infty(\sigma)$. If $f \in L^\infty(d\sigma)$, then the solution
  \begin{align*}
    u_f(X) = \int_{\pom} f(y) \, d\hm^X(y)
  \end{align*}
  converges non-tangentially to $f$; that is, 
  \begin{equation}
    \label{ntlimlem1eq1.eq}\lim\limits_{Y \to x, \, \nt} u_f(Y) = f(x), \quad  \text{ for $\sigma$-a.e. } x \in \pom.
  \end{equation}
  Equivalently,
  \begin{equation}
    \label{ntlimlem1eq2.eq} \lim\limits_{Y \to x, \, \nt} u_f(Y) = f(x), \quad \text{ for $\hm^Y$-a.e. } x \in \pom.
  \end{equation}
\end{lemma}

\begin{proof}
  Let $p$ be as in Lemma \ref{l2solv.lem}. Let us note that $\hm^Y$ and $\hm^{Y'}$ are mutually absolutely continuous for all $Y, Y' \in \Omega$ by the Harnack chain condition and the Harnack inequality (see Lemma \ref{lemma:harnack}). Moreover, $\hm^{Y}$ is mutually absolutely continuous with respect to $\sigma$ by the $A_\infty(\sigma)$ condition \cite[Lemma 5]{coifman-fefferman} and hence, \eqref{ntlimlem1eq1.eq} and \eqref{ntlimlem1eq2.eq} are equivalent. In addition, we have $f \in L^\infty(\pom,d\sigma)$ if and only if $f \in L^\infty(\pom,d\hm^Y)$ for all $Y \in \Omega$, with $\|f\|_{L^\infty(\pom,d\sigma)} = \|f\|_{L^\infty(\pom,d\hm^Y)}$.  
  
  If $\diam(\pom) < \infty$, then $f \in L^p(\pom)$ and the claim follows from Lemma \ref{l2solv.lem}. Thus, we may assume that $\diam(\pom) = \infty$. We show that the claim holds in any ball $B$ centered at $\pom$. Let us fix $x_0 \in \pom$ and $r > 0$. We write  $f = g + h$ for $g = f\mathbbm{1}_{\Delta(x_0,100r)}$, where $\Delta(x_0,100r) \coloneqq B(x_0,100r) \cap \pom$. By linearity, we have that $u_f(X) = u_g(X) + u_h(X)$.   Since $f \in L^\infty(\pom)$, we know that $g \in L^p(\pom)$ and thus, Lemma \ref{l2solv.lem} gives us
  \begin{align*}
    \lim\limits_{Y \to x, \, \nt} u_g(Y) = g(x) = f(x), \quad  \text{ for $\sigma$-a.e. } x \in \pom \cap B(x_0,r).
  \end{align*}
  Therefore it suffices to show 
  \begin{equation}\label{ntlem1hgoal.eq}
    \lim\limits_{Y \to x, \, \nt} u_h(Y) = 0, \quad  \text{ for $\sigma$-a.e. } x \in \pom \cap B(x_0,r).
  \end{equation}
  We now have 
  \begin{equation}\label{uhsimpbnd.eq}
    |u_h(X)| \le \|f\|_{L^\infty(\pom)} \hm^X(\pom \setminus \Delta(x_0,100r))
  \end{equation}
  for any $X \in \Omega$. Let $\phi \in C_{\smallc}(\Delta(x_0,100r))$ be a non-negative function such that $\phi(x) \le 1$ for every $x \in \pom$ and $\phi \equiv 1$ on $\Delta(x_0,50r)$. Then, since $\omega^X$ is a probability measure, we get
  \begin{align}
    \label{ineq:measure_of_complement} \hm^X(\pom \setminus \Delta(x_0,100r )) \le 1 - v(X) \coloneqq 1 - \int_{\pom} \phi(y) \, d\hm^X(y).
  \end{align}
  Since $\phi$ is a compactly supported continuous function on $\pom$, we know that $v$ is a continuous bounded solution to $Lv = 0$ in $\overline{\Omega}$ such that $v = \phi \equiv 1$ on $\Delta(x_0,50r)$. In particular, we have
  \begin{align}
    \label{convergence:upper_bound} \lim_{Y \to x, \, \nt} v(Y) = 1
  \end{align}
  for every $x \in \Delta(x_0,50r)$. Thus, combining \eqref{uhsimpbnd.eq} and \eqref{ineq:measure_of_complement} gives us $|u_h(X)| \le \|f\|_{L^\infty(\pom)}\left( 1 - v(X) \right)$ for every $X\in\Omega$   and \eqref{ntlem1hgoal.eq} follows then from \eqref{convergence:upper_bound}. This completes the proof.
\end{proof}

\subsection{A few important lemmas}\label{lemmas.sect}
In this subsection, we prove some lemmas that will be important for the proof of Theorem \ref{theorem:approximability}. The first three lemmas were inspired by the ideas in \cite{DJK}.

\begin{lemma}\label{prelimlem1.lem}
Let $\Omega$ be a domain in $\bb R^{n+1}$, $n\geq1$, and let $L$ be a not necessarily symmetric divergence form elliptic operator. Suppose that $u \in W^{1,2}_{\loc}(\Omega)$ is a weak solution to $Lu = 0$ in $\Omega$, that $\Omega'\subset\Omega$ is a Wiener-regular domain which is compactly contained  in $\Omega$, and fix $X_*\in\Omega'$. If $n=1$, assume in addition that $\Omega'$ is a uniform domain with $n$-Ahlfors regular boundary. Then  $|\nabla u|^2G_{L,\Omega'}(X_*,\cdot)\in L^1(\Omega')$, and 
\begin{equation}\label{eq.formula}
  \iint_{\Omega'} |\nabla u(Y)|^2 G_{L,\Omega'}(X_*, Y) \, dY \approx \int_{\partial \Omega'} (u(y) - u(X_*))^2 \, d\hm_{L, \Omega'}^{X_*}(y),
\end{equation}
where the implicit constants depend only on ellipticity.
\end{lemma}
\begin{proof} Throughout this argument we fix $X_*\in\Omega'$, we let $r=\dist(X_*,\partial\Omega')/8$, and we write $G(Y) \coloneqq G_{\Omega'}(X_*,Y)$. First, we show the finiteness of the integral in the left-hand side of (\ref{eq.formula}). If $n=1$, then it is trivial by (\ref{fungf1.eq}) and the Meyers reverse H\"older estimate for gradients of solutions \cite{mey63}, so suppose that $n\geq2$. We write
\begin{equation}\nonumber
\iint_{\Omega'} |\nabla u|^2 G \, dY=\iint_{\Omega'\backslash B(X_*,r)} |\nabla u|^2 G\, dY\\+\iint_{B(X_*,r)} |\nabla u|^2 G \, dY \eqqcolon T_1+T_2.
\end{equation}
Since $G\in L^{\infty}(\Omega'\backslash B_*(X_*,r))$ and $\nabla u\in L^2(\Omega')$, it is clear that $T_1<\infty$. As for $T_2$, for each $k\in\bb N_0$, let $A_k \coloneqq B(X_*,2^{-k}r)\backslash B(X_*,2^{-k-1}r)$, and using (\ref{fungf1.eq}), Lemma \ref{lemma:caccioppoli}, and Lemma \ref{DGN.lem}, we see that
\begin{multline}\nonumber
T_2=\sum_{k=0}^{\infty}\iint_{A_k}|\nabla u|^2 G\,dY \lesssim\sum_{k=0}^{\infty}(2^{-k}r)^{-n+1}\iint_{B(X_*,2^{-k}r)}|\nabla(u(Y)-u(X_*))|^2\,dY\\ \lesssim\sum_{k=0}^{\infty}(2^{-k}r)^{-n-1}\iint_{B(X_*,2^{-k+1}r)}|u(Y)-u(X_*)|^2\,dY \lesssim\sum_{k=0}^{\infty}2^{-2\alpha k}\Vert u\Vert_{L^{\infty}(B(X_*,4r))}<\infty.
\end{multline}
We turn to the proof of (\ref{eq.formula}). By the ellipticity of $A$ and the product rule, we see that
\begin{equation}\label{eq.for1}
\iint_{\Omega'}|\nabla u|^2 G\,dY\approx\iint_{\Omega'}A\nabla u\cdot(\nabla u) G\,dY\\ =\iint_{\Omega'}A\nabla u\cdot\nabla(uG)\,dY-\iint_{\Omega'}A\nabla u\cdot(\nabla G)u\,dY \eqqcolon T_3+T_4,
\end{equation} 
provided that the last two integrals are finite, which we now show. First, we claim that $|u||\nabla u||\nabla G|\in L^1(\Omega')$. As before, it is enough to show that $|u||\nabla u||\nabla G|\in L^1(B(X_*,r))$. Let $A_k$ as above, and for each $k\in\bb N_0$, let $\{B_k^j\}_{j=1}^{J_k}$ be a cover of $A_k$ by balls centered on $A_k$ of radius $2^{-k-4}r$ with uniformly bounded overlap. Then $J_k\lesssim1$, and we have that
\begin{equation}\nonumber
\iint_{B(X_*,r)}|u||\nabla u||\nabla G|\,dY \leq C\sum_{k=0}^{\infty}\max_j\iint_{B_k^j}|\nabla u||\nabla G|\,dY\leq\left\{\begin{matrix}C\sum_{k=0}^{\infty}2^{-k\alpha},&\quad n\geq2,\\[1mm] C\sum_{k=0}^{\infty}k2^{-k\alpha},&\quad n=1,\end{matrix}\right. 
\end{equation}
where the last estimate follows from  using  Lemma \ref{lemma:caccioppoli}  for both $u$ and $G$, then (\ref{fungf1.eq})  and Lemma \ref{DGN.lem}. This proves the claim. By the product rule and the triangle inequality, we have also shown that $|\nabla u||\nabla(uG)|\in L^1(\Omega')$. Finally, by boundedness of $A$ we conclude that  $A\nabla u\nabla(uG)$ and $A\nabla u(\nabla G)u$ belong to $L^1(\Omega')$, as desired.

The next step is to show that $T_3=0$. For each $M\in\bb N$ large enough, let $\psi_M\in C^{\infty}(\Omega)$ satisfy $\psi_M\geq0$, $\psi_M\equiv1$ in $\Omega\backslash B(X_*,\frac2M)$, $\psi_M\equiv0$ in $B(X_*,\frac1M)$, and $|\nabla\psi_M|\lesssim M$. We claim that
\begin{equation}\label{eq.for2}
\iint_{\Omega'}A\nabla u\nabla(uG)\psi_M\,dY\ra0,\qquad\text{as }M\ra\infty.
\end{equation}
Fix $M\in\bb N$ large enough. By the product rule and the fact that $uG\psi_M\in W_0^{1,2}(\Omega')$, since $Lu=0$ in $\Omega'$, we have that 
\begin{multline}\nonumber
\Big|\iint_{\Omega'}A\nabla u\nabla(uG)\psi_M\,dY\Big|=\Big|\iint_{\Omega'}A\nabla u(\nabla\psi_M)uG\,dY\Big|\\ \leq CM\iint_{B(X_*,\frac2M)\backslash B(X_*,\frac1M)}|\nabla u|G\,dY \leq\left\{\begin{matrix}CM^{-\alpha},&\quad n\geq2,\\[1mm] CM^{-\alpha}\log M,&\quad n=1,\end{matrix}\right. 
\end{multline}
where in the last estimate we once again used (\ref{fungf1.eq}), Lemma \ref{lemma:caccioppoli}, and Lemma \ref{DGN.lem}. Thus we have shown (\ref{eq.for2}). Since it is also true that $A\nabla u\nabla(uG)\psi_M\ra A\nabla u\nabla(uG)$ pointwise a.e. in $\Omega'$ and since we have already proved that $|\nabla u||\nabla(uG)|\in L^1(\Omega')$, then by the Lebesgue Dominated Convergence Theorem we conclude that $T_3=0$.

We proceed with the proof of (\ref{eq.formula}) as follows:
\begin{equation}\nonumber
  T_4=-\frac{1}{2} \iint_{\Omega'} A^T\nabla G\nabla (u^2)\, dY   = \frac{1}{2}  \Big(\int_{\partial \Omega'} u(y)^2 \, d\hm_{L, \Omega'}^{X_*}(y)-u(X_*)^2\Big) \\
   = \frac{1}{2} \int_{\partial \Omega'} (u(y) - u(X_*))^2 \, d\hm_{L, \Omega'}^{X_*}(y),
\end{equation}
where we used the Riesz formula \eqref{Rieszformfin.eq}, and in the last identity we used that $\hm_{L, \Omega'}^{X_*}$ is a probability measure, and that $Lu= 0$ in $\Omega'$ so that
\begin{align*}
  \int_{\partial \Omega'} u(y) \, d\hm_{L, \Omega'}^{X_*}(y) = u(X_*),
\end{align*}
and hence 
\begin{align*}
  \int_{\pom'} u(X_*)u(y) \, d\hm_{L, \Omega'}^{X_*}(y) = u(X_*) \int_{\partial \Omega'} u(y) \, d\hm_{L, \Omega'}^{X_*}(y)  = u(X_*)^2.
\end{align*}
This finishes the proof.
\end{proof}

The following result is a direct consequence of the maximum principle and the DeGiorgi--Nash--Moser theory (see for instance the proof of \cite[Proposition 3.2]{adfjm}). 
\begin{lemma}\label{lem5.lem}
  Let $\Omega_1$ and $\Omega_2$ be Wiener regular domains such that $\Omega_1 \subseteq \Omega_2$. If $G_i(X,Y)$ is the Green function for $\Omega_i$, $i = 1,2$  then
  \begin{align*}
    G_1(X,Y) \le G_2(X,Y) \quad \text{ for every } (X, Y) \in \Omega_1 \times \Omega_1  \setminus \{X = Y\}.
  \end{align*}
\end{lemma}

We will also use a corona-type decomposition of the elliptic measure $\omega_L$ \cite{HLMN,GMT,AGMT}. To formulate this decomposition, we recall the coherency and semicoherency of subcollections of $\dd$:

\begin{definition}\label{defin:coherency}
  We say that a subcollection $\Sc \subset \dd$ is \emph{coherent} if the following three conditions hold.
  \begin{enumerate}
    \item[(a)] There exists a maximal element $Q(\Sc) \in \Sc$ such that $Q \subset Q(\Sc)$ for every $Q \in \Sc$.
    \item[(b)] If $Q \in \Sc$ and $P \in \dd$ is a cube such that $Q \subset P \subset Q(\Sc)$, then also $P \in \Sc$.
    \item[(c)] If $Q \in \Sc$, then either all children of $Q$ belong to $\Sc$ or none of them do.
  \end{enumerate}
  If $\Sc$ satisfies only conditions (a) and (b), then we say that $\Sc$ is \emph{semicoherent}.
\end{definition}

We are ready to present the corona decomposition that we shall use in the sequel.   
\begin{lemma}[{\cite[Proposition 3.1]{AGMT}}]\label{Coronaellipmeasure.lem}
  Suppose that $\Omega$ is a uniform domain and $L = -\div A \nabla$ is a divergence form elliptic operator such that $\omega_L \in A_\infty(\sigma)$. Then there exist constants, $C, M > 1$ and a decomposition of the dyadic system $\mathbb{D} = \mathbb{D}(\partial \Omega)$, with the following properties. 
  \begin{enumerate}[(i)]
    \item The dyadic grid breaks into a disjoint decomposition $\mathbb{D} = \mathcal{G} \cup \mathcal{B}$, the good cubes and bad cubes respectively. 
    
    \item The family $\mathcal{G}$ has a disjoint decomposition $\mathcal{G} = \cup \Sc$ where each $\Sc$ is a coherent   stopping time regime with maximal cube $Q(\Sc)$.

    \item The maximal cubes and bad cubes satisfy a Carleson packing condition:
    \begin{align*}
      \sum_{\substack{Q \in \mathcal{B} \\ Q \subseteq R }} \sigma(Q) + \sum_{\Sc \colon Q(\Sc) \subseteq R} \sigma(Q(\Sc)) \le C \sigma(R) \quad \text{ for every } R \in \dd.
    \end{align*}

    \item On each stopping time $\Sc$, the elliptic measure `acts like surface measure' in the sense that if $Q \in \Sc$, then
    \begin{align*}
      M^{-1} \frac{\sigma(Q)}{\sigma(Q(\Sc))} \le  \frac{\omega^{X_{Q(\Sc)}}(Q)}{\omega^{X_{Q(\Sc)}}(Q(\Sc))}  \le M \frac{\sigma(Q)}{\sigma(Q(\Sc))}.
    \end{align*}
  \end{enumerate}
\end{lemma}

\begin{proof}
  Since $\omega_L\in A_\infty(\sigma)$, then by Lemma \ref{realCME.lem} we have that the hypothesis (b) of \cite[Proposition 3.1]{AGMT} is satisfied. By \cite[Proposition 3.1]{AGMT}, this yields a decomposition $\bb D=\mathcal G\cup\mathcal B$ like the one described above (in fact we have that $\mathcal B=\varnothing$), but with the caveat that the stopping time regimes need only be semicoherent. Then, by a standard mechanism (see for instance \cite[pp.56--57]{DS2}, and \cite[Remark 2.13]{chm22}), we can modify the stopping time regimes so that they are coherent, while the rest of the properties are still satisfied.
\end{proof}

It is straightforward to check that for each semicoherent stopping time regime $\Sc$, there exists a collection of pairwise disjoint cubes $\F_\Sc$ such that $\Sc = \mathbb{D}_{\F_\Sc, Q(\Sc)}$. By Lemma \ref{CFMSest.lem}   and the Harnack inequality, we get the following:
\begin{lemma}\label{GFdeltacomp.lem}
  Suppose that $\Omega$ is a uniform domain in $\bb R^{n+1}$, $n\geq1$, with $n$-Ahlfors regular boundary, that $L = -\div A \nabla$ is a divergence form elliptic operator such that $\omega_L \in A_\infty(\sigma)$, and that $\Sc$ is a semicoherent stopping time regime satisfying the property in Lemma \ref{Coronaellipmeasure.lem} (iv). Then 
  \begin{align*}
    G_L(X_Q, Y)  \sigma(Q) \approx \delta(Y) \qquad \text{ for every } Q \in \Sc \text{ and } Y \in \Omega^*_{\F_\Sc, Q},
  \end{align*}
where the implicit constants depend only on dimension, ellipticity,  $\kappa$, Harnack chain, corkscrew, and Ahlfors regularity constants, as well as the constant $M$ in Lemma \ref{Coronaellipmeasure.lem} (iv).
\end{lemma}

\begin{proof} Fix $Q\in\Sc$ and\footnote{For relevant notation in this proof, see Section \ref{section:dyadic_and_whitney}.} $Y\in\Omega_{\F_\Sc,Q}^*$. By definition, there exists $P\in\bb D_{\F_S,Q}$ so that $Y\in U_P^*$, and thus $P\in\Sc$, $P\subseteq Q$, and $\delta(Y)\approx\ell(P)$. Let $B_P' \coloneqq B(x_P,\eta a_0\ell(P))$ with $\eta\in(0,1)$ small, and let $X_P'$ be the corkscrew point for $B_P'$. Define $B_Q'$ and $X_Q'$ analogously. We may guarantee that $X_Q\in\Omega\backslash2B_P'$   if $\eta$ is chosen small enough depending only on the corkscrew constant. Then, since  $\delta(Y)\approx\delta(X_P')$, and  $|Y-X_P'|\lesssim\delta(Y)$, by the Harnack inequality (for $L$ and $L^*$) and Harnack chains, Lemma \ref{CFMSest.lem}, the doubling property of   elliptic measure, and Ahlfors regularity, we have that
\begin{equation}\nonumber
\frac{G_L(X_Q,Y)}{\delta(Y)}\approx\frac{G_L(X_Q,X_P')}{\ell(P)}\approx\frac{\omega^{X_Q}(2B_P')}{\ell(P)^n} \approx \frac{\omega^{X_Q'}(B_P')}{\sigma(P)}.
\end{equation} 
Now, since $X_{Q(\Sc)}\in\Omega\backslash2B_Q'$ and $P,Q\in\Sc$, then by Lemma \ref{lm.change}, the doubling property of elliptic measure, Harnack chains, Harnack inequality, Bourgain's estimate, and the property (iv) in Lemma \ref{Coronaellipmeasure.lem}, it follows that
\begin{equation}\nonumber
\frac{\omega^{X_Q'}(B_P)}{\sigma(P)}\approx\frac{\omega^{X_{Q(\Sc)}}(B_P)}{\sigma(P)}\frac{1}{\omega^{X_{Q(\Sc)}}(B_Q')}\approx\frac{\omega^{X_{Q(\Sc)}}(P)}{\sigma(P)}\frac{1}{\omega^{X_{Q(\Sc)}}(Q)}\\ \approx\frac{1}{\sigma(Q(\Sc))}\frac{\sigma(Q(\Sc))}{\sigma(Q)}=\frac1{\sigma(Q)},
\end{equation}
which completes the proof.
\end{proof}

\section{Existence of $\eps$-approximators: proof of Theorem \ref{theorem:approximability}}
\label{section:proof_of_approximability}

In this section we prove Theorem \ref{theorem:approximability}. The key step in the proof is the construction of $\BV_{\loc}$ approximators since the existence of smooth approximators with more delicate properties follows then by using almost black box regularization arguments. More precisely, the heart of Theorem \ref{theorem:approximability} is the following result:

\begin{theorem}\label{theorem:existence_approximators}
	Let $\Omega \subset \ree$, $n \ge 1$, be a uniform domain with Ahflors regular boundary. Let $L = -\div A\nabla$ be a (not necessarily symmetric) divergence form elliptic operator satisfying that $\hm_L \in A_\infty(\sigma)$. Then, for any $\eps \in (0,1)$ there exists a constant $C_\eps$ such that if $u \in W^{1,2}_{\loc}(\Omega) \cap L^\infty(\Omega)$ is a solution to $Lu = 0$ in $\Omega$, then there exists   $\Phi = \Phi^{\eps} \in \BV_{\loc}(\Omega)$ satisfying
	\begin{align*}
		\|u - \Phi\|_{L^\infty(\Omega)} \le \eps \|u\|_{L^\infty(\Omega)} \quad \text{ and } \quad \sup_{x \in \pom, r > 0} \frac{1}{r^n} \iint_{B(x,r) \cap \Omega} |\nabla \Phi(Y)|\, dY \le C_\eps \|u\|_{L^{\infty}(\Omega)},
	\end{align*}
	where $C_\eps$ depends on $\eps$, structural constants, ellipticity and the $\omega_L\in A_{\infty}(\sigma)$ constants.
\end{theorem}
 
 As mentioned in the introduction (see the paragraph after Theorem \ref{theorem:approximability}), the proof of Theorem \ref{theorem:existence_approximators} is based on the construction of $\eps$-approximators in \cite[Theorem 1.3]{HMM}, with important differences. We will point out when we have to diverge from the strategy in \cite{HMM}.

\subsection{Set-up for the proof of Theorem \ref{theorem:existence_approximators}}
We start by making some preliminary considerations and fixing some notation and terminology. Let $u$ be a bounded solution to $Lu = 0$, and let $\eps \in (0,1)$. Without loss of generality, we may assume that $\|u\|_{L^\infty(\Omega)} = 1$, since otherwise we have $u = 0$ or we can replace $u$ with $u/\|u\|_{L^\infty}$. We fix a stopping time regime $\Sc$ from Lemma \ref{Coronaellipmeasure.lem}, with maximal cube $Q(\Sc)$. Following \cite{Gar-BAF,HMM}, we label each cube $Q \in \dd$ depending on how much the the function $u$ oscillates in the corresponding fattened Whitney region $U_Q^*$. To be more precise, we say that  
\begin{align*}
Q\in\bb D \text{ is {\bf red} if }\quad	\sup_{X,Y\in U_Q^*} |u(X) - u(Y)| \ge \frac{\eps}{1000},
\end{align*}
and 
\begin{align*}
Q\in\bb D \text{ is {\bf blue} if }\quad	\sup_{X,Y\in U_Q^*} |u(X) - u(Y)| < \frac{\eps}{1000}.
\end{align*}
We note that these conditions differ from the, ones in \cite{HMM} in two ways: in \cite{HMM}, the oscillation threshold was $\eps/10$ and the level of oscillation was measured in the unfattened region $U_Q$. With our conditions we have a little bit more control on the blue cubes, which will be useful for us later. In \cite{HMM}, it was enough to use these two labels but for our analysis it is important to take into consideration the minimal cubes of $\Sc$, which we label as follows:
\begin{align*}
Q\in\Sc \text{ is {\bf yellow} if }\quad	Q \text{ has a child $Q'$ such that $Q' \not \in \Sc$}.
\end{align*}
Recall that $Q'$ is a child of $Q$ if $Q' \subset Q$ and $\ell(Q') =   \ell(Q)/2$. Notice that yellow cubes are not a separate collection from the red and blue cubes but each yellow cube is also red or blue. We denote the collections of red and yellow cubes by
\begin{equation*}
	\mathcal{R} \coloneqq \{Q \in \dd \colon Q \text{ is red}\},\qquad
	\mathcal{Y} = \mathcal{Y}(\Sc) \coloneqq \{Q \in \Sc \colon Q \text{ is yellow}\}.
\end{equation*}
The rough idea of the construction of the $\eps$-approximator $\Phi$ of $u$ is to first construct $\Phi$ inside a Carleson box $T_{Q_0}$ for a fixed cube $Q_0$ and then use a ``local to global''-type argument to define a global approximator. Working in a Carleson box allows us to reduce many of the challenging estimates to working with just one stopping time regime $\Sc$ given by Lemma \ref{Coronaellipmeasure.lem}. We then set $\Phi \equiv u$ in the regions $U_Q$ such that $Q \in \Rc \cup \Yc$ and break up the blue cubes into smaller stopping time regimes where $u$ does not vary by more than $\eps/100$. In these new regimes, we set $\Phi = u(X_0)$ in the union of $U_Q$, where $X_0$ is any point in the union. The $L^\infty$ approximation property follows then just from the way we defined $\Phi$ but verifying the $L^1$-type Carleson measure estimate for $\Phi$ is more challenging. For this, one of the key steps is to show that the collections of red, yellow and maximal cubes from the new stopping time regimes satisfy Carleson packing conditions. For the first two collections, this follows in a straightforward way:

\begin{lemma}\label{lemma:packing_of_red_and_yellow}
	The collections $\Rc$ and $\bigcup_{\Sc} \Yc(\Sc)$ satisfy the following Carleson packing conditions: for any $P \in \dd$ we have
	\begin{align*}
		\sum_{Q \in \Rc, Q \subset P} \sigma(Q) \lesssim \frac{1}{\eps^2} \sigma(P)
	\end{align*}
	and
	\begin{align*}
		\sum_{\Sc} \sum_{Q \in \Yc(\Sc), Q \subset P} \sigma(Q) \le (C+1)\sigma(P).
	\end{align*}
	The constant $C$ is the same constant as $C$ in Lemma \ref{Coronaellipmeasure.lem}.
\end{lemma}

\begin{proof}
	The packing condition for $\bigcup_{\Sc} \Yc(\Sc)$ follows from the definition and the facts that the stopping time regimes are coherent and their maximal cubes satisfy a Carleson packing condition. Indeed, each yellow cube $Q \in \Yc(\Sc)$ is contained in $Q(\Sc)$ and by the coherency of $\Sc$, no yellow cube can contain a smaller yellow cube $\widetilde Q\in \Yc(\Sc)$. Thus, the cubes in $\Yc(\Sc)$ are disjoint and we get
	\begin{equation*}
		\sum_{\Sc} \sum_{Q \in \Yc(\Sc), Q \subset P} \sigma(Q)
		\le \sum_{\substack{\Sc: Q \in \Yc(\Sc), Q \subset P, \\ Q(\Sc) \subset P}} \sigma(Q)
		+ \sum_{\substack{\Sc: Q \in \Yc(\Sc), Q \subset P, \\ P \in \Sc}} \sigma(Q) \\
		\le \sum_{\substack{\Sc: Q(\Sc) \subset P}} \sigma(Q(\Sc))
		+ \sum_{\substack{\Sc: P \in \Sc}} \sigma(P).
	\end{equation*}
	Since the collection $\{Q(\Sc)\}_{\Sc}$ satisfies a Carleson packing condition by Lemma \ref{Coronaellipmeasure.lem}, we know that the first sum on the right-hand side is bounded by $C\sigma(P)$. In addition, since the stopping time regimes $\sbf$ are disjoint, we have $P \in \Sc$ for at most one stopping time regime $\Sc$. Thus, the second sum on the right-hand side is bounded by $\sigma(P)$. The desired bound follows from combining these two estimates.
	
	Let us then prove the Carleson packing condition for $\Rc$. If $Q \in \Rc$, then there exist points $X_1, X_2 \in U_Q^*$ such that $|u(X_1) - u(X_2)| > \eps$. 
	Since $X_1, X_2 \in U_Q^*$, there exist Whitney cubes $I_1, I_2 \subset U_Q^{*}$ such that $X_1 \in (1+\tau)I_1$, $X_2 \in (1+\tau)I_2$ and $\ell(Q) \approx \ell(I_1) \approx \ell(I_2) \approx \delta(X_1) \approx \delta(X_2)$. In particular, we have $\dist(X_1,\partial U_Q^{**}) \approx \dist(X_2,\partial U_Q^{**}) \approx \ell(Q)$ for the twice-fattened region $U_Q^{**}$ (see Section \ref{section:dyadic_and_whitney}), where the implicit constants depend on the dilation parameter $\tau$. Thus, since $U_Q^{**}$ satisfies the Harnack chain condition by Lemma \ref{lemma:hm}, there exists a uniformly bounded number of balls $B_1, B_2, \ldots, B_N$ and points $Y_1, Y_2, \ldots, Y_{N-1}$ such that
	\begin{enumerate}
		\item[$\bullet$] $X_1 \in B_1$, $X_2 \in B_N$ and $Y_i \in B_{i} \cap B_{i+1}$ for every $i = 1,2,\ldots,N-1$,
		
		\item[$\bullet$] for a constant $\lambda > 1$ (depending on $\tau$), we have $2\lambda B_i \subset U_Q^{**}$ for every $i = 1,2,\ldots,N$, and
		
		\item[$\bullet$] $|2\lambda B_i| \approx |U_Q^{**}| \approx \ell(Q)^{n+1}$ for every $i = 1,2,\ldots,N$.
	\end{enumerate}
	These properties combined with the triangle inequality, the local H\"older continuity (that is, Lemma \ref{DGN.lem}) and the Poincar\'e inequality applied for solution $v$, $v(X) \coloneqq u(X) - \fint_{U_Q^{**}} u(Z) \, dZ$, then give us
	\begin{multline*}
		\eps 
		\le |v(X_1) - v(X_2)| 		\le |v(X_1) - v(Y_1)| + \sum_{i=1}^{N-2} |v(Y_i) - v(Y_{i+1})| + |v(Y_N) - v(X_2)| \\
		\lesssim \dashiint_{2\lambda B_1} |v| \, dX + \sum_{i=1}^{N-2} \dashiint_{2\lambda B_{i+1}} |v| \, dX + \dashiint_{2\lambda B_N} |v| \, dX\lesssim \dashiint_{U_Q^{**}} |v| \, dX 
		\lesssim \ell(Q) \Big( \dashiint_{U_Q^{**}} |\nabla u|^2 \, dX \Big)^{1/2}.
	\end{multline*}
	Thus, since $\delta(X) \approx \ell(Q)$ for every $X \in U_Q^{**}$, we have
	\begin{equation*}
		\eps^2 \sigma(Q)
		\approx \ell(Q)^n \eps^2 \lesssim \ell(Q)^n \ell(Q)^2 \dashiint_{U_Q^{**}} |\nabla u(X)|^2 \, dX \\
		\lesssim \iint_{U_Q^{**}} |\nabla u(X)|^2 \ell(Q) \, dX 
		\approx \iint_{U_Q^{**}} |\nabla u(X)|^2 \delta(X) \, dX.
	\end{equation*}
	By construction, we know that the regions $U_P^{**}$ have bounded overlaps and any twice-fattened Carleson box $T_P^{**}$ satisfies $T_P^{**} \subset B(x_P,R_P) \cap \Omega$ the center $x_P$ of $P$ and for a radius $R_P \approx \ell(P)$. Hence, for any $Q_0 \in \dd$, these facts, the previous estimate and the Carleson measure estimate \eqref{realCME.lem} give us
	\begin{equation*}
		\sum_{\substack{P \in \mathcal{R} \\ P \subset Q_0}} \sigma(P)
		\lesssim \eps^{-2}\sum_{\substack{P \in \mathcal{R} \\ P \subset Q_0}}\iint_{U_P^{**}} |\nabla u(X)|^2 \delta(X) \, dX \\
		\lesssim \eps^{-2} \iint_{T^{**}_{Q_0}} |\nabla u(X)|^2 \delta(X) \, dX \lesssim \eps^{-2}\ell(Q_0)^n \approx \eps^{-2}\sigma(Q_0),
	\end{equation*}
	which proves the claim.
\end{proof}
 
\subsection{A stopping time decomposition for the family of blue cubes} We now move to decomposing the collection of blue cubes into more manageable subcollections using a stopping time procedure. A similar idea is utilized  in \cite[p. 2360]{HMM} but, due to our geometry, we use different stopping time conditions and different analysis of the subcollections. Set $\Lc = \Lc(\Sc)$ to be the collection of blue cubes in $\Sc$. We first take the largest blue cube in $\Sc$ with respect to side length (if there is more than one such cube, we just pick one) and denote this cube by $Q(\sbf_1)$. The cube $Q(\sbf_1)$ will be the maximal cube in our first refined stopping time regime $\sbf_1$. We let $\F_{\sbf_1}$ be the collection of cubes $Q \in \Sc \cap \dd_{Q(\sbf_1)} \setminus \{Q(\sbf_1)\}$ such that $Q$ is a maximal cube with respect to having one of the following three properties:
\begin{enumerate}
	\item[(1)] $Q$ or one of its siblings\footnote{Here $Q'$ is a sibling of $Q$ if $Q, Q' \in \dd_k$ and $Q, Q' \subset Q^* \in \dd_{k-1}$, that is, $Q$ and $Q'$ have the same parent. For simplicity, we consider $Q$ to be its own sibling.} is red.
	
	\item[(2)] $Q$ and all of its siblings are blue, but for some $Q'$ that is either $Q$ or a sibling of $Q$ it holds that 
	\[
	|u(X_{Q'}) - u(X_{Q(\sbf_1)})| > \eps/100.
	\] 
	
	\item[(3)] $Q$ is yellow.
\end{enumerate}
Recall that for every $P \in \dd$, the point $X_P$ is a corkscrew point relative to the center $x_P$ at scale $10^{-5} a_0 \ell(P)$, as defined in Section \ref{section:dyadic_and_whitney}. We now set $\sbf_1 \coloneqq \dd_{\F_{\sbf_1}, Q(\sbf_1)}$. By  construction, $\sbf_1$ is a coherent stopping time regime in the sense of Definition \ref{defin:coherency}. Since the cubes in $\sbf_1$ are blue and none of them satisfy the stopping condition (2) above, we know that
\begin{equation}\label{closeinstoptime.eq}
	|u(X) - u(X_{Q(\sbf_1)})| \le \eps/50 \quad \text{ for every } X \in \Omega_{\F_{\sbf_1}, Q(\sbf_1)}^*.
\end{equation}
We now express $\F_{\sbf_1}$ as a union of three collections, 
\begin{align*}
	\F_{\sbf_1} = \F_{\sbf_1}^{\bigR} \cup \F_{\sbf_1}^{\bigSB} \cup \F_{\sbf_1}^{\bigY},
\end{align*}
where $\F_{\sbf_1}^{\bigR}$ contains the cubes for which (1) holds, $\F_{\sbf_1}^{\bigSB}$ contains the cubes for which (2) holds and $\F_{\sbf_1}^{\bigY}$ contains the cubes for which (3) holds. The superscripts stand for ``red'', ``stopping blue'' and ``yellow''. We note that the collections $\F_{\sbf_1}^{\bigR}$ and $\F_{\sbf_1}^{\bigSB}$ are disjoint but $\F_{\sbf_1}^{\bigY}$ may overlap with both of them. 

We now continue this way: we let $Q(\sbf_2)$ be the largest blue cube in $\Sc \setminus \sbf_1$ with respect to side length, we extract the collection of maximal stopping cubes $\F_{Q(\sbf_2)}$ (with an updated stopping condition (2)), we define the coherent stopping regime $\sbf_2$ and the collections $\F_{\sbf_2}^{\bigR}$, $\F_{\sbf_2}^{\bigSB}$ and $\F_{\sbf_2}^{\bigY}$, choose the largest blue cube $Q(\sbf_3)$ in $\Sc \setminus \sbf_1 \cup \sbf_2$, and so on. Since each $\sbf_i$ contains at least the cube $Q(\sbf_i)$, we know that this procedure exhausts $\Lc$ and gives us a disjoint decomposition $\Lc = \cup_j \sbf_j$ where each $\sbf_j$ is a coherent stopping time regime. Just like \eqref{closeinstoptime.eq}, we have the oscillation estimate
\begin{equation}\label{estimate:oscillation_stopping_j}
	|u(X) - u(X_{Q(\sbf_j)})| \le \eps/50 \quad \text{ for every } X \in \Omega_{\F_{\sbf_j}, Q(\sbf_j)}^*
\end{equation}
for every $j$. We also get the collections $\F_{\sbf_j}^{\bigR}$, $\F_{\sbf_j}^{\bigSB}$ and $\F_{\sbf_j}^{\bigY}$ for each $j$.

Our next goal is to show that the maximal cubes $\{Q(\sbf_j)\}_j$ satisfy a Carleson packing condition. This goal is an analog of \cite[Lemma 5.16]{HMM} but since the proof of this lemma is based on the use of an ``$N \lesssim S$'' estimate in sawtooth regions (which is possible in the presence of uniform rectifiability of $\pom)$, this is the part where we significantly depart from \cite{HMM}. Following an idea in \cite{DS1}, we let $\lambda \in (0,10^{-10})$  be a small parameter (to be chosen) and break the stopping times into four groups. We say that $\sbf_j$ is of
\begin{enumerate}
	\item[$\bullet$] \emph{Type 1 (T1)} if  $\sigma(Q(\sbf_j) \setminus \cup_{Q \in \F_{\sbf_j}} Q) \ge \lambda \sigma(Q(\sbf_j))$.\\
	
	\item[$\bullet$] \emph{Type 2 (T2)} if $\sigma(\cup_{Q \in \F_{\sbf_j}^{\bigR}} Q) \ge \lambda \sigma(Q(\sbf_j))$.\\
	
	\item[$\bullet$] \emph{Type 3 (T3)} if $\sigma(\cup_{Q \in \F_{\sbf_j}^{\bigY}} Q) \ge \lambda \sigma(Q(\sbf_j))$.\\
	
	\item[$\bullet$] \emph{Type 4 (T4)} if $\sbf_j$ is not type 1, 2 or 3. 
\end{enumerate}
The Carleson packing condition for the cubes $Q(\sbf_j)$ follows in a straightforward way when $\sbf_j$ is of Type 1, 2 or 3:

\begin{lemma}\label{lemma:packing_for_types_1-3}
	We have the following Carleson packing conditions for the maximal cubes of the subregimes $\sbf_j$ in the decomposition $\Lc(\Sc) = \cup_j \sbf_j$: for any $P \in \dd$, we have
	\begin{align*}
		\sum_{\substack{j \colon \sbf_j \text{ is T1} \\ Q(\sbf_j) \subset P}} \sigma(Q(\sbf_j)) + \sum_{\substack{j \colon \sbf_j \text{ is T3} \\ Q(\sbf_j) \subset P}} \sigma(Q(\sbf_j)) &\lesssim	 \frac{1}{\lambda} \sigma(P),\qquad\text{and,}\\
		\sum_{\substack{j \colon \sbf_j \text{ is T2} \\ Q(\sbf_j) \subset P}} \sigma(Q(\sbf_j)) &\lesssim \frac{1}{\lambda\eps^2} \sigma(P).
	\end{align*}
\end{lemma}

\begin{proof}
	For the regimes of Type 1, we first notice that if $Q(\sbf_i) \subsetneq Q(\sbf_j)$, then $Q(\sbf_i) \subset Q$ for some $Q \in \F_{\sbf_j}$. In particular, the sets $Q(\sbf_j) \setminus \cup_{Q \in \F_{\sbf_j}} Q$ are pairwise disjoint. Thus, by the definition of Type 1, we get
	\begin{align*}
		\sum_{\substack{j \colon \sbf_j \text{ is T1} \\ Q(\sbf_j) \subset P}} \sigma(Q(\sbf_j))
		\le \frac{1}{\lambda} \sum_{\substack{j \colon \sbf_j \text{ is T1} \\ Q(\sbf_j) \subset P}} \sigma(Q(\sbf_j) \setminus \cup_{Q \in \F_{\sbf_j}} Q)
		\le \frac{1}{\lambda} \sigma(P).
	\end{align*}

	For Type 3 regimes, since the cubes $Q \in \F_{\sbf_j}^{\bigY} \subset \Yc(\Sc)$ are yellow cubes in $\Sc$, they are disjoint. Thus, the claim for the regimes of Type 3 follows immediately from definition.
	
	For the regimes of Type 2, we recall that if $Q \in \F_{\sbf_j}^{\bigR}$, then $Q$ is red or one of its siblings is red. In particular, each $Q \in \F_{\sbf_j}^{\bigR}$ has approximately the same measure as some red sibling $R_Q \subset Q(\sbf_j)$ of $Q$. If there is more than one red sibling, we just choose one of them for each $Q \in \F_{\sbf_j}^{\bigR}$. On the other hand, since each cube has only a uniformly bounded number of siblings, for each $Q \in \bigcup_j \F_{\sbf_j}^{\bigR}$ we can have $R_Q = R_{Q'}$ only for a uniformly bounded number of cubes $Q'$.
	 Thus, we get
	\begin{equation*}
		\sum_{\substack{j \colon \sbf_j \text{ is T2} \\ Q(\sbf_j) \subset P}} \sigma(Q(\sbf_j))
		\le \frac{1}{\lambda} \sum_{\substack{j \colon \sbf_j \text{ is T2} \\ Q(\sbf_j) \subset P}} \sum_{Q \in \F_{\sbf_j}^{\bigR}} \sigma(Q)
		\approx \frac{1}{\lambda} \sum_{\substack{j \colon \sbf_j \text{ is T2} \\ Q(\sbf_j) \subset P}} \sum_{Q \in \F_{\sbf_j}^{\bigR}} \sigma(R_Q)\\
		\lesssim \frac{1}{\lambda} \sum_{R \in \Rc, R \subset P} \sigma(R) 
		\lesssim \frac{1}{\lambda\eps^2} \sigma(P),
	\end{equation*}
	where we used Lemma \ref{lemma:packing_of_red_and_yellow} in the final estimate.
\end{proof}

By Lemma \ref{lemma:packing_for_types_1-3}, to show the Carleson packing condition for the collection $\{Q(\sbf_j)\}_j$ it remains only to consider the regimes $\sbf_j$ of Type 4.
\begin{lemma}\label{lemma:packing_for_type_4} There exists $\lambda_0 > 0$ depending only on structural constants, ellipticity, and the $\omega_{L,\Omega}\in A_\infty(\sigma)$ constants, such that for any $\lambda\in(0,\lambda_0)$, there exist constants $C_1, C_2\geq1$ depending only on structural constants, ellipticity, and the $\omega_{L,\Omega}\in A_\infty(\sigma)$ constants (and independent of $\eps$, $\lambda$, and $\Sc$), so that
\begin{align*}
\sum_{\substack{j \colon \sbf_j \text{ is T4} \\ Q(\sbf_j) \subset P}} \sigma(Q(\sbf_j)) &\le\frac{C_1}{\eps^2}\sigma(P).
\end{align*}
for every $P \in \dd$.  In particular, we have
\begin{align*}
\sum_{j \colon Q(\sbf_j) \subset P} \sigma(Q(\sbf_j)) &\le \frac{C_2}{\lambda\eps^2}\sigma(P).
\end{align*}
\end{lemma}

Proving Lemma \ref{lemma:packing_for_type_4} is much more delicate than proving Lemma \ref{lemma:packing_for_types_1-3} and we do this in several steps. The key idea is to reduce the proof to proving estimates for which we can use lemmas from Section \ref{section:set-up_for_approximability}.
 
Let us fix a regime $\sbf_j$ that is of Type 4. Since $\sbf_j$ is not of Type 1, 2 or 3, we know that, roughly speaking, at the ``bottom'' of the sawtooth domain $\Omega_{\F_{\sbf_j}, Q(\sbf_j)}$ there is a large region where $u$ has some (uniform) oscillation from the value $u(X_{Q(\sbf_j)})$. Let us be more precise. Since $\sbf_j$ is not of Type 1, we know that $\sigma(Q(\sbf_j) \setminus \cup_{Q \in \F_{\sbf_j}} Q) < \lambda \sigma(Q(\sbf_j)$, and since $\sbf_j$ is not of Type 2 or 3, we have
\begin{align*}
	\sigma\big((\cup_{Q \in \F_{\sbf_j}^{\bigR}} Q)\cup (\cup_{Q \in \F_{\sbf_j}^{\bigY}}Q)\big) < 2\lambda \sigma(Q(\sbf_j)).
\end{align*}
Thus, since $\F_{\sbf_j} = \F_{\sbf_j}^{\bigR} \cup \F_{\sbf_j}^{\bigSB} \cup \F_{\sbf_j}^{\bigY}$, it holds that
\begin{align*}
	\sigma\big(\cup_{Q \in \F_{\sbf_j}^{\bigSB}} Q\big) \ge (1- 3\lambda) \sigma(Q(\sbf_j)).
\end{align*}
Let $N = N(j)$ be so large that the subcollection
\begin{align*}
	\F_{N,j}^{\bigSB} \coloneqq \{Q \in \F_{\sbf_j}^{\bigSB} \colon \ell(Q) > 2^{-N}\ell(Q(\sbf_j))\}
\end{align*}
satisfies 
\begin{equation}\label{BBNbd.eq}
	\sigma\big( \cup_{Q \in\F_{N,j}^{\bigSB}} Q \big) \ge (1- 4\lambda) \sigma(Q(\sbf_j)).
\end{equation}
Our estimates will not depend on $N$ but we work with the subcollection $\F_{N,j}^{\bigSB}$   to avoid dealing with estimates on the boundary. We let $\F_{N,j}$ denote the collection of maximal cubes in the collection $\F_{\sbf_j} \cup \dd_{N, Q(\sbf_j)}$, where $\dd_{N,Q} = \{Q' \in \dd_Q \colon \ell(Q') = 2^{-N}\ell(Q)\}$ as earlier. We note that  $\Omega_{\F_{N,j}, Q(\sbf_j)} \subseteq \Omega_{\F_{\sbf_j},Q(\sbf_j)}$. Then 
\begin{align*}
	\F_{N,j} = \F_{N,j}^{\bigSB} \cup \F_{N,j}^{\bigO}, \ \text{ where } \  \F_{N,j}^{\bigO} \coloneqq \F_{N,j} \setminus \F_{N,j}^{\bigSB},
\end{align*}
where the superscript O in $\F_{N,j}^{\bigO}$ stands for ``other'' cubes. By \eqref{BBNbd.eq} we have 
\begin{align}\label{noBBNbd.eq}
	\sigma\big(\cup_{Q \in\F_{N,j}^{\bigO}} Q\big) = \sum_{Q \in\F_{N,j}^{\bigO}} \sigma(Q)  \le (4\lambda) \sigma(Q(\sbf_j)).
\end{align} 

With the notation above, we can formulate our key estimate for the proof of Lemma \ref{lemma:packing_for_type_4}:

\begin{lemma}\label{lemma:sawtooth_estimate}
	Suppose that $\sbf_j$   is a stopping time regime of Type 4. 	There exists $\lambda_0 > 0$ depending only on structural constants, ellipticity, and the $\omega_{L,\Omega}\in A_\infty(\sigma)$ constants, such that for any $\lambda\in(0,\lambda_0)$, there exists a constant $C_3\geq1$ depending only on structural constants, ellipticity, and the $\omega_{L,\Omega}\in A_\infty(\sigma)$ constants (and independent of $\eps$, $\lambda$, $N$, $j$, and $\Sc$), so that the following estimate holds:
	\begin{align}\label{type3goal.eq}
		\sigma(Q(\sbf_j)) \le \frac{C_3}{\eps^2} \iint_{\Omega_{\F_{N,j}, Q(\sbf_j)}} |\nabla u(Y)|^2 \delta(Y) \, dY.
	\end{align}
\end{lemma}

Taking Lemma \ref{lemma:sawtooth_estimate} for granted momentarily, we can prove Lemma \ref{lemma:packing_for_type_4} in a straightforward way:

\begin{proof}[Proof of Lemma \ref{lemma:packing_for_type_4}]
Fix $P \in \dd$. Then we have  
\begin{equation*}
\sum_{\substack{j \colon \sbf_j \text{ is T4} \\ Q(\sbf_j) \subset P}} \sigma(Q(\sbf_j))
\le\frac{C_3}{\eps^2} \sum_{\substack{j \colon \sbf_j \text{ is T4} \\ Q(\sbf_j) \subset P}}  \iint_{\Omega_{\F_{N,j}, Q(\sbf_j)}} |\nabla u(Y)|^2 \delta(Y) \, dY \\
\lesssim \frac{C_3}{\eps^2} \iint_{T_P} |\nabla u(Y)|^2 \delta(Y) \, dY
\lesssim \frac{C_3}{\eps^2} \sigma(P),
\end{equation*}
where we used  Lemma \ref{lemma:sawtooth_estimate},  the fact that the bounded overlap of the regions $U_Q$ and the disjointness of the collections $\sbf_j$ imply that the regions $\Omega_{\F_{N(j),j}, Q(\sbf_j)}$ have bounded overlaps,   the fact that $T_P \subset B(x_P, C\ell(P))$ and Lemma \ref{realCME.lem}. The rest of the claim follows now from Lemma \ref{lemma:packing_for_types_1-3}.
\end{proof}

\subsection{Proof of Lemma \ref{lemma:sawtooth_estimate}: A high oscillation estimate}
Let us then start processing the estimate \eqref{type3goal.eq}. Let $\sbf_j$ be a fixed stopping time regime of Type 4. To relax the notation, we denote
\begin{gather*}
	\sbf_* \coloneqq \dd_{\F_{N,j},Q(\sbf_j)} \subseteq \sbf_j,\qquad
	Q(\sbf_*) \coloneqq Q(\sbf_j),\qquad
	X_* \coloneqq X_{Q(\sbf_j)},\\[2mm]
	\Omega_* \coloneqq \Omega_{\F_{N,j}, Q(\sbf_j)},\qquad
	\F \coloneqq \F_{N,j},\qquad
	\F^{\bigSB} \coloneqq \F_{N,j}^{\bigSB}, \quad \text{ and}\quad
	\F^{\bigO} \coloneqq \F_{N,j}^{\bigO}. 
\end{gather*}
Recall that   $Q(\sbf_*) \in \Sc$ and that $\Sc$ is a coherent stopping time regime satisfying the property (iv) in Lemma \ref{Coronaellipmeasure.lem}.
Thus, by using Lemma \ref{GFdeltacomp.lem}, Lemma \ref{lem5.lem} and Lemma \ref{prelimlem1.lem} in this order, we get
\begin{multline}\label{t3reduc1.eq}
\iint_{\Omega_*} |\nabla u(Y)|^2 \delta(Y) \, dY
		\approx \sigma(Q(\sbf_*))  \iint_{\Omega_*} |\nabla u(Y)|^2 G_{L,\Omega}(X_*, Y) \, dY \\
		\ge \sigma(Q(\sbf_*)) \iint_{\Omega_*} |\nabla u(Y)|^2 G_{L,\Omega_*}(X_*, Y)\, dY 
		\approx \sigma(Q(\sbf_*)) \int_{\pom_*} (u(y) - u(X_*))^2 \, d\hm_*(y),
\end{multline}
where  
\begin{center}
	$\hm_*$ is the elliptic measure for  $L$ in  $\Omega_*$ with pole at $X_*$.
\end{center}
Thus, Lemma \ref{lemma:sawtooth_estimate} follows immediately from the following estimate. Recall that $\lambda$ is the parameter we used when we defined the Types 1--4 for the stopping time regimes.
\begin{lemma}\label{lotsofosc.cl}
	There exists $\lambda_0 > 0$ depending on structural constants, ellipticity, and the $\omega_{L,\Omega}\in A_\infty(\sigma)$ constants, such that for any $\lambda\in(0,\lambda_0)$, there exists a constant $c_4>0$ depending only on   structural constants, ellipticity, and the $\omega_{L,\Omega}\in A_\infty(\sigma)$ constants (and independent of $\eps$, $\lambda$, $N$, $j$, and $\Sc$), so that the following estimate holds:
	\begin{equation}\label{lotsofosceq.eq}\nonumber
		\int_{\pom_*} (u(y) - u(X_*))^2 \, d\hm_*(y) \geq c_4 \eps^2.
	\end{equation}
\end{lemma}
For the the proof of Lemma \ref{lotsofosc.cl}, we need some auxiliary constructions and estimates. Recall that $X_* = X_{Q(\sbf_*)}$ is a corkscrew point relative to $Q(\sbf_*)$ at scale $r_* \coloneqq 10^{-5} a_0 \ell(Q(\sbf_*))$ (see Section \ref{kappachoice.subsect}). Let $\hat{x}_* \in \pom$ be a touching point for $X_*$ on $\pom$, that is, $|\hat{x}_* - X_*| = \dist(X_*, \pom)$. For $\xi \in [0,1]$, consider the points $X(\xi) \coloneqq \hat{x}_* + \xi(X_*- \hat{x}_*)$ which lie on the line segment from $\hat{x}_*$ to $X_*$. Since we know that $\hat{x}_* \not \in\overline{\Omega}_*$ and $X_* \in \Omega_*$, there exists $\xi_0 \in (0,1)$ such that  $X(\xi_0) \in \partial \Omega_*$.  Now we set	$X_{**}=X(\xi_0)$, and $\Delta_* \coloneqq B(X_{**}, r_*) \cap \partial\Omega_*$. Since we are working with the truncated collection of stopping cubes $\F = \F_{N,j}$, we know that $\pom \cap \partial\Omega_* = \emptyset$. In particular, $\Delta_* \subset \Omega$. By Lemma \ref{lemma:touching_point_surface_ball}, we know that
\begin{align}\label{deltastartouchpteq.eq}
	\text{if } y \in \Delta_*, \ \ \text{ then } \hat{y} \in \Delta_{Q(\sbf_*)} \subseteq Q(\sbf_*),
\end{align}
where $\hat{y}$ is a touching point for $y$ in  $\Omega$, that is, $|y - \hat{y}| = \dist(y, \pom)$. 

By Lemma \ref{lemma:hm}, we know that $\Omega_*$ is also a uniform domain with Ahlfors regular boundary.\footnote{By Lemma \ref{lemma:hm}, the structural constants of $\Omega_* = \Omega_{\F_{N,j}, Q(\sbf_j)}$ do not depend on $j$ or the truncation parameter $N = N(j)$.} Thus, by Lemma \ref{bour.lem}, we have
\begin{align}\label{estimate:elliptic_measure_omegastar}
	\hm^{\widetilde{X}}_{L, \Omega_*}(\Delta_* ) \gtrsim 1,
\end{align}
where $\widetilde{X}$ is a corkscrew point relative to $X_{**}$ at scale $r_*$ in the domain $\Omega_*$. Recall that by the construction in Section \ref{section:dyadic_and_whitney}, we know that $B(X_*, \delta(X_*)/2) \subset \Omega_*$ and $\diam(\Omega_*) \approx \ell(Q(\sbf_*))$, and we have $\delta(X_*) \approx \ell(Q(\sbf_*)) \approx r_* \approx \delta(\widetilde{X})$ by the definition of $X_*$ and $\widetilde{X}$. Thus, by the Harnack chain property of $\Omega_*$, there exists a Harnack chain of uniformly bounded length from $\widetilde{X}$ to $X_*$ inside $\Omega_*$. Thus, by \eqref{estimate:elliptic_measure_omegastar}, formula \eqref{formula:elliptic_measure} and Lemma \ref{lemma:harnack} (that is, Harnack inequality), there exists a constant $c_* > 0$ that depends only on structural constants such that
\begin{align}\label{hmxstarnondegen.eq}
	\hm_*(\Delta_* ) = \hm^{X_*}_{L, \Omega_*}(\Delta_* ) > c_*.
\end{align}

Next, we will construct a cover of $\Delta_*$ that consists of dilated surface balls on $\pom_*$ associated to the cubes $Q \in \F$. We will construct the cover in such a way that there is oscillation of $u$ on the balls associated to cubes in $\F^{\bigSB}$ and the balls associated to cubes in $\F^{\bigO}$ do not have much $\hm_*$-mass, provided $\lambda$ is sufficiently small. Given $Q\in\bb D$, denote by $\hat x_Q$ a touching point for the corkscrew point $X_Q$. For $\theta \in [0,1]$, recall that we denote $	P_Q(\theta) = \hat{x}_Q + \theta(X_Q - \hat{x}_Q)$, and
by Lemma \ref{thetachoice.cl} we showed that there exists $\theta_0 \in (0,1)$ such that if for some $Q'\in\bb D$ we have $B(P_Q(\theta_0), \tfrac{\gamma\theta_0}{10} r_Q) \cap U_{Q'} \neq \emptyset$,  then $Q' \subset Q$ and $\ell(Q') < \ell(Q)$, where $\gamma$ is the corkscrew constant in Definition \ref{CS.def}.

Fix $Q \in \mathcal{F}$. Then its parent $\widetilde{Q}$ satisfies $\widetilde{Q} \in \dd_{\F, Q(\sbf_*)}$ and hence, by the construction of the Whitney regions in Section \ref{section:dyadic_and_whitney}, we have $P_Q(1) = X_Q \in U_{\widetilde{Q}} \subset \Omega_*$. By Lemma \ref{thetachoice.cl}, we also know that $P_Q(\theta_0) \not \in \Omega_*$. Indeed, otherwise  $P_Q(\theta_0) \in U_{Q'}$ for a cube $Q' \in \dd_{\F, Q(\sbf_*)}$, but by Lemma \ref{thetachoice.cl} we would then have $Q' \subset Q$ with $\ell(Q') < \ell(Q)$. This is impossible since $Q \in \mathcal{F}$. Thus, there exists $\theta' \in (\theta_0, 1)$ such that 
\begin{align*}
  X^*_Q \coloneqq P_Q(\theta') \in \partial \Omega_*.
\end{align*}
We set
\begin{align*}
	\Delta^*_Q \coloneqq B(X^*_Q, \tfrac{\lambda\theta_0}{2} r_Q) \cap \pom_*
	\quad \text{ and } \quad
	M\Delta^*_Q \coloneqq B(X^*_Q, M \tfrac{\lambda\theta_0}{2} r_Q) \cap \pom_*
\end{align*}
for a constant $M \ge 1$ to be chosen momentarily.

Let us describe some of the properties of $\Delta^*_Q$. First, by Lemma \ref{lemma:cylinder_inclusion} and  definition of $\Delta_Q^*$, it holds that 
\begin{equation}\label{deltstrqinuqstr.eq}
	\Delta^*_Q \subseteq \Xi_Q \subseteq U_{Q'}^*
\end{equation}
whenever $Q'$ is a sibling of $Q$, where
\begin{align*}
  \Xi_{Q'} = \bigcup_{\theta \in [\theta_0,1]} B\big(P_{Q'}(\theta), \tfrac{\gamma\theta_0}{10} r_{Q'} \big).
\end{align*}
Next, let us observe that if $Q \in \F^{\bigSB}$, then there exists a sibling $Q'$ of $Q$ such that $Q'$ is blue and $|u(X_{Q'}) - u(X_{Q(\sbf_*)})| > \eps/100$. Thus, for every $Q \in \F^{\bigSB}$, it holds that
\begin{multline}\label{oscondeltstrq.eq}
|u(X) - u(X_*)| \ge |u(X_{Q'}) - u(X_{Q(\sbf_*)})| - |u(X_{Q'}) - u(X)| 
\\  \ge \eps/100 - \eps/1000 \ge \eps/200,
\end{multline}
for every $X \in \Delta^*_Q$ since $Q'$ is blue and $\Delta^*_Q \subset U_{Q'}^*$ by \eqref{deltstrqinuqstr.eq}.

\begin{lemma}\label{lemma:covering_argument}
There exists $M \ge 1$, depending only on structural constants,   such that
\begin{align*}
\Delta_* = B(X_{**}, r_*) \cap \partial\Omega_* \subset \bigcup_{Q \in \mathcal{F}} M\Delta^*_Q.
\end{align*}
\end{lemma}

\begin{proof}
  Fix $y \in \Delta_* \subset \partial \Omega_*$ and a touching point $\hat{y} \in \pom$ for $y$. Then by \eqref{deltastartouchpteq.eq} we have $\hat{y} \in \Delta_{Q(\sbf_*)}$. Since $X_{**}$ lies on the line segment from the corkscrew point $X_*$ relative to $x_{Q(\sbf_*)}$ at scale $r_*$ to its touching point $\hat{x}_*$, we have $|\hat{x}_{*} - X_{**}| =  \dist(X_{**}, \pom) \le \dist(X_*, \pom) \le r_* = 10^{-5}a_0 \ell(Q(\sbf_*))$. Thus, $y \in B(\hat{x}_{*}, 2r_*) = B(\hat{x}_{Q(\sbf_*)}, 2(10)^{-5}a_0 \ell(Q(\sbf_*)))$. By \eqref{deltastartouchpteq.eq}, we have
  \begin{equation}\label{yclosetoqsbf.eq}
	|y-\hat{y}| = \dist(y, \pom) = \dist(y,Q(\sbf)) \le \ell(Q(\sbf_*)),
  \end{equation}
  and by the definition of $\W_{Q'}(K_0)$ (which we used in the construction of the Whitney regions), for any cube $Q'$ it holds that
  \begin{align}\label{whereisy.eq}
	\text{if } \hat{y} \in Q' \text{ and } C_\tau^{-1} K_0^{-1}|y-\hat{y}| \leq \ell(Q') \leq C_\tau K_0|y-\hat{y}|, \quad \text{ then } y \in \interior(U_{Q'}),
  \end{align}
  where $\tau$ is the dilation parameter in the definition of $U_Q$. Since $\sbf_*$ and $\F$ are the truncated collections, we have $Q(\sbf_*) = \cup_{Q \in \F} Q$. In particular, by Lemma \ref{lemma:touching_point_surface_ball}, we have $\hat{y} \in Q(\sbf_*) = \cup_{Q \in \F} Q$. Let $Q_{\hat{y}} \in \F$ be the cube such that $\hat{y} \in Q_{\hat{y}}$. We now have $\ell(Q_{\hat{y}}) \geq C_\tau^{-1} K_0^{-1} |y - \hat{y}|$   for the same constant $C_\tau$ as in \eqref{whereisy.eq}, since 
  otherwise there exists a cube $Q'$ such that $Q_{\hat{y}} \subset Q' \subseteq Q(\sbf_*)$ with $C^{-1}_\tau K_0^{-1}|y-\hat{y}| \leq \ell(Q') \leq C_\tau K_0|y-\hat{y}|$. This is not possible since \eqref{whereisy.eq} and the fact that $Q_{\hat{y}} \in \F$ would then imply that $y \in \text{int} \, (U_{Q'}) \subset \Omega_*$, but we know that $y \in \partial \Omega_*$. Thus, it holds that
  \begin{align}\label{thereisy.eq}
	\dist(Q_{\hat{y}}, y) = |\hat{y} - y| \leq C_\tau K_0 \ell(Q_{\hat{y}}).
  \end{align}
  Recall that $x_{Q_{\hat{y}}}$ is the center of $Q_{\hat{y}}$ and $X_{Q_{\hat{y}}}^*$ lies on a line segment from a corkscrew point $X_{Q_{\hat{y}}}$ to its touching point $\hat{x}_{Q_{\hat{y}}}$. By Lemma \ref{lemma:touching_point_surface_ball}, we know that $\hat{x}_{Q_{\hat{y}}} \in Q_{\hat{y}}$. Thus, \eqref{thereisy.eq}, the definitions of the points and the fact that $\hat{y},\hat{x}_{Q_{\hat{y}}} \in Q_{\hat{y}}$ give us
  \begin{align*}
    |y - X^*_{Q_{\hat{y}}}| 
    \le |y - \hat{y}| + |\hat{y} - \hat{x}_{Q_{\hat{y}}}| + |\hat{x}_{Q_{\hat{y}}} - X^*_{Q_{\hat{y}}}|
    \lesssim
    K_0 \ell(Q_{\hat{y}}) + \diam(Q_{\hat{y}}) + r_{Q_{\hat{y}}}
    \approx r_{Q_{\hat{y}}}.
  \end{align*}
  In particular, there exists $M \ge 1$ such that
  \begin{align*}
    y \in  M\Delta^*_{Q_{\hat{y}}} \subseteq  \bigcup_{Q \in \mathcal{F}} M\Delta^*_Q,
  \end{align*}
  which is what we wanted.
\end{proof}

Let us then fix $M \ge 1$ as in Lemma \ref{lemma:covering_argument}. By \eqref{hmxstarnondegen.eq}, it holds that
\begin{align}\label{nondegenonMdeltstar.eq}
	\hm_*\Big(\bigcup_{Q \in \mathcal{F}} M\Delta^*_Q \Big) \ge c_*.
\end{align}
Our next goal is to analyze how much the cubes $Q \in \F^{\bigO}$ contribute to \eqref{nondegenonMdeltstar.eq} and then limit this contribution by choosing $\lambda$ in a suitable way. For this, we prove the following bound:
\begin{lemma}
  For any $Q \in \F$, we have
  \begin{align}\label{hmprojest1.eq}
	\hm_*(\Delta^*_Q) \lesssim \frac{\sigma(Q)}{\sigma(Q(\sbf_*))}
  \end{align}
  for an implicit constant depending only on structural constants, ellipticity, and the $\omega_{L,\Omega}\in A_\infty(\sigma)$ constants (and independent of $\lambda$, $\eps$, $N$, $j$, and $\Sc$).
\end{lemma}

\begin{proof}
  Let   $Q \in \F$. Recall that $\Delta^*_Q = B(X^*_Q, \tfrac{\lambda\theta_0}{2} r_Q) \cap \pom_*$ is a surface ball on $\pom_*$ with $X^*_Q \in \pom_*$. Since $\Omega_*$ is a uniform domain, it satisfies the corkscrew condition. Let $\widetilde{X}_Q$ be a corkscrew point in $\Omega_*$ relative to $X^*_Q$ at scale approximately $r_Q' = \frac{\lambda\theta_0}{2} r_Q \approx r_Q$. By perhaps insisting that $\theta_0$ is smaller we have that this corkscrew point $\widetilde{X}_Q$ is far from $X_*$. Then by connecting $\widetilde{X}_Q$ to $X_Q$ with a Harnack chain\footnote{Here we use that $\dist(X^*_Q, \pom) \approx \dist(X^*_Q, Q) \approx \ell(Q) \approx r_Q$; with very crude bounds this can be seen from \eqref{deltstrqinuqstr.eq}. Then we also use that $\Omega_* \subset \Omega$, so that $\ell(Q) \approx r_Q' \approx \dist(\widetilde{X}_Q, \pom_*) \le  \dist(\widetilde{X}_Q, \pom)$.}  (of uniformly bounded length) in $\Omega$ and using   Lemma \ref{CFMSest.lem},
\[
G_{L,\Omega}(X_*, \widetilde{X}_Q) \approx G_{L,\Omega}(X_*, X_Q) \lesssim \ell(Q) \frac{\hm^{X_*}_{L,\Omega}(Q)}{\ell(Q)^n}.
\]
Now, by \ref{lem5.lem} we have that $G_{L,\Omega_*}(X_*, \widetilde{X}_Q) \le G_{L,\Omega}(X_*, \widetilde{X}_Q)$, and then by  Lemma \ref{CFMSest.lem} in $\Omega_*$ we conclude that 
\[
\frac{\hm_*(\Delta^*_Q)}{(r'_Q)^n} r'_Q = \frac{\hm^{X_*}_{L, \Omega_*}(\Delta^*_Q)}{(r'_Q)^n} r'_Q \lesssim  G_{L,\Omega_*}(X_*, \widetilde{X}_Q) \le G_{L,\Omega}(X_*, \widetilde{X}_Q).
\]
Combining the two previously displayed inequalities and using that $\ell(Q) \approx r'_Q$ we have $\hm_*(\Delta^*_Q) \lesssim \hm^{X_*}_{L,\Omega}(Q)$. By Lemma \ref{lm.change} and Lemma \ref{Coronaellipmeasure.lem} applied twice for both $Q,Q(\sbf_*)\in\Sc$,    it holds that 
\[
\hm^{X_*}_{L,\Omega}(Q) = \hm^{X_{Q(\sbf_*)}}_{L,\Omega}(Q) \approx\frac{\omega_{L,\Omega}^{X_{Q(\Sc)}}(Q)}{\omega_{L,\Omega}^{X_{Q(\Sc)}}(Q(\sbf_*))}\approx\frac{\sigma(Q)}{\sigma(Q(\Sc))}\frac{\sigma(Q(\Sc))}{\sigma(Q(\sbf_*))}=\frac{\sigma(Q)}{\sigma(Q(\sbf_*))},
\]
which ends the proof of \eqref{hmprojest1.eq}.
\end{proof}

Now we are  ready to conclude the proof of Lemma \ref{lotsofosc.cl}. Using \eqref{hmprojest1.eq} and the doubling property of $\hm_*$ we find that
\begin{equation}\label{FOsmall.eq}
\sum_{Q \in \mathcal{F}^{\bigO}} \hm_*\left( M\Delta^*_Q \right) \leq C \sum_{Q \in \mathcal{F}^{\bigO}} \hm_*\left(\Delta^*_Q \right) \leq C \sum_{Q \in \mathcal{F}^{\bigO}} \frac{\sigma(Q)}{\sigma(Q(\sbf_*))}  \leq \hat C  \lambda, 
\end{equation}
where we used \eqref{noBBNbd.eq} in the last inequality.
Now we choose $\lambda> 0$ so that $\hat C\lambda < c_*/2$ and use \eqref{FOsmall.eq} and \eqref{nondegenonMdeltstar.eq} to deduce
\begin{equation}\label{nondegenonBBMdeltstar.eq}
	\hm_*\Big(\bigcup_{Q \in \mathcal{F}^{\bigSB}} M\Delta^*_Q \Big) \ge c_*/2,
\end{equation}
where we used $\F = \F^{\bigSB} \cup \F^O$.  Now we use the $5R$-covering lemma \cite{Mat} to produce a countable collection of disjoint surface balls $\{M\Delta^*_k\}\coloneqq \{M\Delta^*_{Q_k}\}$ where each $Q_k$ is in  $\F^{\bigSB}$ and such that
\[\bigcup_{Q \in \mathcal{F}^{\bigSB}} M\Delta^*_Q\subset \cup_{k} 5M\Delta^*_k.\]
Then using \eqref{nondegenonBBMdeltstar.eq} and the doubling property of $\hm_*$ it holds
\[\hm_*(\cup_k \Delta^*_k) = \sum_k \hm_*(\Delta^*_k) \gtrsim \sum_{k} \hm_*(5M\Delta^*_k)  \gtrsim \hm_*\Big(\bigcup_{Q \in \mathcal{F}^{\bigSB}} M\Delta^*_Q \Big) \ge c_*/2,\]
where we used that $\Delta^*_k$ are disjoint. To summarize we have produced  a sequence of surface balls $\Delta^*_k = \Delta^*_{Q_k}$ with $Q_k \in \F^{\bigSB}$ such that
\begin{equation}\label{nondegendeltak.eq}
	\hm_*(\cup_k \Delta^*_k)  \ge c_{**},
\end{equation}
where $c_{**}$ depends on dimension, ellipticity, the Ahlfors regularity constant for $\pom$, the corkscrew and Harnack Chain constants for $\Omega$, and the  $\hm_L \in A_\infty(\sigma)$ constants. Thus, using \eqref{oscondeltstrq.eq} we have
\begin{equation*}
	\int_{\pom_*} (u(y) - u(X_*))^2 \, d\hm_*(y) \ge \hm_*(\cup_k \Delta^*_k)\inf\limits_{y \in \cup_k \Delta^*_k}(u(y) - u(X_*))^2
	\\  \ge  c_{**}(\eps/200)^2 \eqqcolon c_4 \eps^2,
\end{equation*}
which proves Lemma \ref{lotsofosc.cl}. 

As we had reduced the proof of the packing of the Type 4 maximal cubes to Lemma \ref{lotsofosc.cl}, this completes the proof of Lemma \ref{lemma:packing_for_type_4}.

\subsection{Construction of $\eps$-approximators} With the help of the previous constructions and estimates, we can prove the existence of $\BV_{\loc}$ $\eps$-approximators in a similar way as in \cite{HMM}. For the convenience of the reader, we recall the key steps of the construction below. For some of the details, we follow the construction of $L^p$-type approximators in \cite{HT1} which are an adaptation of the arguments in \cite{HMM}. Recall that we denote the collection of blue cubes in the disjoint stopping time regimes $\Sc$ in Lemma \ref{Coronaellipmeasure.lem} by $\Lc = \Lc(\Sc)$, and each of these collections has a decomposition $\Lc = \cup_j \sbf_j$.

\begin{proof}[Proof of Theorem \ref{theorem:existence_approximators}]
  Let us fix a dyadic cube $Q_0 \in \dd(\pom)$ and construct an $\eps$-approximator $\Phi_{Q_0} = \Phi_{Q_0}^\eps$ first in the Carleson box $T_{Q_0}$. We start by dividing the Carleson box $T_{Q_0}$ into a few types of different regions where we define the approximator differently. Let us choose the largest good (in the sense of the corona decomposition from Lemma \ref{Coronaellipmeasure.lem}), blue subcube $Q_1 \subset Q_0$ which may be $Q_0$ itself; if there are several such cubes with the largest side length, we choose just one of them. Since $Q_1$ is a good blue cube, there exists a stopping time regime $\Sc_{Q_1}$ in Lemma  \ref{Coronaellipmeasure.lem} and subregime $\sbf_{Q_1} \subset \Lc(\Sc_{Q_1})$ such that $Q_1$ is the maximal element of the regime $\sbf^1 \coloneqq \sbf_{Q_1} \cap \dd_{Q_0}$. We then choose the largest good blue cube $Q_2$ from the collection $\dd_{Q_0} \setminus \sbf^1$. Similarly, $Q_2$ is the maximal element of the regime $\sbf^2 \coloneqq \sbf_{Q_2} \cap \dd_{Q_0}$. We then choose the largest good blue cube $Q_3 \in \dd_{Q_0} \setminus (\sbf^1 \cup \sbf^2)$, and continue like this. This gives us a sequence of good blue cubes $Q_1, Q_2, \ldots$ such that $\ell(Q_1) \ge \ell(Q_2) \ge \ldots$, each cube $Q_i$ is a maximal element of a regime $\sbf^i$ and the collection $\cup_i \sbf^i$ contains all the good blue cubes in $\dd_{Q_0}$. The cubes $Q_i$ are ``mostly'' of the form $Q(\sbf)$ as in the decomposition of the collections $\Lc(\Sc)$ earlier in the sense that there exists a collection of pairwise disjoint cubes $\{P_k\}_k \subset \dd_{Q_0}$ (that may be empty) such that every cube in the collection $\{Q_i\}_i \setminus \{P_k\}_k$ is of the form of $Q(\sbf)$ for some $\sbf$. This is because the cube $Q_0$ is arbitrary and hence, it may be a bad cube or a red cube. For each $i$, we define the ``bottom'' cubes of $\sbf^i$ in the obvious way: we set $\F_{\sbf^i} \coloneqq \F_{\sbf_{Q_i}}$, that is, $\F_{\sbf^i}$ is the collection of the stopping cubes associated to the unique regime $\sbf_{Q_i}$ that contains $Q_i$.
  
  For each $i$, we define the regions $A_i$ recursively the following way:
  \begin{align*}
    A_1 \coloneqq \Omega_{\F_{\sbf^1},Q_1}, \qquad A_i \coloneqq \Omega_{\F_{\sbf^i,Q_i}} \setminus \bigcup_{k=1}^{i-1} A_k \quad \text{ for } i \ge 2.
  \end{align*}
  By construction, the regions $A_i$ are pairwise disjoint. We also set $\Omega_0 \coloneqq \bigcup_i A_i$,   and we define the function $\Phi_0$ on $\Omega_0$ as
  \begin{align*}
    \Phi_0 \coloneqq \sum_i u(X_{Q(\sbf_{Q_i})}) \mathbbm{1}_{A_i},
  \end{align*}
  where $X_{Q(\sbf_{Q_i})}$ is the corkscrew point we used in the stopping conditions in the definition of $\sbf_{Q_i}$. In particular, for any $X \in A_i$ we have $|u(X) - u(X_{Q(\sbf_{Q_i})})| \le \eps/100 < \eps$. Furthermore, by the disjointness of the regions $A_i$, we have $\|u - \Phi_0\|_{L^\infty(\Omega_0)} < \eps$.
  
  Let us then consider the cubes in $\dd_{Q_0} \setminus \cup_i \sbf^i$. Let us fix some enumeration $\{R_j\}_j$ for the cubes $\dd_{Q_0} \setminus \cup_i \sbf^i$. The cubes $R_j$ are red cubes or bad blue cubes. For each $j$, we define the regions $V_j$ recursively the following way:
  \begin{align*}
    V_1 \coloneqq U_{R_1}, \qquad V_j \coloneqq U_{R_j} \setminus \bigcup_{k=1}^{j-1} V_k \quad \text{ for } j \ge 2.
  \end{align*}
  By construction, the regions $V_j$ are pairwise disjoint. We also set $\Omega_1 \coloneqq \bigcup_j V_j$,  and we define the function $\Phi_1$ on $\Omega_1$ as
  \begin{align*}
    \Phi_1(X) \coloneqq \left\{ \begin{array}{cl}
                                  u(X), &\text{ if } X \in V_k \text{ for a red cube } R_k,\\
                                  u(X_k), &\text{ if } X \in V_k \text{ for a blue cube } R_k,
                                \end{array} \right.
  \end{align*}
  where $X_k$ is any fixed point on $U_{R_k}$. By the definitions, we have $\|u - \Phi_1\|_{L^\infty(\Omega_1)} < \eps / 1000 < \eps$.
  
  We define the $\eps$-approximator $\Phi_{Q_0}$ of $u$ in the Carleson box $T_{Q_0}$ as
  \begin{align*}
    \Phi_{Q_0}(X) \coloneqq \left\{ \begin{array}{cl}
                                  \Phi_0(X), &\text{ if } X \in \Omega_0,\\
                                  \Phi_1(X), &\text{ if } X \in T_{Q_0} \setminus \Omega_0.
                                \end{array} \right.
  \end{align*}
  By the construction, we have $\|u - \Phi_{Q_0}\|_{L^\infty(T_{Q_0})} < \eps$. The $L^1$-type Carleson measure estimate for $\Phi_{Q_0}$ in $T_{Q_0}$ can be proven as in \cite{HMM} with small but quite obvious changes. Using a covering argument, the claim can be reduced to proving the estimate on Carleson boxes $T_{Q'}$, and since $u \in L^\infty(\Omega)$, the core challenge is to handle the jumps across the boundaries of the sets $A_i$ and $V_j$ that contribute to the total variation of $\Phi_{Q_0}$ inside $T_{Q_0}$. Since the boundaries of the sawtooth regions, Whitney regions and Carleson boxes are Ahlfors regular by Lemma \ref{lemma:hm}, the estimates reduce to using the Carleson packing conditions in Lemma \ref{Coronaellipmeasure.lem}, Lemma \ref{lemma:packing_of_red_and_yellow}, Lemma \ref{lemma:packing_for_types_1-3} and Lemma \ref{lemma:packing_for_type_4}. The Carleson norm of the measure $\mu_{\Phi_{Q_0}}$ such that $d\mu_{\Phi_{Q_0}}(Y) = |\nabla \Phi_{Q_0}(Y)| \, dY$ is given (up to a structural constant) by the sizes of the Carleson packing norms in these results. We omit the details.
  
  Using these kinds of local approximators, we build the global approximator of $u$. If $\diam(\Omega) < \infty$, it is enough to build a local approximator for a Carleson box that covers the whole space $\Omega$. Thus, we may assume that $\diam(\Omega) = \infty$. Suppose first that $\diam(\pom) < \infty$. Then there exists a dyadic cube $Q_0$ that covers the whole boundary $\pom$. We build the local approximator $\Phi_{Q_0}$ on $T_{Q_0}$, extend it to whole $\Omega$ by setting it to be $0$ outside $T_{Q_0}$ and define the global approximator as $\Phi = \mathbbm{1}_{T_{Q_0}} \Phi_{Q_0} + \mathbbm{1}_{\Omega \setminus T_{Q_0}} u$. The $L^1$-type Carleson measure estimate follows from the same arguments as with the local approximators.
  
  Finally, suppose that $\diam(\pom) = \infty$. Fix a sequence of dyadic cubes $P_k$ such that $P_1 \subset P_2 \subset \cdots$, $\ell(P_1) < \ell(P_2) < \cdots$ and $\pom = \cup_k P_k$. This type of sequence of cubes does not exist in every dyadic system, but we can always construct a system where it exists (see, for example, \cite{hyt-tap}). We build a local approximator $\Phi_{P_k}$ in $T_{P_k}$ for every $k$ and extend the approximators to whole $\Omega$ by setting each of them to be $0$ outside $T_{P_k}$. We then define the global approximator as $\Phi = \mathbbm{1}_{T_{P_1}} \Phi_{P_1} + \sum_{k=2}^\infty \mathbbm{1}_{T_{P_k} \setminus T_{P_{k-1}}} \Phi_{P_k}$. The $L^1$-type Carleson measure estimate follows from the Carleson measure estimates of the local approximators and the fact that the collection $\{P_k\}_k$ satisfies a Carleson packing condition with a uniformly bounded Carleson packing norm depending only on structural constants. Again, we omit the details.
\end{proof}

To finish the proof of Theorem \ref{theorem:approximability}, we regularize the approximators in Theorem \ref{theorem:existence_approximators}. This regularization makes the constant $C_\eps$ significantly larger but since the size of this constant is not important for our results, we do not track its size.

\begin{lemma}\label{lemma:smooth_approximators}
	Let $\eps \in (0,1)$. There exists a unifomly bounded constant $\widetilde{C}_\eps \ge 1$ such that we can choose the $\eps$-approximator $\Phi = \Phi^\eps$ for the solution $u \in W^{1,2}(\Omega) \cap L^\infty(\Omega)$ to $Lu = 0$ in Theorem \ref{theorem:existence_approximators} so that
	\begin{enumerate}
		\item[i)] $\|u - \Phi\|_{L^\infty(\Omega)} \le 2 \eps \|u\|_{L^\infty(\Omega)}$,
		\item[ii)] $\sup_{x \in \pom, r > 0} \frac{1}{r^n} \iint_{B(x,r) \cap \Omega} |\nabla \Phi(Y)|\, dY \le \widetilde{C}_\eps \|u\|_{L^{\infty}(\Omega)},$
		\item[iii)] $\Phi \in C^\infty(\Omega)$,
		\item[iv)] $|\nabla \Phi(Y)| \le \tfrac{\widetilde{C}_\eps}{\delta(Y)}$ for every $Y \in \Omega$,
		\item[v)] if $|X-Y| \ll \delta(X)$, then $|\Phi(X) - \Phi(Y)| \le \tfrac{\widetilde{C}_\eps |X-Y|}{\delta(X)}$,
		\item[vi)] there exists a function $\varphi \in L^\infty(\pom)$ such that
		\begin{align*}
			\lim_{Y \to x, \, \nt} \Phi(Y) = \varphi(x) \text{ for } \sigma\text{-a.e. } x \in \pom.
		\end{align*}
	\end{enumerate}
	The constant $\widetilde{C}_\eps$ depends on $\eps$, the structural constants of $\Omega$, the constant $C_\eps$ in Theorem \ref{theorem:existence_approximators} and the H\"older continuity constants $C$ and $\alpha$ in Lemma \ref{DGN.lem}.
\end{lemma}

\begin{proof}
	The proof uses tweaked mollifier techniques combined with a regularized distance function. The properties follow mostly from \cite[Section 3]{HT2} but for the convenience of the reader, we define the core objects and give some explicit details below.
	
	Let $\beta$ be a regularized version of the distance function $\delta = \dist(\cdot,\pom)$, that is, a smooth function in $\Omega$ such that $\beta \approx \delta$ (see \cite[Theorem 2, p. 171]{stein-SIOs}). Let $\zeta \ge 0$ be a smooth non-negative function supported on $B(0,\tfrac{1}{m})$ for a suitable constant $m > 0$ (depending on the implicit constants in $\delta \approx \beta$), satisfying $\zeta \le 1$ and $\int \zeta = 1$.  For a constant $\xi_\eps > 0$ to be chosen momentarily, we set
	\begin{align*}
		\Lambda_{\xi_\eps}(X,Y) \coloneqq \zeta_{\xi_\eps \beta(X)}(X-Y) = \frac{1}{(\xi_\eps \beta(X))^{n+1}} \zeta\Big( \frac{X-Y}{\xi_\eps \beta(X)}\Big).
	\end{align*}
	For a suitable choice of $m$, we have $\text{supp} \, \Lambda(X,\cdot) \subset  B(X,\xi_\eps\delta(X)/2)$. Given the non-smooth $\eps$-approximator $\Phi_0$ of the solution $u \in W^{1,2}_{\loc}(\Omega) \cap L^\infty(\Omega)$ to $Lu = 0$ in Theorem \ref{theorem:existence_approximators}, we set
	\begin{align*}
		\Phi(X) \coloneqq \iint \Lambda_\eps(X,Y) \Phi_0(Y) \, dY.
	\end{align*}
	The property iii) follows from a standard modification of the case $\Omega = \ree$ (for example, see \cite[Theorem 1, p. 123]{EG}) and properties ii), iv) and v) are formulated explicitly in \cite[Section 3]{HT2}. Property i) follows from the local H\"older continuity of $u$ (that is, Lemma \ref{DGN.lem}) and the fact that $\Phi_0$ is an $\eps$-approximator of $u$: for almost every $X \in \Omega$, we get
	\begin{align*}
		|\Phi(X) - u(X)|
		&= \Big| \iint \Lambda_\eps(X,Y) (\Phi_0(Y) - u(X)) \, dY \Big| \\
		&< \eps\|u\|_{L^\infty(\Omega)} + \iint \Lambda_\eps(X,Y) \left| u(Y) - u(X) \right| \, dY \\
		&\le \eps\|u\|_{L^\infty(\Omega)} + \iint \Lambda_\eps(X,Y) \, C \Big(\frac{|X-Y|}{\tfrac{1}{4}\delta(X)} \Big)^\alpha \Big( \fint_{B(X,\tfrac{1}{2}\delta(X))} |u(Z)|^2 \, dZ \Big)^{1/2} \, dY \\
		&\le \eps\|u\|_{L^\infty(\Omega)} + C \left(2\xi_\eps\right)^\alpha \|u\|_{L^\infty(\Omega)} \iint \Lambda_\eps(X,Y) \, dY \le 2\eps \|u\|_{L^\infty(\Omega)}
	\end{align*}
	as long as we choose $\xi_\eps \le \tfrac{1}{2} \big(\tfrac{\eps}{C}\big)^{\frac1{\alpha}}$, where $C$ and $\alpha$ are the H\"older continuity constants in Lemma \ref{DGN.lem}.
	
	Property vi) follows from the same argument that is used in the proof of \cite[Lemma 4.14]{HT2} after some small additional considerations. The proof of \cite[Lemma 4.14]{HT2} is based on showing that almost every cone on a codimension $1$ uniformly rectifiable set has locally exactly two components, these local components satisfy the Harnack chain condition and the Harnack chain condition combined with the $L^1$-type Carleson measure estimate ii) implies the existence of the a.e. non-tangential trace $\varphi$. We do not assume that $\pom$ is uniformly rectifiable but by the definition of dyadic cones \eqref{defin:dyadic_cone} and Lemma \ref{lemma:hm} we know that any truncated cone on $\pom$ has exactly one component inside $\Omega$ and this component satisfies the Harnack chain condition. Thus, the argument in the proof of \cite[Lemma 4.14]{HT2} works also for us. In particular, we can choose $\Phi$ in such a way that all the properties i) -- vi) hold. This completes the proofs of Lemma \ref{lemma:smooth_approximators} and Theorem \ref{theorem:approximability}.
\end{proof}

\section{Proof of Theorem \ref{theorem:converse}}
\label{section:proof_of_converse}

In this section, we prove Theorem \ref{theorem:converse}, that is, we prove that $\eps$-approximability of solutions $u$ to $Lu = 0$ implies that $\omega_L \in A_\infty(\sigma)$. To be more precise, we prove the following seemingly stronger result: 
\begin{theorem}\label{theorem:converse_indicator_data}
	Let $\Omega \subset \ree$, $n\geq1$, be a uniform domain with Ahflors regular boundary, and let $L$ be a divergence form elliptic operator $L = -\div A\nabla$ in $\Omega$. Suppose also that for every bounded Borel set $S \subset \pom$ the solution $u=u_S$ to $Lu = 0$ such that $u(X) = \omega_L^X(S)$ is $\eps$-approximable for every $\eps \in (0,1)$ in the sense of Theorem \ref{theorem:approximability} with the $\eps$-approximability constants depending only on structural constants and $\eps$. Then $\omega_L \in A_\infty(\sigma)$.
\end{theorem}

In particular, by Theorem \ref{theorem:approximability} and Theorem \ref{theorem:converse_indicator_data}, $\eps$-approximability of the subclass of solutions $u$ to $Lu = 0$ in Theorem \ref{theorem:converse_indicator_data} is equivalent with $\eps$-approximability of all solutions $u$ to $Lu = 0$ (and hence, it is equivalent with the other conditions in Corollary \ref{corollary:characterization}).

The proof of Theorem \ref{theorem:converse_indicator_data} is based on the proof of \cite[Theorem 1.1]{CHMT}, which itself is based on the techniques used in \cite{KKoPT} and \cite{kkipt}. The key idea is the following. We fix a cube $Q_0$ and a Borel set $F \subset Q_0$ and we build a suitable solution $u = u_F$ associated to $F$ such that $u$ oscillates a significant amount. We then use this oscillation to control the $L^1$ norm of the gradient of an $\eps$-approximator of $u$ from below near $Q_0$ for a very small $\eps$. The $L^1$-type Carleson measure estimate of the $\eps$-approximator then allows us to verify the  $\omega_L\in A_\infty(\sigma)$ condition.

For the proof of Theorem \ref{theorem:converse_indicator_data}, we need some definitions and notation from \cite{CHMT}. For clarity, we adopt most of the notation as it is from \cite{CHMT}. We define the following fattened version of the Whitney regions $U_Q$ and a ``wider'' version of the truncated dyadic cone:
\begin{align*}
	U_{Q,\eta^3} \coloneqq \bigcup_{\substack{Q' \in \dd_Q \\ \ell(Q') > \eta^3 \ell(Q)}}U_{Q'},\qquad \text{ and } \qquad
	\Gamma_{Q_0}^{\eta}(x) \coloneqq \bigcup_{\substack{Q \in \dd_{Q_0} \\ Q \ni x}} U_{Q,\eta^3}.
\end{align*}

The main difference between our approach and the proof of the implication ``CME $\implies$ $A_\infty$'' in \cite{CHMT} is the following lemma which is a modification of \cite[Lemma 3.10]{CHMT}:

\begin{lemma}\label{lemma:lower_bound_sq} There exist $\eps\in(0,1)$, $0<\eta\ll1$, depending only on structural constants, and $\alpha_0\in(0,1)$, $C_\eta\geq1$, both depending on structural constants and on $\eta$, such that for each $Q_0\in\bb D$, for every $\alpha\in(0,\alpha_0)$, and for every   $F\subset Q_0$ satisfying $\omega_L^{X_{Q_0}}(F)\leq\alpha\omega_L^{X_{Q_0}}(Q_0)$, there exists a Borel set $S\subset Q_0$ such that if $u(X)=\omega_L^X(S)$ and $\Phi=\Phi^{\eps}$ is an $\eps$-approximator of $u$, then
\begin{align*}
\iint_{\Gamma_{Q_0}^\eta(y)} |\nabla \Phi(Y)| \delta(Y)^{-n} \, dY \ge C_\eta^{-1} \log(\alpha^{-1}),\qquad\text{for each }y\in F.
\end{align*} 
\end{lemma}

Before proving Lemma \ref{lemma:lower_bound_sq}, let us see how it gives us Theorem \ref{theorem:converse_indicator_data}.

\begin{proof}[Proof of Theorem \ref{theorem:converse_indicator_data}] Following \cite{CHMT}\footnote{Although the results of \cite{CHMT} are stated only for $n\geq2$, the proof of ``(a) $\implies$ (b)'' in \cite[Theorem 1.1]{CHMT} works for a large part also for $n=1$; see \cite{fp}.}, we show that for each $\beta\in(0,1)$, there exists $\alpha\in(0,1)$ such that for every $Q_0\in\bb D$ and every Borel set $F\subset Q_0$, we have that
\begin{equation}\label{eq.show}
\frac{\omega_L^{X_{Q_0}}(F)}{\omega_L^{X_{Q_0}}(Q_0)}\leq\alpha\quad\implies\quad\frac{\sigma(F)}{\sigma(Q_0)}\leq\beta,
\end{equation}
where the constants $\alpha$ and $\beta$ are independent of the choice of the dyadic system $\dd$. This is a dyadic version of the $A_\infty(\sigma)$ condition. Although it looks different than Definition \ref{defin:a_infty}, the conditions are equivalent since we consider doubling measures (see, for example, \cite[Chapter IV, Theorem 2.11]{GCRDF}) and the constants are independent of the system $\dd$ (see \cite[pp. 16]{CHMT} or use adjacent dyadic techniques \cite{hyt-k,hyt-tap}).

Fix $\beta\in(0,1)$ and $Q_0\in\bb D$. Moreover, fix $\eta\in(0,1)$ small enough, and constants $\eps$, $\alpha_0$, and $C_\eta$ as in Lemma \ref{lemma:lower_bound_sq}. Let $F\subset Q_0$ satisfy $\omega_L^{X_0}(F) \le \alpha \omega_L^{X_0}(Q_0)$ with $\alpha\in(0,\alpha_0)$. We now use Lemma \ref{lemma:lower_bound_sq} to see that there exists $S\subset Q_0$ such that if $\Phi^{\eps}$ is an $\eps$-approximator of $u(X)=\omega_L^X(S)$, then
\[
C_{\eta}^{-1}\log(\alpha^{-1})\sigma(F)\leq\int_F\iint_{\Gamma_{Q_0}^\eta(y)} |\nabla \Phi^{\eps}(Y)| \delta(Y)^{-n} \, dY\,d\sigma(y)\leq C\eta^{-3n} \sigma(Q_0),
\]
where the last estimate follows by Fubini's theorem and property ii) of $\eps$-approximability in Theorem \ref{theorem:approximability} (see \cite[pp. 15--16]{CHMT} for more details). We may then choose $\alpha$ small enough depending on $\beta$ so that (\ref{eq.show}) holds.
\end{proof}

We turn to the proof of Lemma \ref{lemma:lower_bound_sq}. For this, we need the following machinery:

\begin{definition}
	Let $Q_0 \in \dd$ be a dyadic cube, $\mu$ a regular Borel measure on $Q_0$, $F \subset Q_0$ be a Borel set, $\epsilon_0 > 0$, and $k \in\bb N$ be fixed. We say that a collection of nested Borel subsets $\{\mathcal{O}_\ell\}_{\ell = 1}^k$ of $Q_0$ is a \emph{good $\epsilon_0$-cover of $F$ of length $k$} for $\mu$ if
	\begin{enumerate}
		\item[(a)] $F \subset \oo_k \subset \oo_{k-1} \subset \cdots \subset \oo_2 \subset \oo_1 \subset Q_0$,
		\item[(b)] $\oo_{\ell} = \bigcup_i Q_i^{\ell}$ for disjoint subcubes $Q_i^\ell \in \dd_{Q_0}$,
		\item[(c)] $\mu(\oo_\ell \cap Q_i^{\ell-1}) \le \epsilon_0 \mu(Q_i^{\ell-1})$ for every $i$ and every $2 \le \ell \le k$.
	\end{enumerate}
\end{definition}

\begin{lemma}[{\cite[Lemma 3.5]{CHMT}}]\label{lemma:existence_eps_covers} Fix $Q_0\in\bb D$ and suppose that $\epsilon_0\in(0,\tfrac{1}{e})$. If $F \subset Q_0$ is a Borel set such that $\omega_L^{X_{Q_0}}(F) \le \alpha \omega_L^{X_{Q_0}}(Q_0)$ for $\alpha\in(0,\epsilon_0^2/C')$, then there exists a good $\epsilon_0$-cover of $F$ of length $k \approx \tfrac{\log \alpha^{-1}}{\log \epsilon_0^{-1}}$ for $\omega_L^{X_{Q_0}}$. Here $C'$ depends only on the constant $C$ of Lemma \ref{lm.doubling}. 
\end{lemma}

Let $Q_0 \in \dd$ be a fixed cube and $F \subset Q_0$ be a fixed subset such that $\omega_L^{X_0}(F) \le \alpha \omega_L^{X_0}(Q_0)$ for $\alpha$ small enough.  For each $Q \in \dd$, we let $\widetilde{Q} \in \dd$ be the unique dyadic cube such that $x_Q \in \widetilde{Q}$ and $\ell(\widetilde{Q}) = \eta \ell(Q)$ for $\eta > 0$ a small enough parameter to be determined later. Fix $\epsilon_0>0$ small enough. Let $\{\oo_l\}_{l=1}^k$ be a good $\epsilon_0$-cover of $F$ given by Lemma \ref{lemma:existence_eps_covers}. We set
\begin{align*}
	\widetilde{\oo}_j \coloneqq \bigcup_{i} \widetilde{Q}_i^l \quad \text{ and } \quad S \coloneqq \bigcup_{j=2}^k \widetilde{\oo}_{j-1} \setminus \oo_j
\end{align*}
By the nestedness of the sets $\oo_j$, we have $\mathbbm{1}_S = \sum_{j=2}^k \mathbbm{1}_{\widetilde{\oo}_{j-1} \setminus \oo_j}$. We define the nonnegative solution $u \coloneqq u_F \colon \Omega \to \mathbb{R}$ to $Lu = 0$ as
\begin{align*}
	u(X) = \int_{\pom} \mathbbm{1}_S(y) \, d\omega_L^X(y) = \omega_L^X(S) = \sum_{j=2}^k \omega_L^X(\widetilde{\oo}_{j-1} \setminus \oo_j).
\end{align*}

We have the following lower oscillation bound:

\begin{lemma}[{\cite[Lemma 3.24]{CHMT}}]\label{lemma:oscillation_lower_bound}
	There exists a structural constant $c_0 > 0$ such that if $\eta$ and $\epsilon_0 = \epsilon_0(\eta,c_0)$ are small enough, then for any $y \in F$ and any $1 \le l \le k-1$, there exist dyadic cubes $Q_i^l$ and $P_i^l$ such that
	\begin{align*}
		\big| u(\hat{X}_{\widetilde{Q}_i^l}) - u(\hat{X}_{\widetilde{P}_i^l}) \big| \ge c_0,
	\end{align*}
	where $Q_i^l$ is the unique cube (in $\oo_l$) such that $y \in Q_i^l$, $P_i^l \in \dd_{Q_i^l}$ is the unique cube such that $y \in P_i^l$ and $\ell(P_i^l) = \eta \ell(Q_i^l)$, $\widetilde{Q}_i^l$ and $\widetilde{P}_i^l$ are defined as we did after Lemma \ref{lemma:existence_eps_covers} and $\hat{X}_{\widetilde{Q}_i^l}$ and $\hat{X}_{\widetilde{P}_i^l}$ are corkscrew points relative to $x_{\widetilde{Q}_i^l}$ at scale $a_1 \ell(\widetilde{Q}_i^l)$ and relative to $x_{\widetilde{P}_i^l}$ at scale $a_1 \ell(\widetilde{P}_i^l)$, respectively.
\end{lemma}

With the help of Lemma \ref{lemma:oscillation_lower_bound}, we can now prove Lemma \ref{lemma:lower_bound_sq}:

\begin{proof}[Proof of Lemma \ref{lemma:lower_bound_sq}] The proof is a straightforward modification of the proofs of \cite[Lemma 3.10]{CHMT} and \cite[Lemma 4.14]{HT2}. Let us fix $y \in F$ and $l \in \{1,2,\ldots,k-1\}$. We borrow the notation from Lemma \ref{lemma:oscillation_lower_bound}, and set $\eps=c_0/4$. By adjusting the construction parameters for the Whitney regions in Section \ref{section:dyadic_and_whitney}, we may assume that there exist Whitney cubes $I_{\widetilde{Q}_i^l}$ and $I_{\widetilde{P}_i^l}$ such that\footnote{In Section \ref{section:dyadic_and_whitney}, we constructed the Whitney regions $U_Q$ in such a way that we know that a corkscrew point relative to $x_Q$ at scale $10^{-5}a_0\ell(Q)$ belongs to $U_Q$. Corkscrew points of the type $\hat{X}_Q$ are also relative to $x_Q$ but they are at scale $a_1 \ell(Q)$.}
\begin{align}\label{whitney_region_corkscrews}
	\hat{X}_{\widetilde{Q}_i^l} \in I_{\widetilde{Q}_i^l} \subset (1+\tau)I_{\widetilde{Q}_i^l} \subset U_{\widetilde{Q}_i^l} \quad \text{ and } \quad \hat{X}_{\widetilde{P}_i^l} \in I_{\widetilde{P}_i^l} \subset (1+\tau)I_{\widetilde{P}_i^l} \subset U_{\widetilde{P}_i^l}.
\end{align}  
Since $\ell(\widetilde{Q}_i^l) = \eta \ell(Q_i^l)$ and $\ell(\widetilde{P}_i^l) = \eta \ell(P_i^l) = \eta^2 \ell(Q_i^l)$. In particular, since $\eta < 1$, we have $U_{\widetilde{Q}_i^l}, U_{\widetilde{P}_i^l} \subset U_{Q_i^l,\eta^3}$.

Since $\Phi$ is a $\tfrac{c_0}{4}$-approximator of $u$, Lemma \ref{lemma:oscillation_lower_bound} gives us
\begin{multline*}
	c_0
	\le  | u(\hat{X}_{\widetilde{Q}_i^l}) - u(\hat{X}_{\widetilde{P}_i^l})  | 
	\le  | u(\hat{X}_{\widetilde{Q}_i^l}) - \Phi(\hat{X}_{\widetilde{Q}_i^l})  | + 
	 | \Phi(\hat{X}_{\widetilde{Q}_i^l}) - \Phi(\hat{X}_{\widetilde{P}_i^l})  | +
	 | \Phi(\hat{X}_{\widetilde{P}_i^l}) - u(\hat{X}_{\widetilde{P}_i^l})  | \\
	\le \frac{c_0}{2} +  | \Phi(\hat{X}_{\widetilde{Q}_i^l}) - \Phi(\hat{X}_{\widetilde{P}_i^l})  |.
\end{multline*}
By \eqref{whitney_region_corkscrews}, we know that the points $\hat{X}_{\widetilde{Q}_i^l}$ and $\hat{X}_{\widetilde{P}_i^l}$ are well inside $U_{Q_i^l,\eta^3}$ in the sense that $\dist(\hat{X}_{\widetilde{P}_i^l}, \partial U_{Q_i^l,\eta^3}) \approx \dist(\hat{X}_{\widetilde{Q}_i^l}, \partial U_{Q_i^l,\eta^3}) \approx \ell(Q_i^l)$. Since $\diam(U_{Q_i^l,\eta^3}) \approx_\eta \ell(Q_i^l)$ and $U_{Q_i^l,\eta^3}$ satisfies the Harnack chain condition by Lemma \ref{lemma:hm}, there exists a chain of $N = N(\eta)$ balls $B_1, B_2, \ldots, B_N$ such that
\begin{enumerate}
	\item[(i)] $\hat{X}_{\widetilde{Q}_i^l} \in B_1$, $\hat{X}_{\widetilde{P}_i^l} \in B_N$ and $B_{j} \cap B_{j+1} \neq \emptyset$ for every $j = 1,2,\ldots,N-1$,
	
	\item[(ii)] the radii of the balls $B_j$ are comparable to $\ell(Q_i^l)$, depending on $\eta$ and $c_0$,
	
	\item[(iii)] $|\Phi(\hat{X}_{\widetilde{Q}_i^l}) - \Phi(X)| < \tfrac{c_0}{8}$ for every $X \in B_1$ and $|\Phi(\hat{X}_{\widetilde{P}_i^l}) - \Phi(Y)| < \tfrac{c_0}{8}$ for every $Y \in B_N$,
	
	\item[(iv)] for each $j = 1,2,\ldots,N-1$, there exists a cylinder $\Cfk_j$ connecting $B_j$ to $B_{j+1}$ such that $B_j \cup B_{j+1} \subset \Cfk_j \subset U_{Q_i^l,\eta^3}$, the cylinders have bounded overlaps and $|\Cfk_j| \approx_\eta \ell(Q_i^l)^{n+1}$,
\end{enumerate}
where the bound (iii) follows from  properties of $\eps$-approximators (see property iii) in Theorem \ref{theorem:approximability}). Then
\begin{multline*}
	 | \Phi(\hat{X}_{\widetilde{Q}_i^l}) - \Phi(\hat{X}_{\widetilde{P}_i^l}) | \\
	\le \dashiint_{B_1}  | \Phi(\hat{X}_{\widetilde{Q}_i^l}) - \Phi(X) | \, dX +
	\sum_{j=1}^{N-1} \Big| \dashiint_{B_j} \Phi(X) \, dX - \dashiint_{B_{j+1}} \Phi(X) \, dX \Big|  + \dashiint_{B_N}  | \Phi(X) - \Phi(\hat{X}_{\widetilde{P}_i^l})  | \, dX \\
	\le \frac{c_0}{4} + \sum_{j=1}^{N-1} \Big| \dashiint_{B_j} \Phi(X) \, dX - \dashiint_{B_{j+1}} \Phi(X) \, dX \Big|
\end{multline*}
and furthermore, by Poincar\'e inequality and properties of the Whitney regions and cylinders $\Cfk_j$, we have
\begin{multline*}
	\Big| \dashiint_{B_j} \Phi(X) \, dX - \dashiint_{B_{j+1}} \Phi(X) \, dX \Big| 
	= \Big| \dashiint_{B_j} \Phi(X) \, dX - \dashiint_{\Cfk_j} \Phi(X) \, dX + \dashiint_{\Cfk_j} \Phi(X) \, dX - \dashiint_{B_{j+1}} \Phi(X) \, dX \Big| \\
	\lesssim_\eta \dashiint_{\Cfk_j} \Big| \Phi(X) - \dashiint_{\Cfk_j} \Phi(Y) \, dY \Big| \, dX 
	\lesssim_\eta \frac{\ell(Q_i^l)}{|\Cfk_j|} \iint_{\Cfk_j} |\nabla \Phi(X)| \, dX 
	\lesssim_\eta \iint_{\Cfk_j} |\nabla \Phi(X)| \delta(X)^{-n} \, dX
\end{multline*}
for every $j = 1,2,\ldots,N-1$. Combining the previous estimates and using the bounded overlaps of the cylinders then gives us
\begin{align*}
	\frac{c_0}{4}
	\le \sum_{j=1}^{N-1} \Big| \dashiint_{B_j} \Phi(X) \, dX - \dashiint_{B_{j+1}} \Phi(X) \, dX \Big|
	&\lesssim_\eta \iint_{U_{Q_i^l,\eta^3}} |\nabla \Phi(X)| \delta(X)^{-n} \, dX.
\end{align*}
Finally, we sum over $l$ and use Lemma \ref{lemma:existence_eps_covers}, the bounded overlaps of the regions $U_{Q_i^l,\eta^3}$ and the structure of $\Gamma_{Q_i^l}^\eta$ to get
\begin{equation*}
	\frac{c_0}{4} \frac{\log \alpha^{-1}}{\log \epsilon_0^{-1}}
	\approx \frac{c_0}{4} (k-1)
	\lesssim_\eta \sum_{l=1}^{k-1} \iint_{U_{Q_i^l,\eta^3}} |\nabla \Phi(X)| \delta(X)^{-n} \, dX
	\lesssim_\eta \iint_{\Gamma_{Q_i^l}^{\eta}(y)} |\nabla \Phi(X)| \delta(X)^{-n} \, dX,
\end{equation*}
which proves the claim.
\end{proof}

\section{Proof of Theorem \ref{main.thrm}}
\label{section:proof_of_varopoulos}

With the help of the tools from the previous sections, we can now prove Theorem \ref{main.thrm}. The proof is an adaptation of the corresponding proof from \cite{HT2} (which itself uses some core ideas of Varopoulos \cite{Var1,Var2} and Garnett \cite{Gar-BAF}).

\begin{proof}[Proof of Theorem {\ref{main.thrm}}]
  Let $f \in \BMO_{\smallc}(\pom)$ and $\dd$ be a dyadic system on $\pom$. By Lemma 10.1 and Remark 10.3 in \cite{HT2}, we know that
  \begin{align}
    \label{decomp:data} f = f_0 + g,
  \end{align}
  where $f_0 \in L^\infty(\pom)$ with $\|f_0\|_{L^\infty(\pom)} \lesssim \|f\|_{\BMO}$ and $g = \sum_j \alpha_j\mathbbm{1}_{Q_j}$, for a collection of dyadic cubes $\Ac \coloneqq \{Q_j\}_j \subset \dd$ satisfying a Carleson packing condition with $\C_{\Ac} \lesssim 1$ and coefficients $\alpha_j$ satisfying $\sup_j |\alpha_j| \lesssim \|f\|_{\BMO(\pom)}$. By \cite[Proposition 1.3]{HT2}, there exists an extension $G$ of $g$ in $\Omega$ that satisfies the corresponding versions of the properties (1) -- (3) in Theorem \ref{main.thrm}. Thus, it is enough to construct the extension for the function $f_0$ in the decomposition \eqref{decomp:data}.
  
  Construction of the extension for $f_0$ follows the proof of \cite[Theorem 1.2]{HT2}. We repeatedly use Lemma \ref{lemma:solution_bounded_data} to solve boundary value problems for $L$ with updated data and Theorem \ref{theorem:approximability} to take smooth $\tfrac{1}{2}$-approximators (that is, $\eps$-approximators for $\eps = \tfrac{1}{2}$) for the corresponding solutions and to take their a.e.\ non-tangential boundary traces. The idea is that a suitable approximator is a Varopoulos-type extension plus an error term and we can keep on halving the size of the error term through iteration.
  
  We start by taking the solution $u_0$ to the boundary value problem with data $f_0$, satisfying $u_0(X) = \int_{\pom} f_0 \, d\omega^X$. We then take the $\tfrac{1}{2}$-approximator $\Phi_0$ to $u_0$. This approximator has an a.e.\ non-tangential boundary trace $\varphi_0$. By Theorem \ref{theorem:approximability}, these functions satisfy $\|u_0 - \Phi_0\|_{L^\infty(\Omega)} \le \tfrac{1}{2}\|u_0\|_{L^\infty(\Omega)} \le \tfrac{1}{2} \|f_0\|_{L^\infty(\pom)}$, and
  \begin{align*}
    \sup_{x \in \pom, r > 0} \frac{1}{r^n} \iint_{B(x,r) \cap \Omega} |\nabla \Phi_0(Y)|\, dY
    \le \widetilde{C} \|u_0\|_{L^{\infty}(\Omega)}
    \le \widetilde{C} \|f_0\|_{L^\infty(\pom)}.
  \end{align*}
  We set $f_1 \coloneqq f_0 - \varphi_0$ which is the a.e.\ non-tangential boundary trace of $u_0 - \Phi_0$. We take the solution $u_1$ with the data $f_1$, satisfying $u_1(X) = \int_{\pom} f_1 \, d\omega^X$. This solution satisfies
  \begin{align*}
    \|u_1\|_{L^\infty(\Omega)} \le \|f_1\|_{L^\infty(\pom)} \le \|u_0 - \Phi_0\|_{L^\infty(\Omega)} \le \tfrac{1}{2}\|u_0\|_{L^\infty(\Omega)} \le \tfrac{1}{2} \|f_0\|_{L^\infty(\pom)},
  \end{align*}
  and thus, we can take a $\tfrac{1}{2}$-approximator $\Phi_1$ of $u_1$, with boundary trace $\varphi_1$. We get
  \begin{align*}
    \|u_1 - \Phi_1\|_{L^\infty(\Omega)} \le \tfrac{1}{2} \|u_1\|_{L^\infty(\Omega)} \le \tfrac{1}{4} \|u_0\|_{L^\infty(\Omega)}
  \end{align*}
  and
  \begin{align*}
    \sup_{x \in \pom, r > 0} \frac{1}{r^n} \iint_{B(x,r) \cap \Omega} |\nabla \Phi_1(Y)|\, dY
    \le \widetilde{C} \|u_1\|_{L^{\infty}(\Omega)}
    \le \frac{\widetilde{C}}{2} \|f_0\|_{L^\infty(\pom)}.
  \end{align*}
  We set $f_2 \coloneqq f_1 - \varphi_1 = f_0 - \varphi_0 - \varphi_1$, and continue in the previous way. This gives us a sequences of solutions $u_k$ and their $\tfrac{1}{2}$-approximators $\Phi_k$. The functions $u_k$ and $\Phi_k$ have a.e.\ non-tangential boundary traces $f_k$ and $\varphi_k$, respectively, and we have
  \begin{enumerate}
  \item[(i)] $f_{k+1} = f_0 - \sum_{i=0}^k \varphi_i$,
  
  \item[(ii)] $\|f_{k+1}\|_{L^\infty(\pom)} \le \|u_k - \Phi_k\|_{L^\infty(\Omega)} \le 2^{-k-1}\|u_0\|_{L^\infty(\Omega)} \le 2^{-k-1} \|f_0\|_{L^\infty(\pom)}$,
  
  \item[(iii)] $\|u_k\|_{L^\infty(\Omega)} \le \|f_{k}\|_{L^\infty(\pom)} \le 2^{-k} \|u_0\|_{L^\infty(\Omega)}$,
  
  \item[(iv)] $\sup_{x \in \pom, r > 0}  \frac{1}{r^n} \iint_{B(x,r) \cap \Omega} |\nabla \Phi_k(Y)| \, dY \le
  \widetilde{C} \|u_k\|_{L^\infty(\Omega)} \le  \widetilde{C}2^{-k} \|u_0\|_{L^\infty(\Omega)}$, and
  
  \item[(v)] $ \| \Phi_k\|_{L^\infty(\Omega)} \lesssim  2^{-k} \|u_0\|_{L^\infty(\Omega)}$.
\end{enumerate}
Property (v) follows from the properties (ii) and (iii) combined with the fact that $\Phi_k$ is a $\tfrac{1}{2}$-approximator of $u_k$. By this property, we can define a uniformly convergent series
\begin{align*}
  \Phi(X) \coloneqq \sum_{k=0}^\infty \Phi_k(X)
\end{align*}
for $X \in \Omega$. Let $F_0$ be a version of $\Phi$ that has been smoothened using similar convolution techniques as in the proof of Lemma \ref{lemma:smooth_approximators} (see \cite[Section 3]{HT2} for details). By the arguments in the proof of \cite[Theorem 1.1]{HT2}, $F_0$ is an extension of $f_0$ that satisfies the properties (1)--(3) in Theorem \ref{main.thrm}. Thus, by the decomposition \eqref{decomp:data}, we may define the extensions $F$ in Theorem \ref{main.thrm} by setting $F \coloneqq F_0 + G$. This completes the proof.
\end{proof}

\section{An example in $\re^3$}
\label{section:example}

In this section we construct a three-dimensional version of the example provided in Corollary \ref{corollary:main}. That is, we
show there is
a domain $\Omega$ in $\re^3$ whose boundary is not rectifiable and such that every function $f \in \BMO(\pom)$ has a Varopoulos extension in $\Omega$. 

To define $\Omega$, denote by $E$ the $4$-corner Cantor set in $\re^2$. That is $E=\bigcap_{k=0}^\infty E_k$, where $E_k$ equals the union of $4^k$ closed squares $Q_i^k$ of side length $4^{-k}$ located in the corners of the squares $Q_j^{k-1}$ of the previous generation (see \cite[Section 3]{DM-4cc} for the precise definition).
We assume that the center of $E_0=Q_1^0$ coincides with the origin in $\re^2$.
We consider the half-plane $\Pi=\{(x,y)\in\re^2:y>-2\}$ and we set $V=\Pi\setminus E$. We also write $\bb L=\partial\Pi=\{(x,y)\in\re^2:y=-2\}$.
Then we define $\Omega= V\times \re$. It is straightforward to check that both $V$ and $\Omega$ are uniform domains.

Notice that $\partial V= E\cup\bb L$ is $1$-Ahlfors regular, while $\partial \Omega= (E\cup\bb L)\times \re$ is $2$-Ahlfors regular. In fact, the purpose of introducing the half plane $\Pi$ and the line $\bb L$ in this previous construction is to ensure the $2$-Ahlfors 
regularity of $\partial \Omega$. It is also clear that $E\times \re$ is purely  $2$-unrectifiable (since this set has no approximate tangent planes at any point). 

Let $A$ be the $2\times 2$ matrix in the David--Mayboroda example in Theorem \ref{DMthrm.thrm}, let $L=-\dv A\nabla$ be the associated elliptic operator in $\bb R^2$, and let
$$\hat A =\left( \begin{array}{ll} A & 0\\ 0 &1\end{array}\right).$$ 
Set $\hat L =  -\div \hat A\,\nabla$ in $\re^3$.
Below we will show that   $\omega_{\hat L,\Omega}\in A_\infty(\sigma)$, where $\sigma$ is the surface measure on $\partial\Omega$. Consequently, by Theorem \ref{main.thrm}, every function $f \in \BMO(\pom)$ has a Varopoulos extension in $\Omega$. 

First, we prove the following.

\begin{lemma}\label{lem**1} We have that $\omega_{L,V}\in A_{\infty}(\H^1|_{\partial V})$. Further, there is a constant $C>0$ such that
for any surface ball $\Delta\subset \partial V$ and any corkscrew point $p\in V$ for $\Delta$, 
\begin{equation}\label{eq**0.5}
\frac{d\omega_{L,V}^p}{d\H^1|_{\partial V}}(x)\leq \frac C{\H^1(\Delta)},\qquad\text{for each }x\in\Delta.
\end{equation}  
\end{lemma}

\begin{proof}
It is clear that the estimate \eqref{eq**0.5} implies the local $A_\infty$ condition of $\omega_L\equiv\omega_{L,V}$.
First we will show that, for any $p\in \Pi$, 
\begin{equation}\label{eq**1}
\frac{d\omega_{L,\Pi}^p}{d\H^1|_{\bb L}}(x)\leq C\,\frac{d\omega_{-\Delta,\Pi}^p}{d\H^1|_{\bb L}}(x),
\end{equation}
where $\omega_{L,\Pi}$ stands for the $L$-elliptic measure for the domain $\Pi$, and $\omega_{-\Delta,\Pi}$ is the harmonic measure (i.e.\ for the Laplacian) for the domain $\Pi$. To this end, we consider the auxiliary domain $U= \Pi \setminus \overline{B_1}$, where we denoted $B_r=B(0,r)$. 
Observe that $A$ is the identity matrix on $U$, and thus $\omega_{L,U}^X\equiv \omega_{-\Delta,U}^X$ for all $X\in U$.
Let $F\subset\bb L$ be an arbitrary
closed subset, and let $q\in\partial B_{3/2}$ be such that
$$\omega_{L,\Pi}^q(F) = \max_{X\in\partial B_{3/2}}\omega_{L,\Pi}^X(F).$$
Since $\omega_{L,\Pi}^x(F)$ is a harmonic function of $x$ in $U$, we have
\begin{multline}\nonumber
\omega_{L,\Pi}^q(F)  = \int_{\partial U} \omega^X_{L,\Pi}(F)\,d\omega^q_{-\Delta,U}(X) 
= 
\int_{\bb L} \omega_{L,\Pi}^X(F)\,d\omega^q_{-\Delta,U}(X) + \int_{\partial B_1} \omega_{L,\Pi}^X(F)\,d\omega^q_{-\Delta,U}(X)\\
\leq \omega_{-\Delta,U}^q(F) + \sup_{X\in\partial B_1}\omega_{L,\Pi}^X(F)\,\omega_{-\Delta,U}^q(\partial B_1)
.
\end{multline}
By the maximum principle and the definition of $q$, we have
$$\sup_{X\in\partial B_1}\omega_{L,\Pi}^X(F) \leq \sup_{X\in\partial B_{3/2}}\omega_{L,\Pi}^X(F) = \omega_{L,\Pi}^q(F).$$
Also, it is immediate that $c_B \coloneqq \omega_{-\Delta,U}^q(\partial B_1)<1$. Hence,
$$\omega_{L,\Pi}^q(F) \leq \omega_{-\Delta,U}^q(F) + c_B\,\omega_{L,\Pi}^q(F),$$
or equivalently, $\omega_{L,\Pi}^q(F) \leq (1-c_B)^{-1}\omega_{-\Delta,U}^q(F)$. By a Harnack chain argument we deduce that
$\omega_{L,\Pi}^p(F) \lesssim \omega_{-\Delta,U}^p(F) = \omega_{L,U}^p(F)
$ for all $p\in \partial B_{3/2}$, and then by the maximum principle, it follows that the same estimate is valid for all $p\in
\Pi\setminus \overline B_{3/2}$.
Using again the maximum principle, we get, for all $p\in \Pi\setminus\overline B_{3/2}$,
$$\omega_{L,\Pi}^p(F) \lesssim \omega_{-\Delta,U}^p(F) \leq \omega_{-\Delta,\Pi}^p(F).$$
Since $\omega_{L,\Pi}^p(F)$ and $\omega_{-\Delta,U}^p(F)$ are, respectively, elliptic and harmonic in $B_{3/2}$,  by a Harnack chain argument it follows that the estimate above also holds for all $p\in B_{3/2}$, possibly with a different implicit constant.
This is equivalent to \eqref{eq**1}.

To prove \eqref{eq**0.5}, let $B$ be a ball centered in $\partial V$ such that $\Delta=B \cap\pom$ and consider an arbitrary closed
set $F\subset \Delta$. Let $p\in B\cap V$ be a corkscrew point for $\Delta$. By the maximum principle and by \eqref{eq**1},
$$\omega_{L,V}^p (F\cap\bb L) \leq \omega_{L,\Pi}^p (F\cap\bb L) \lesssim  \omega_{-\Delta,\Pi}^p (F\cap\bb L)\lesssim\frac{\H^1(F\cap\bb L)}{\H^1(\Delta)},$$
where in the last estimate we used that $\frac{d\omega_{-\Delta,\Pi}^p}{d\H^1|_{\bb L}}(x)\lesssim \sigma(\Delta)^{-1}$ for all $x\in \Delta$. Indeed, if $F\cap\bb L=\varnothing$, then there is nothing to show; if $B$ is centered on $\bb L$, this follows from Ahlfors regularity of $\partial V$ and classical properties of the harmonic measure, and if $B$ is centered on $E$ and there exists $z\in F\cap\bb L$, then if $B'=B(z,\operatorname{rad}(B))$ and $p'$ is a corkscrew point for $B'$ in $V$, then by Harnack chains we have that $\omega^p_{-\Delta,\Pi}(F\cap\bb L)\approx\omega^{p'}_{-\Delta,\Pi}(F\cap\bb L)$, and the claim follows as in the previous case.

To estimate $\omega_{L,V}^p (F\cap E)$, we assume that $B\cap E\neq\varnothing$ and we
distinguish two cases. Suppose first that $r(B)$, the radius of $B$, satisfies $r(B)\leq 1$. Denote $X_E=(0,2)$ and let $B'$ be a ball
centered in $E$ containing $B$ with radius at most $2r(B)$.
By the maximum principle, the change of poles formula, 
and the fact that $\frac{d\omega_{L,E^c}^{X_E}}{d\H^1|_{E}}(x)\approx 1$ (by Theorem \ref{DMthrm.thrm}),
we obtain
$$\omega_{L,V}^p(F\cap E) \leq \omega_{L,E^c}^p(F\cap E) \approx \frac{\omega_{L,E^c}^{X_E}(F\cap E)}{\omega_{L,E^c}^{X_E}(B')}
\approx \frac{\H^1(F\cap E)}{\H^1(B'\cap E)} \approx \frac{\H^1(F\cap E)}{\H^1(\Delta)}.
$$

In the case $r(B)>1$, recall that $B_1=B(0,1)$, and then using the change of poles formula and the maximum principle, we write
$$\omega_{L,V}^p(F\cap E) \approx \omega_{L,V}^{X_E}(F\cap E)\,\omega_{L,V}^{p}(B_1) \leq \omega_{L,E^c}^{X_E}(F\cap E)\,\omega_{L,V}^{p}(B_1) \approx \H^1(F\cap E)\,\omega_{L,V}^{p}(B_1).$$
To estimate $\omega_{L,V}^{p}(B_1)$,
consider a ball $B_1'$ of radius $1$ centered in $\bb L$, disjoint from $E$, and contained in $4B_1$. We claim that
\begin{equation}\label{eq.b1}
\omega_{L,V}^{p}(B_1) \approx \omega_{L,V}^{p}(B_1').
\end{equation}
Indeed, if $p\notin 8B_1$, then (\ref{eq.b1}) follows directly from Lemma \ref{lm.doubling}. On the other hand, if $p\in 8B_1$, then denote $p'=(0,16)$ and note that $\delta(p)\gtrsim1$, $\delta(p')\geq8$, and $|p-p'|\leq24$. Then, by the Harnack inequality, Harnack chains, and Lemma \ref{lm.doubling}, the estimate (\ref{eq.b1}) follows. Next, by the maximum principle, \eqref{eq**1}, and Ahlfors regularity, we get
$$\omega_{L,V}^{p}(B_1') \leq \omega_{L,\Pi}^{p}(B_1') \lesssim \omega_{-\Delta,\Pi}^{p}(B_1')
\lesssim \frac{\H^1(B_1'\cap\bb L)}{\H^1(\Delta)}\approx \frac1{\H^1(\Delta)}.
$$
Therefore, again we derive
$$\omega_{L,V}^p(F\cap E)\lesssim \frac{\H^1(F\cap E)}{\H^1(\Delta)}.$$

Altogether, we deduce that
$$\omega_{L,V}^p (F) = \omega_{L,V}^p (F\cap\bb L) + \omega_{L,V}^p (F\cap E)\lesssim \frac{\H^1(F)}{\H^1(\Delta)}$$
for any closed set $F\subset\Delta$, which is equivalent to the statement in the lemma, by the Lebesgue-Radon-Nykodim Theorem.
\end{proof}

Now we are ready to prove the $A_\infty$ property of the elliptic measure $\omega_{\hat L}$ for the three-dimensional domain $\Omega$:

\begin{proposition}\label{proposition:3-dimensional_example}
The elliptic measure $\omega_{\hat L}$ for the domain $\Omega\subset \re^3$ defined above satisfies the $A_\infty$ condition with respect to $\H^2|_{\pom}$.
Further, there is a constant $C>0$ such that
for any surface ball $\Delta\subset \pom$ and any corkscrew point $p\in \Omega$ for $\Delta$, 
\begin{equation}\label{eq**3}
\frac{d\omega_{\hat L}^p}{d\H^2|_{\pom}}(x)\leq \frac C{\H^2(\Delta)},\qquad\text{for each }x\in\Delta.
\end{equation}
\end{proposition}

\begin{proof}
Let $B$ be a ball with radius $r(B)$ centered in $\pom$ such that $\Delta= B\cap \pom$. Clearly, to prove the  $A_\infty$ condition for $\omega_{\hat L}$, it suffices to show \eqref{eq**3}. In turn, since $\hat L$ is symmetric, by a direct application of Lemma \ref{CFMSest.lem}     and the Lebesgue-Radon-Nykodim Theorem, it is enough to prove that 
\begin{equation}\label{eqfac10}
G_{\hat L}(X,p) \lesssim \frac{\dist(X,\pom)}{r(B)^2}\quad \mbox{ for all $X\in B\cap \Omega \setminus B(p,\frac12 \dist(p,\pom))$,}
\end{equation}
where $G_{\hat L}$ is the $\hat L$-Green function for $\Omega$.

Denote by $P$ the orthogonal projection of $\re^3$ onto $\re^2\equiv \re^2\times\{0\}$.
Let $p_0 = P(p)$ and consider the function $u:\overline\Omega\setminus P^{-1}(\{p_0\})\to \re$ defined by
$$u(X) = G_L(P(X),P(p)) = G_L(P(X),p_0),$$
where $G_L$ is the  Green's function for $L$ in the domain $V=P(\Omega)$.
It is immediate to check that $u$ is $\hat L$-elliptic in $\Omega\setminus P^{-1}(\{p_0\})$, and clearly it can be extended continuously
by zero to the whole $\pom$. Thus, by the boundary Harnack principle, choosing $p'\in \partial B(p,\frac12 \dist(p,\pom))$ such that
$P(p') \in \partial B(p_0,\frac12 \dist(p,\pom))\cap \re^2$, we have
$$\frac{G_{\hat L}(X,p)}{G_{\hat L}(p',p)} \approx  \frac{u(X)}{u(p')}
= \frac{G_L(P(X),p_0)}{G_L(P(p'),p_0)}
 \quad \mbox{ for all $X\in B\cap \Omega \setminus B(p,\frac12 \dist(p,\pom))$.}
$$
Thus, for such points $X$ and by \eqref{fungf1.eq} and \eqref{fungf2.eq} applied both to $G_{\hat L}$ and $G_{L}$,
$$\frac{G_{\hat L}(X,p)}{|p'-p|^{-1}} \approx  
\frac{G_L(P(X),p_0)}{1}.$$
Thus, by Lemma \ref{CFMSest.lem} applied to $G_{L,V}$, Harnack chains, and the Harnack inequality, we see that
\begin{equation}\label{eq.g1}
G_{\hat L}(X,p) \approx  \frac{G_{L,V}(P(X),p_0)}{r(B)} \approx \frac{\omega_{L,V}^{p_0}(\Delta_{V,X})}{r(B)},
\end{equation}
where $\Delta_{V,X}= B(P(X),2\dist(P(X),\partial V))$.
From \eqref{eq**0.5} we infer that 
\begin{equation}\label{eq.g2}
\omega_{L,V}^{p_0}(\Delta_{V,X})\lesssim  \frac{\H^1(\Delta_{V,X})}{\H^1(P(B)\cap \partial V)} \approx
\frac{\dist(X,\pom)}{r(B)}.
\end{equation}
From (\ref{eq.g1}) and (\ref{eq.g2}), we deduce \eqref{eqfac10}, which concludes the proof.
\end{proof}

\bibliography{BPTTrefs}
\bibliographystyle{alpha-sort-max}

\end{document}